\documentclass[12pt]{report}
\usepackage[utf8]{inputenc}
\usepackage{amsmath,amssymb,amsthm,amsfonts,graphics,stmaryrd,hyperref,bbm, biblatex}
\usepackage[margin = 1in, includehead, footskip=0.25in]{geometry}

\usepackage{setspace}
\usepackage[T1]{fontenc}
\usepackage[shortlabels]{enumitem}
\theoremstyle{plain}
\newtheorem{theorem}{Theorem}[section]
\newtheorem{lemma}[theorem]{Lemma}
\newtheorem{corollary}[theorem]{Corollary}
\newtheorem{definition}[theorem]{Definition}
\newtheorem{prop}[theorem]{Proposition}

\newtheorem{question}{Question}
\newcommand{\sbp}{\hspace{2pt}|\hspace{2pt}}
\newcommand{\lk}{<\hspace{-2pt} \kappa}
\mathchardef\mhyphen="2D
\def\frontmatter{%
    \pagenumbering{roman}
    \setcounter{page}{2}
}%

\def\mainmatter{%
    \pagenumbering{arabic}
    \setcounter{page}{1}
    \setcounter{section}{0}
}%

\begin{document}

\begin{titlepage}
    \begin{center}
        \vspace*{144pt}
        \textsc{$\Sigma_n$-correct Forcing Axioms} \\[0.5in]
by \\[0.5in]
\textsc{Ben Goodman} 
\vfill
A dissertation submitted to the Graduate Faculty in Mathematics in partial fulfillment of the requirements for the degree of Doctor of Philosophy, The City University of New York \\[0.25in]
2024
    \end{center}
\end{titlepage}

\frontmatter

\phantom{}\vspace{\fill}
\begin{center}
\copyright~2024\\
\textsc{Ben Goodman}\\
All Rights Reserved\\
\end{center}

\newpage

\begin{center}
APPROVAL

$\Sigma_n$-correct Forcing Axioms

by

Ben Goodman

\vspace{20pt}

This manuscript has been read and accepted by the Graduate Faculty in Mathematics in satisfaction of the dissertation requirement for the degree of Doctor of Philosophy.

Approved: May 2024

\vspace{0.5in}

Gunter Fuchs, Chair of Examining Committee

Christian Wolf, Executive Officer

\vspace{0.5in}

Supervisory Committee:

Gunter Fuchs, Advisor

Arthur Apter

Russell Miller
\end{center}

\vspace{\fill}
\begin{center}
\textsc{The City University of New York}
\end{center}

\newpage

\begin{center}
Abstract \\
\textsc{$\Sigma_n$-correct Forcing Axioms} \\
by \\
\textsc{Ben Goodman} \\[0.25in]
\end{center}

\vspace{0.25in}

\noindent Advisor: Gunter Fuchs

\vspace{0.25in}

\noindent I introduce a new family of axioms extending ZFC set theory, the $\Sigma_n$-correct forcing axioms. These assert roughly that whenever a forcing name $\dot{a}$ can be forced by a poset in some forcing class $\Gamma$ to have some $\Sigma_n$ property $\phi$ which is provably preserved by all further forcing in $\Gamma$, then $\dot{a}$ reflects to some small name such that there is already in $V$ a filter which interprets that small name so that $\phi$ holds. $\Sigma_1$-correct forcing axioms turn out to be equivalent to classical forcing axioms, while $\Sigma_2$-correct forcing axioms for $\Sigma_2$-definable forcing classes are consistent relative to a supercompact cardinal (and in fact hold in the standard model of a classical forcing axiom constructed as an extension of a model with a supercompact), $\Sigma_3$-correct forcing axioms are consistent relative to an extendible cardinal, and more generally $\Sigma_n$-correct forcing axioms are consistent relative to a hierarchy of large cardinals generalizing supercompactness and extendibility whose supremum is the first-order version of Vopenka's Principle.

By analogy to classical forcing axioms, there is also a hierarchy of $\Sigma_n$-correct bounded forcing axioms which are consistent relative to appropriate large cardinals. At the two lowest levels of this hierarchy, outright equiconsistency results are easy to obtain. Beyond these consistency results, I also study when $\Sigma_n$-correct forcing axioms are preserved by forcing, how they relate to previously studied axioms and to each other, and some of their mathematical implications.

\chapter*{Acknowledgements}

Thank you to my advisor, Gunter Fuchs, for his guidance, kindness, and support over the past several years, and for suggesting such a fascinating research topic. Thank you also to Arthur Apter and Russell Miller for agreeing to serve on both my dissertation and oral exam committees, and to everyone in the CUNY logic community who answered a question I had or provided enlightening conversation during my time here.

I would also like to thank Joel David Hamkins, who first kindled my interest in set theory over a decade ago with a talk he gave at my undergraduate institution and the dinner we attended afterwards, has continued to entertain, challenge, and inspire me during his visits to CUNY, and proved a disproportionate number of the prior results this dissertation draws upon; Corey Bacal Switzer, who welcomed me into the CUNY set theory community when I first arrived, talked to me in depth about my research after he graduated, and asked some fruitful questions; and all the friends and family who supported me during this time and endured my attempts to explain what a forcing axiom is in plain language.

\tableofcontents

\newpage

\mainmatter

\chapter*{Introduction}

This dissertation studies the $\Sigma_n$-correct forcing axioms, potential axioms of set theory which can variously be viewed as:
\begin{itemize}
    \item generalizations of the "plus versions" of forcing axioms to properties beyond stationarity
    \item strengthenings of the maximality principles of Stavi and V\"a\"an\"anen \cite{svMP} and Hamkins \cite{hamkinsMP} by intertwining them with reflection principles to accommodate parameters of arbitrary size
    \item unifications of maximality principles with classical forcing axioms
    \item the unbounded versions of the generalized bounded forcing axioms considered by David Asper\'o \cite{asperomax}
\end{itemize}

Chapter 1 covers the background material necessary for what follows; almost none of the results in it are original to me. Section \ref{section:notation} establishes notational definitions and reviews the definitions of the forcing classes which will be most commonly used in examples. Sections \ref{section:fa} and \ref{section:MP} cover classical forcing axioms and maximality principles respectively. Section \ref{section:boolean} briefly surveys the theory of Boolean-valued models and their quotients by ultrafilters. Section \ref{section:interp} addresses the technical issue that there are multiple possible notions of the interpretation of a name by a filter, which can diverge when the filter is not generic or the name is not forced to be a subset of the ground model. Finally, Section \ref{section:geology} lays out the most notable and relevant results in set-theoretic geology, the study of the forcing grounds of the universe (i.e. those inner models of which the universe is a generic extension).

Chapter 2 introduces the large cardinals which will form the hypotheses of later consistency results. Joan Bagaria has done a great deal of work on related topics, though with somewhat different motivations. This chapter is a mix of previously known results (many of them due to or first explicitly written down by Bagaria), straightforward generalizations of old results, and a few genuinely new thoerems. Section \ref{section:Cn} explores the basic properties of the $\Sigma_n$-correct cardinals, i.e. those cardinals $\theta$ such that $V_\theta$ agrees with $V$ on the truth of $\Sigma_n$ formulas with parameters in $V_\theta$. Section \ref{section:scCn} introduces the supercompact cardinals for $C^{(n)}$, which are the critical points of supercompactness embeddings which additionally preserve arbitrarily large initial segments of the $\Sigma_n$-correct cardinals. These are importantly distinct from Bagaria's $C^{(n)}$-supercompact cardinals, and in fact more closely related to his notion of $C^{(n)}$-extendibility. I prove some useful characterizations of supercompactness for $C^{(n)}$, and show that as $n$ increases these cardinals form a hierarchy starting from ordinary supercompactness and then extendibility, reaching upward toward Vopenka's principle. Section \ref{section:SnHlref} covers the $\Sigma_n$-correctly $H_\lambda$-reflecting cardinals, a generalization of Miyamoto's $H_\lambda$-reflecting cardinals, which for any fixed $n$ form a hierarchy reaching from the regular $\Sigma_n$-correct cardinals to the supercompact cardinals for $C^{(n-1)}$.

Chapter 3 finally begins to discuss the main topic of this dissertation. Section \ref{section:sandcon} defines $\Sigma_n$-correct forcing axioms, motivating their statement with a characterization of classical forcing axioms due to Ronald Jensen. It goes on to prove that, for suitably iterable $\Sigma_n$-definable forcing classes, the corresponding $\Sigma_n$-correct forcing axiom holds in a forcing extension by a Baumgartner-style iteration of any model with a cardinal supercompact for $C^{(n-1)}$. It also contains several other basic facts about the axioms. In Section \ref{section:equivforms}, I explore several alternative statements of $\Sigma_n$-correct forcing axioms, motivated by various equivalents or enhancements of classical forcing axioms found in the literature; these alternative formulations are more convenient than the Jensen-style version of the axioms for certain proofs. Section \ref{section:hierarchy} addresses the natural question of whether increasing $n$ necessarily results in a strictly stronger $\Sigma_n$-correct forcing axiom for sufficiently nice forcing classes; this turns out to be surprisingly hard to answer in general. Section \ref{section:collo1} briefly considers forcing axioms for classes which can collapse $\omega_1$; the classical forcing axioms for such classes are of course trivially true or trivially false in most cases, but the $\Sigma_n$-correct versions are potentially interesting. Section \ref{section:intvext} resolves a technical ambiguity in the statement of $\Sigma_n$-correct forcing axioms, specifically whether they are single sentences quantifying over all $\Sigma_n$ formulas encoded in the model or schemes with a separate instance for each $\Sigma_n$ formula in the metatheory. As it turns out, both interpretations are viable, but the consistency proof in the former case requires more care and slightly stronger hypotheses.

Chapter 4 explores the bounded version of the axioms. Asper\'o introduced equivalent principles, but I take a somewhat different approach, defining a notion of a filter being weakly generic over a model in order to facilitate a Jensen-style formulation of $\Sigma_n$-correct bounded forcing axioms. Section \ref{section:cbfacon} proves that these axioms are consistent relative to $\Sigma_n$-correct $H_\lambda$-reflecting cardinals for suitable $\lambda\geq \kappa$. Conversely, Section \ref{section:equicon} obtains the existence of large cardinals in $L$ from $\Sigma_n$-correct bounded forcing axioms, establishing equiconsistency results for the two lowest levels of the boundedness hierarchy.

Chapter \ref{section:preservation} studies when $\Sigma_n$-correct forcing axioms are preserved by forcing, drawing on Sean Cox's unified framework for preservation results in terms of extensions of generic elementary embeddings \cite{coxfa}. Chapter 6 explores the relationships of $\Sigma_n$-correct forcing axioms to classical forcing axioms, maximality principles, and each other, with Section \ref{section:equivalences} addressing when different principles turn out to be equivalent and Section \ref{section:separations} containing non-implication and inconsistency proofs. Both sections conclude by listing some related problems which remain open.

Chapter \ref{section:rrp} introduces residual reflection principles, weakenings of both $\Sigma_n$-correct (bounded) forcing axioms and $\Sigma_n$-correct $H_\lambda$-reflection which state that if a formula whose truth is preserved by some forcing class holds of some (not too large) parameter, then that parameter reflects to a small object of which the formula also holds. The main result of the chapter is that for forcing classes with a $\Sigma_n$ definition which is itself preserved by all posets in the class, $\Sigma_n$-correct forcing axioms (bounded or unbounded) may be factored as the conjunction of the $\Sigma_n$-maximality principle for that class and an appropriate residual reflection principle. Unfortunately, that hypothesis only holds for a small number of forcing classes.

Finally, Chapter 8 explores some of the consequences which follow from $\Sigma_n$-correct forcing axioms. Section \ref{section:continuum} focuses on the cardinality of the continuum. Whereas classical forcing axioms either do not settle the size of the continuum or imply that it is $\aleph_2$, $\Sigma_2$-correct forcing axioms for certain classes can also imply $CH$, or that the continuum is indescribably large. Section \ref{section:combin} looks at combinatorial applications, specifically the tree property, $\diamondsuit$, Kurepa's hypothesis, and stationary reflection. Section \ref{section:large} touches on some implications with a large cardinal or generic large cardinal character.

The appendices contain information on basic facts which will be familiar to most but not all readers, and so is included for completeness. Appendix A analyzes the formula complexity of a number of basic set-theoretic properties, streamlining certain proofs in the main body which rely on such analysis. Appendix B compiles a number of widely but not universally known lemmas which will be used repeatedly.

\chapter{Preliminaries}

\section{Notation and Common Forcing Classes}
\label{section:notation}

Forcing posets are ordered so that $p\leq q$ means $p$ extends $q$.

For any set $X$ such that $\in\upharpoonright X$ is extensional, $\pi_X$ is the Mostowski collapsing isomorphism of $(X, \hspace{3pt}\in\upharpoonright X)$.

\begin{definition}
    If $\kappa\leq \lambda$ are regular uncountable cardinals, $S\subseteq [H_\lambda]^{<\kappa}$ is weakly stationary if for every function $f:[H_\lambda]^{<\omega}\rightarrow H_\lambda$ there is some $Z\in S$ such that $f"[Z]^{<\omega}\subseteq Z$. $S$ is stationary if for any such $f$ we can find a $Z$ closed under it with the additional property that $Z\cap H_\kappa$ is transitive.
\end{definition}

It can be shown that this is equivalent to Jech's original definition of stationarity (that $S$ meets all closed unbounded subsets of $[H_\lambda]^{<\kappa}$).

$\square$ and $\diamondsuit$ are used to represent both modal operators and Jensen's combinatorial principles; which is intended should be clear from context, and fortunately neither is used very often.

$ZFC^-$ denotes the theory consisting of the axioms of Extension, Foundation, Pairing, Union, Infinity, and Well-Ordering (rather than some other form of Choice), together with the schemes of Separation and Collection (rather than Replacement), but not the Power Set axiom. See Gitman, Hamkins, and Johnstone \cite{GHJnopowerset} for more details on why the specific axiomatization matters.

For further definitions of standard set-theoretic concepts, consult Jech's textbook \cite{jech} or a similar reference.

I use first-person plural pronouns in situations which could by some stretch of the imagination include the reader and first-person singular pronouns in situations which could not.

\begin{definition}
    A poset $\mathbb{P}$ satisfies the $\kappa$-cc iff every antichain (i.e. set of pairwise incompatible conditions) of $\mathbb{P}$ has cardinality less than $\kappa$. We write ccc in place of $\omega_1$-cc.
\end{definition}

\begin{definition}
    A poset $\mathbb{P}$ is $<\hspace{-2pt}\kappa$-closed iff for every $\theta<\kappa$ and every order-reversing function $f:\theta\rightarrow \mathbb{P}$, there is a $p\in\mathbb{P}$ such that $p\leq f(\alpha)$ for all $\alpha<\theta$. We sometimes write countably closed in place of $<\hspace{-2pt}\omega_1$-closed.
\end{definition}

\begin{definition}
    A poset $\mathbb{P}$ is proper iff for all sufficiently large cardinals $\theta$ there is a club of countable elementary substructures of $H_\theta$ such that each $M$ in the club contains $\mathbb{P}$ and for all $p\in\mathbb{P}\cap M$, there is a $q\leq p$ such that for all names for an ordinal $\dot{\alpha}\in M$, $q\Vdash \exists \beta\in\check{M}\hspace{2pt} \dot{\alpha}=\beta$. Equivalently, for all infinite cardinals $\lambda$ and all stationary $S\subseteq [\lambda]^\omega$, $\Vdash_\mathbb{P} \check{S}$ is stationary.
\end{definition}

\begin{definition}
    A poset $\mathbb{P}$ is semiproper iff for all sufficiently large cardinals $\theta$ there is a club of countable elementary substructures of $H_\theta$ such that each $M$ in the club contains $\mathbb{P}$ and for all $p\in\mathbb{P}\cap M$, there is a $q\leq p$ such that for all names for a countable ordinal $\dot{\alpha}\in M$, $q\Vdash \exists \beta\in\check{M}\hspace{2pt} \dot{\alpha}=\beta$.
\end{definition}

\begin{definition}
    A poset $\mathbb{P}$ is stationary set preserving iff for every stationary $S\subseteq\omega_1$, $\Vdash_\mathbb{P}\check{S}$ is stationary.
\end{definition}

To define subcomplete forcing, we need the following preliminary notion:

\begin{definition}
    A transitive structure $M$ is full iff $\omega\in M$ and there is an ordinal $\gamma$ such that $L_\gamma(M)\models ZFC^-$ and for all $x\in M$ and $f:x\rightarrow M$ in $L_\gamma(M)$, $f"x\in M$.
\end{definition}

\begin{definition}
    A poset $\mathbb{P}$ is subcomplete iff for all sufficiently large $\theta$, if $\mathbb{P}\in H_\theta\subseteq N=L_\tau[A]\models ZFC^-$ for some $\tau>\theta$ and $A\subset \tau$, $s\in N$, $\bar{N}$ is a countable full model with an elementary embedding $\sigma:\bar{N}\rightarrow N$, $\bar{\theta},\bar{\mathbb{P}}, \bar{s}\in \bar{N}$ are such that $\sigma(\bar{\theta})=\theta$, $\sigma(\bar{\mathbb{P}})=\mathbb{P}$, and $\sigma(\bar{s})=s$, and $\bar{G}\subseteq\bar{\mathbb{P}}$ is an $\bar{N}$-generic filter, then there is a $p\in \mathbb{P}$ such that whenever $G\subseteq \mathbb{P}$ is $V$-generic and contains $p$, there is an elementary embedding $\sigma':\bar{N}\rightarrow N$ in $V[G]$ such that:
    \begin{itemize}
        \item $\sigma'(\bar{\theta})=\theta$
        \item $\sigma'(\bar{\mathbb{P}})=\mathbb{P}$
        \item $\sigma'(\bar{s})=s$
        \item $\sigma' "\bar{G}\subset G$
    \end{itemize}
\end{definition}

This is in fact the class that Fuchs and Switzer \cite{FSiteration} introduced as $\infty$-subcomplete forcing. Jensen's original definition of subcompleteness includes an extra condition which is seemingly necessary to prove that subcomplete forcing is closed under revised countable support iterations, but makes it unclear whether restrictions of subcomplete posets or even posets forcing-equivalent to subcomplete posets are necessarily subcomplete. Fuchs and Switzer, building on the work of Miyamoto, showed that there is a notion of iteration under which the class defined above is closed, that it is further closed under forcing equivalence, and that it preserves most of the same things as Jensen's class. In light of these desirable qualities, I have adopted it as my official definition.

\section{Forcing Axioms}
\label{section:fa}

The study of forcing axioms began in 1970, when Martin and Solovay, seeking a statement that would settle many of the independent questions resolved by the continuum hypothesis but would be consistent with $\lnot CH$, introduced Martin's Axiom \cite{MAorigin}.

\begin{definition}
    Martin's Axiom MA is the statement that if $\mathbb{P}$ is a forcing poset which satisfies the countable chain condition and $\mathcal{D}$ is a collection of dense subsets of $\mathbb{P}$ of size less than the continuum, then there is a $\mathcal{D}$-generic filter $F\subseteq\mathbb{P}$ (i.e. for all $D\in\mathcal{D}$, $D\cap F$ is nonempty).
\end{definition}

Martin and Solovay noted allowing $\mathbb{P}$ to be arbitrary leads to inconsistency when $\mathbb{P}$ can collapse $\omega_1$, and that the ccc "is not the weakest restriction on $\mathbb{P}$ which will prevent cardinal collapse, but it has the virtue of being strong enough to permit the proof of" the consistency of $MA$ with the continuum being any regular uncountable cardinal.

In the following decades, set theorists studied forcing axioms for a much broader range of forcing classes:

\begin{definition}
    If $\kappa>\omega_1$ is a cardinal and $\Gamma$ is a forcing class, the forcing axiom $FA_{<\kappa}(\Gamma)$ is the statement that for any $\mathbb{P}\in\Gamma$ and any collection $\mathcal{D}$ of dense subsets of $\mathbb{P}$ with $|\mathcal{D}|<\kappa$, there is a $\mathcal{D}$-generic filter $F\subseteq\mathbb{P}$.

    If $\Gamma$ is the class of proper posets, subcomplete posets, or posets preserving the stationarity of subsets of $\omega_1$, we write $FA_{<\omega_2}(\Gamma)$ as $PFA$, $SCFA$, or $MM$ (Martin's Maximum) respectively.
\end{definition}

It is more common to take the subscript to be the maximum allowable size of $\mathcal{D}$ rather than a strict upper bound. I have chosen to depart from this convention in order to allow general consistency proofs to be stated more cleanly and to handle the case where $\kappa$ is a limit cardinal: while a classical forcing axiom up to but not including a limit cardinal is simply the conjunction of the axioms at each cardinal below the limit, this will not hold for the principles which we study later. As a reminder of this unusual notation, I write the $<$ explicitly in the subscript.

$MA+\lnot CH$ is equiconsistent with $ZFC$, essentially because all the dense sets involved can be taken to be small and the ccc is inherited by arbitary subposets, so forcing with the ccc posets smaller than the desired size of the continuum is sufficient to add the desired filters. This is not true of most other forcing classes, so consistency proofs for other forcing axioms typically follow Baumgartner's proof of $Con(PFA)$ \cite{BaumgartnerPFA} in starting from a supercompact cardinal $\kappa$ and harnessing its reflection properties to show that all the desired filters can be added by iterating posets below $\kappa$. This process in fact leads to a model of a stronger principle:

\begin{definition}
    If $\nu\leq\omega_1$ is a cardinal and $\Gamma$ is a forcing class, $FA^{+\nu}(\Gamma)$ is the statement that for all $\mathbb{P}\in \Gamma$, all collections $\mathcal{D}$ of $\omega_1$ dense subsets of $\mathbb{P}$, and every sequence $\sigma=\langle\sigma_\alpha\sbp \alpha<\nu\rangle$ of $\mathbb{P}$ names for stationary subsets of $\omega_1$, there is a $\mathcal{D}$-generic filter $F\subseteq \mathbb{P}$ such that $\sigma_\alpha^F$ is a stationary subset of $\omega_1$ in $V$.
\end{definition}

$FA^+(\Gamma)$ means $FA^{+1}(\Gamma)$. $FA^{++}(\Gamma)$ is most often used to mean $FA^{+\omega_1}(\Gamma)$, but also sometimes $FA^{+2}(\Gamma)$, so I will avoid this notation.

Exactly what is meant by the interpretation of a name by a nongeneric filter is somewhat ambiguous; see Section \ref{section:interp} for more details. However, it will turn out that all reasonable definitions will coincide in this case if we add at most $\omega_1$ additional dense sets to $\mathcal{D}$, so the ambiguity is not too important.

In proving that the "plus versions" of forcing axioms hold in the standard models of classical forcing axioms, we make essential use of the fact that any forcing class for which $FA_{<\omega_2}$ is consistent must preserve the stationarity of subsets of $\omega_1$. The central goal of this work is to generalize the plus versions of forcing axioms to properties other than stationarity which the forcing class in question also preserves.

Also of interest are a natural weakening of forcing axioms introduced by Goldstern and Shelah \cite{GSbpfa}:

\begin{definition}
    If $\Gamma$ is a class of complete Boolean algebras and $\lambda\geq\kappa>\omega_1$ are cardinals, the bounded forcing axiom $BFA_{<\kappa}^{<\lambda}(\Gamma)$ is the statement that for all $\mathbb{B}\in \Gamma$ and all sets $\mathcal{A}$ of maximal antichains of $\mathbb{B}$ such that $|\mathcal{A}|<\kappa$ and for all $A\in\mathcal{A}$, $|A|<\lambda$, there is a (proper) filter $F\subset\mathbb{B}$ which meets every antichain in $\mathcal{A}$.
\end{definition}

This definition uses Boolean algebras rather than posets in order to guarantee the existence of sufficiently small maximal antichains.

I will frequently write "symmetric bounded forcing axiom" to refer to the case where $\lambda=\kappa$ (the axiom originally studied by Goldstern and Shelah) and "asymmetric bounded forcing axiom" for the case where $\lambda>\kappa$ (first proposed by Miyamoto \cite{miyamotosegments}).

\section{Maximality Principles}
\label{section:MP}

Maximality principles were introduced by Stavi and V\"a\"an\"anen \cite{svMP} and further developed by Hamkins \cite{hamkinsMP}. We will use the notation of the latter but a formulation more similar to the former, for which the following concept is needed:

\begin{definition}
    For a definable forcing class $\Gamma$, a formula $\phi$ is provably $\Gamma$-persistent iff 
$$ZFC\vdash \forall x(\phi(x)\rightarrow \forall \mathbb{Q}\in\Gamma \hspace{3pt}\Vdash_{\mathbb{Q}} \phi(\check{x}))$$
\end{definition}

\begin{definition}
    If $n$ is a positive integer, $\Gamma$ is a $\Sigma_n$-definable forcing class, and $S$ is a class of parameters, the $\Sigma_n$ maximality principle for $\Gamma$ with parameters in $S$ $\Sigma_n\mhyphen MP_\Gamma(S)$ is the statement that for every provably $\Gamma$-persistent $\Sigma_n$ formula $\phi$ and every $a\in S$, if there is some $\mathbb{P}\in \Gamma$ such that $\Vdash_\mathbb{P}\phi(\check{a})$, then $\phi(a)$ already holds in $V$.
\end{definition}

Maximality principles are appealing axioms because they capture the intuition that the universe should be as wide as possible, at least with regard to $\Gamma$-extensions and $\Sigma_n$ truth. They can easily be seen to imply symmetric bounded forcing axioms:

\begin{prop}
    \label{prop:mpbfa}
    (Folklore?) If $\Gamma$ is a forcing class, $\kappa$ is a regular uncountable cardinal, and $n\geq 1$, $\Sigma_n\mhyphen MP_\Gamma(H_\kappa)$ implies $BFA_{<\kappa}^{<\kappa}(\Gamma)$.
\end{prop}
\begin{proof}
    Let $\mathbb{B}\in \Gamma$ be a Boolean algebra and $A$ be a collection of fewer than $\kappa$ maximal antichains of $\mathbb{B}$, each of size less than $\kappa$. For $\lambda>2^{|\mathbb{B}|}$, let $X$ be an elementary substructure of $H_\lambda$ of size less than $\kappa$ containing $\mathbb{B}$, $A$, all antichains in $A$, and all elements of all antichains in $A$.

    Set $\bar{\mathbb{B}}:=\pi_X(\mathbb{B})$ and $\bar{A}:=\pi_X(A)$. If $G\subset\mathbb{P}$ is $V$-generic, then for each antichain $C\in A$, by genericity there is some $p\in C\cap G$. It follows that $\pi_X(p)\in \pi_X(C)\cap \pi_X"G$. Since every element of $\bar{A}$ is of the form $\pi_X(C)$ for some $C\in A$, it is thus $\Gamma$-forceable that there is a filter on $\bar{\mathbb{B}}$ meeting every antichain in $\bar{A}$. This property can easily be seen to be $\Sigma_1$, since it requires only an existential quantifier asserting the existence of the filter $F$ and bounded quantifiers over $F$, $\bar{A}$, and its elements. It is preserved by further forcing because all $\Sigma_1$ formulas remain true when moving to a larger structure in which the original model is transitive, and $\bar{A},\bar{\mathbb{B}}\in H_\kappa$ because they are contained in a transitive structure of size less than $\kappa$. Thus by the maximality principle there is some $F\subset\bar{\mathbb{B}}$ in $V$ meeting every antichain in $\bar{A}$. $\pi_X^{-1}"F$ then generates a filter on $\mathbb{B}$ meeting every antichain in $A$.
\end{proof}

George Leibman (related by Hamkins \cite{hamkinsMP}, Theorem 5.5) proved the above result for the case where $\Gamma$ is ccc forcing and $\kappa$ is the continuum (in which case bounded Martin's Axiom is simply equivalent to full $MA$). Kaethe Minden (\cite{mindensubcomplete}, Lemma 4.1.8 and Proposition 4.1.9) did the same for subcomplete forcing with $\kappa=\omega_2$.

Bagaria's result that bounded forcing axioms are principles of generic $\Sigma_1$-absoluteness (\cite{bagariagenabs}, Theorem 5) can be restated as:

\begin{theorem}
\label{thm:bagariagenabs}
    (Bagaria) If $\Gamma$ is a forcing class and $\kappa$ is a regular uncountable cardinal, $\Sigma_1\mhyphen MP_\Gamma(H_\kappa)$ is equivalent to $BFA_{<\kappa}^{<\kappa}(\Gamma)$.
\end{theorem}

Hamkins phrases maximality principles in terms of forceably necessary formulas (i.e. formulas which can be forced to be true and to remain true in all further forcing extensions), which is subtly different from our formulation in terms of forceable and provably persistent formulas, since a formula can be preserved in all forcing extensions of a particular model without that preservation being provable in ZFC. However, the following observations, essentially due to David Asper\'o in the last part of the proof of Theorem 2.6 in \cite{asperomax}, show that the two formulations are equivalent if, like Hamkins, one is only interested in the version of the maximality principle encompassing formulas of arbitrary complexity. Since we will primarily be considering the $\Sigma_n$-restricted forms, the provably persistent formulation is more convenient.

\begin{definition}
    For $\Gamma$ a forcing class and $\phi$ a formula in the language of set theory, $\square_\Gamma\phi$ is the assertion that $\phi$ holds in all $\Gamma$-extensions of the universe.
\end{definition}
\begin{lemma}
    \label{lemma:neccompl}
    If $\Gamma$ is a $\Sigma_n$-definable forcing class and $\phi$ is a $\Sigma_n$ formula, $\square_\Gamma\phi$ is $\Pi_{n+1}$ and $\square_\Gamma\lnot\phi$ is $\Pi_n$. If $\Gamma$ is provably closed under two-step iterations, both are provably $\Gamma$-persistent.
\end{lemma}
\begin{proof}
    We can express $\square_\Gamma\phi$ as "for all $\mathbb{P}$, $\mathbb{P}\not\in\Gamma$ or $\Vdash_\mathbb{P}\phi$". The first clause of the disjunction is $\Pi_n$ and the second $\Sigma_n$, so with the added universal quantifier we get a $\Pi_{n+1}$ formula. If we used $\lnot\phi$ in place of $\phi$, it would be a universal quantifier added to a disjunction of two $\Pi_n$ formulas, which is $\Pi_n$.

    For provable persistence, work in $ZFC$ and let $\mathbb{P}\in\Gamma$. Then for any $\mathbb{P}$-name $\dot{\mathbb{Q}}$ such that $\Vdash_\mathbb{P} \dot{\mathbb{Q}}\in\Gamma$, $\mathbb{P}*\dot{\mathbb{Q}}\in \Gamma$, so $\square_\Gamma \phi$ implies that  $\Vdash_{\mathbb{P}*\dot{\mathbb{Q}}}\phi$. It follows that $\Vdash_\mathbb{P}\forall\mathbb{Q}\in\Gamma \Vdash_\mathbb{Q}\phi$. We therefore have:
    $$ZFC\vdash \square_\Gamma\phi\rightarrow\forall \mathbb{P}\in\Gamma \Vdash_\mathbb{P}\square_\Gamma \phi$$
\end{proof}

\begin{corollary}
    If $\Gamma$ provably contains the trivial forcing and is closed under two-step iterations, then for any class of parameters $S$, $\Sigma_{n+2}\mhyphen MP_\Gamma(S)$ implies the Hamkins-style maximality principle for $\Sigma_n$ formulas (and in fact $\Pi_{n+1}$ formulas) with parameters in $S$.
\end{corollary}

Hamkins notes (\cite{hamkinsMP}, Observation 1.3) that if $\Gamma$ is the class of all forcing then the parameters must be contained in $H_{\omega_1}$, because any parameter can be forced to be hereditarily countable by adding a surjection from $\omega$ to its transitive closure, and this persists to all forcing extensions. Similarly (as Hamkins discusses after Corollary 5.4), if $\Gamma$ can collapse arbitrary cardinals to $\omega_1$, we cannot consistently allow parameters outside of $H_{\omega_2}$, and if $\Gamma$ can add arbitrarily many reals but can't collapse the cardinality of the continuum, the maximal parameter set is $H_{2^{\aleph_0}}$, since for any $x$ it is $\Gamma$-forceably $\Gamma$-necessary (by adding sufficient reals) that the cardinality of the transitive closure of $x$ is less than the continuum.

However, as Proposition \ref{prop:mpbfa} shows, forcing axioms can be regarded as maximality principles for the specific $\Sigma_1$ assertion that there exists a filter meeting a desired collection of dense sets. While symmetric bounded forcing axioms assert this only for collections of antichains small enough to be collapsed into the allowable parameter set of the corresponding maximality principle, more general forcing axioms assert it for antichains too large for this to work. (Of course one always can mimic the procedure in the proof of $\ref{prop:mpbfa}$ to obtain a small $\bar{\mathbb{B}}$ and $\bar{A}$ from a large Boolean algebra $\mathbb{B}$ and collection of maximal antichains $A$, but if some of the antichains in $A$ are too large, the corresponding antichains in $\bar{A}$ will necessarily omit some of their conditions; then a generic $G\subset\mathbb{B}$ which meets those antichains at those omitted conditions will project to a filter on $\bar{\mathbb{B}}$ which does not meet those antichains in $\bar{A}$, so $\mathbb{B}$ does not force that there is a filter meeting all antichains in $\bar{A}$, so we cannot apply a maximality principle to obtain a forcing axiom.)

Thus classical forcing axioms can be viewed as a highly specific maximality principle intertwined with a reflection principle. Letting $\mathcal{D}$ be the set of all dense sets of a poset $\mathbb{P}$ in the ground model, it is forceably necessary that there is a filter meeting all elements of $\mathcal{D}$; though we cannot hope to find such a filter in the ground model, we can consistently shrink $\mathcal{D}$ down to a set $\bar{\mathcal{D}}$ smaller than the continuum (for ccc forcing) or $\omega_2$ (for proper, semiproper, or subcomplete forcing) in a highly controlled way and find a ground-model filter meeting all dense sets in $\bar{\mathcal{D}}$.

By adding the ability to realize the truth of a forceably necessary formula with excessively large parameters in the ground model with a reflected parameter of reasonable size, unbounded and asymmetrically bounded forcing axioms gain considerable consistency strength over maximality principles, despite the restriction on allowable formulas. Our central goal may alternatively be stated as loosening this restriction on formulas, and in doing so unifying maximality principles with classical forcing axioms.

\section{Boolean-Valued Models}
\label{section:boolean}

For some proofs, it will be convenient to use Boolean-valued models and their quotients, so we briefly review the basic theory here. For proofs of the lemmas below and more details, see the relevant section of Jech's Chapter 14 \cite{jech} or the first few sections of Hamkins and Seabold \cite{HSultrapower}.

\begin{definition}
    Given a complete Boolean algebra $\mathbb{B}$, a $\mathbb{B}$-valued model in signature $\tau$ consists of a class $M$ (called the class of names) and a map $\phi(a_1,\dotsc a_n)\mapsto \llbracket \phi(a_1,\dotsc, a_n)\rrbracket$ from the first-order formulas in the language generated by $\tau$, with assignments of their free variables to elements of $M$, to $\mathbb{B}$, such that the expected Tarski-style relations between the Boolean values of different formulas hold.
    \end{definition}

    \begin{definition}
        If $\mathbb{B}\in N\models ZFC^-\land ``\mathbb{B}$ is a complete Boolean algebra'', then $N^\mathbb{B}$ is the Boolean-valued model constructed within $N$ whose class of names is inductively defined as the class of functions (in $N$) whose domain is a set of $\mathbb{B}$-names and whose codomain is $\mathbb{B}$, and whose Boolean valuation is recursively given by:
        \begin{align*}
            \llbracket \tau\in\sigma\rrbracket &=\bigvee\limits_{\rho\in dom(\sigma)}\llbracket\tau=\rho\rrbracket\land\sigma(\rho)\\
            \llbracket \tau=\sigma\rrbracket&= \llbracket \tau\subseteq\sigma\rrbracket\land\llbracket\sigma\subseteq \tau\rrbracket\\
            \llbracket\tau\subseteq\sigma\rrbracket&=\bigwedge\limits_{\rho\in dom(\tau)} (\lnot\tau(\rho)\lor \llbracket\rho\in\sigma\rrbracket)
        \end{align*}
    \end{definition}
    \begin{definition}
        A Boolean-valued model is full iff for all formulas $\phi$ in its language with $n+1$ free variables and names $b_1,\dotsc, b_n$, there is a name $a$ such that $\llbracket\exists x\hspace{2pt}\phi(x, b_1,\dotsc, b_n)\rrbracket=\llbracket\phi(a, b_1,\dotsc, b_n)\rrbracket$.
    \end{definition}

    \begin{lemma}
        If $\mathbb{B}\in N\models ZFC^-\land ``\mathbb{B}$ is a complete Boolean algebra'', then $N^\mathbb{B}$ is full.
    \end{lemma}
    \begin{definition}
        If $\mathbb{B}$ is a complete Boolean algebra, $M$ is a $\mathbb{B}$-valued model in the language of set theory, and $U\subset\mathbb{B}$ is an ultrafilter, then for $\sigma$ and $\tau$ names in $M$, $\sigma=_U\tau$ means that $\llbracket\sigma=\tau\rrbracket\in U$ and $\sigma\in_U\tau$ means that $\llbracket\sigma\in\tau\rrbracket\in U$.  $M/U$ is the ordinary first-order model of set theory whose elements are $=_U$-equivalence classes (using Scott's trick if $M$ is a proper class) and whose element relation is given by lifting $\in_U$ to equivalence classes.
    \end{definition}
    \begin{lemma}
        \label{lemma:booleanlos}
        If $\mathbb{B}\in N\models ZFC^-\land ``\mathbb{B}$ is a complete Boolean algebra'', and $U\subset\mathbb{B}$ is an ultrafilter, then
        \begin{enumerate}[(a)]
            \item $N^\mathbb{B}/U\models ZFC^-$
            \item $N$ embeds into $N^\mathbb{B}/U$ via the map $x\mapsto[\check{x}]_U$, where $dom(\check{x})=\{\check{y}\sbp y\in x\}$ and $\check{x}(\check{y})=1_\mathbb{B}$ for all $y\in x$.
            \item For all formulas $\phi$ in the language of set theory and $\sigma\in N^\mathbb{B}$, $N^\mathbb{B}/U\models\phi([\sigma]_U)$ iff $\llbracket\phi(\sigma)\rrbracket\in U$.
        \end{enumerate}
    \end{lemma}

    Two things are important to note here. First, there is no need for $U$ to be $N$-generic, and it could even be in $N$. For non-generic $U$, $[\dot{G}]_U$ (where $\dot{G}(\check{b})=b$ for all $b\in^N\mathbb{B}$) will still play the role of a generic filter in $N^\mathbb{B}/U$, since this holds with Boolean value one. However, rather than being a forcing extension of $N$ itself, $N^\mathbb{B}/U$ will be a forcing extension of some elementary extension $\bar{N}$ of $N$, consisting of equivalences classes of names which are forced to be equal to some check name. The elements of $\bar{N}-N$ arise from those names such that $U$ misses the maximal antichain that decides exactly which element of $N$ they are equal to.

    Second, there is no need for $N$ to be wellfounded. Whereas the ordinary construction of a forcing extension $N[G]$ as the interpretation $\sigma^G$ of each $\sigma\in N$ relies on recursion on the $\in$-relation of $N$, which runs into difficulties when $N$ is illfounded, the quotient construction has no such limitations. (Of course, we still need to perform recursion on $\in^N$ to define Boolean values of formulas, but this can be done inside of $N$, where it is justified because $N$ satisfies the foundation axiom, whereas interpreting names by a generic filter requires the recursion to be performed externally.)

\section{Interpretations of Names}
\label{section:interp}

There are at least four natural ways to interpret or evaluate names\footnote{When $\mathbb{P}$ is an arbitrary poset rather than a complete Boolean algebra, a $\mathbb{P}$-name is defined, following Kunen (\cite{kunen}), as merely a set of ordered pairs with the first coordinate a $\mathbb{P}$-name and the second an element of $\mathbb{P}$, not necessarily a function.} with filters, some of which coincide when the filter is generic but can diverge significantly in general. First, our official definition will be:

\begin{definition}
    If $\mathbb{P}$ is a forcing poset, $\dot{a}$ is a $\mathbb{P}$-name, and $F\subseteq\mathbb{P}$ is a filter, the standard interpretation of $\dot{a}$ by $F$ is given recursively by 
    $$\dot{a}^F=\{\tau^F\hspace{2pt}|\hspace{2pt} \exists p\in F\hspace{4pt} \langle \tau, p\rangle\in \dot{a}\} $$
\end{definition}

We also have the following concept, frequently used in the statement of $FA^+$:

\begin{definition}
    If $\mathbb{P}$ is a forcing poset, $\dot{a}$ is a $\mathbb{P}$-name, and $F\subseteq\mathbb{P}$ is a filter, the quasi-interpretation of $\dot{a}$ by $F$ is given by
    $$\dot{a}^{(F)}=\{x\in V\hspace{2pt}|\hspace{2pt} \exists p\in F\hspace{4pt} p\Vdash \check{x}\in\dot{a}\}$$
\end{definition}

The quasi-interpretation terminology comes from Schlicht and Turner \cite{stNP}. Quasi-interpretations are a natural way to attempt to evaluate names in the ground model, but they are badly behaved except on names for subsets of $V$, so it is natural to modify the definition to the following recursive construction:

\begin{definition}
    If $\mathbb{P}$ is a forcing poset, $\dot{a}$ is a $\mathbb{P}$-name, and $F\subseteq\mathbb{P}$ is a filter, the recursive quasi-interpretation of $\dot{a}$ by $F$ is given by
    $$\dot{a}^{((F))}=\{\tau^{((F))}\hspace{2pt}|\hspace{2pt}\tau\in dom(\dot{a})\land\exists p\in F\hspace{4pt} p\Vdash \tau\in\dot{a}\}$$
\end{definition}

Finally, we can interpret names within the quotients of Boolean algebras discussed in the previous section:

\begin{definition}
    If $\mathbb{B}$ is a complete Boolean algebra, $U\subset \mathbb{B}$ is an ultrafilter, and $\dot{a}$ is a $\mathbb{B}$-name whose equivalence class is in the well-founded part of $V^\mathbb{B}/U$, its Boolean interpretation by $U$ is the value of its equivalence class $[\dot{a}]_U$ under the transitive collapse isomorphism on $wfp(V^\mathbb{B}/U)$.
\end{definition}

Sometimes a name of interest $\dot{a}$ is not in the well-founded part of $V^\mathbb{B}/U$, but there is some small structure $N$ containing $\mathbb{B}$ and $\dot{a}$ such the equivalence class of $\dot{a}$ is in the well-founded part of $N^\mathbb{B}/U$; in such cases it is useful to consider the relativized Boolean interpretation $[\dot{a}]_U^N$.

As an example to illustrate the distinctions between these interpretations, consider $\dot{a}=\{\langle \dot{G}, 1\rangle\}$, where $\dot{G}:=\{\langle\check{p},p\rangle\sbp p\in\mathbb{P}\}$ is the canonical name for the generic filter. For any filter $F\subseteq\mathbb{P}$, $\dot{a}^F=\dot{a}^{((F))}=\{F\}$, while if $\mathbb{P}$ is atomless $1\Vdash \dot{G}\neq \check{x}$ for all $x\in V$, so $\dot{a}^{(F)}=\emptyset$. Finally, if $\mathbb{B}$ is the Boolean completion of $\mathbb{P}$ and $U$ is an ultrafilter on $\mathbb{B}$, then $[\dot{a}]_U=\{[\dot{G}]_U\}$, where the singleton is in the sense of $V^\mathbb{B}/U$ and may not correspond to an actual singleton if $[\dot{G}]_U$ is in the ill-founded part of $V^\mathbb{B}/U$ (in which case the Boolean interpretation of $\dot{a}$ will be undefined). $[\dot{G}]_U$ in turn consists of all the elements of $U$ (or more properly the equivalence classes of their check names) as well as potentially the equivalence classes of names which are generically equal to some element of $U$, but $U$ misses the dense set which determines exactly which of its elements they are equal to.

To see how the standard interpretation and recursive quasi-interpretation can differ for non-generic $F$, let $A=\{p_\alpha \sbp \alpha<\lambda\}$ be a maximal antichain of $\mathbb{P}$ such that $F\cap A=\emptyset$. Then if $\dot{a}=\{\langle \check{\emptyset}, p_\alpha\rangle\sbp \alpha<\lambda\}$, $\dot{a}^F=\emptyset$ but $\dot{a}^{((F))}=\{\emptyset\}$.

We now examine how partial genericity is sufficient to ensure that different interpretations coincide.

\begin{lemma}
    If $N$ is a transitive model of $ZFC^-$ and "$\mathbb{B}$ is a complete Boolean algebra", $\dot{a}\in N^\mathbb{B}$, and $U\subseteq\mathbb{B}$ is an ultrafilter which meets all antichains of $\mathbb{B}$ in $N$ whose $N$-cardinality is at most $|trcl(\{\dot{a}\})|^N$, then $[\dot{a}]_U^N$ is contained in the well-founded part of $N^\mathbb{B}/U$ and its value under the transitive collapse isomorphism is $\dot{a}^{((U))}$.
    \label{lemma:rqieqbool}
\end{lemma}
\begin{proof}
    If $[\dot{a}]_U$ is not in the well-founded part of $N^{\mathbb{B}}/U$, then there is a sequence of names $\langle \sigma_n\sbp n\in \omega\rangle$ such that $\sigma_0=\dot{a}$, $[\sigma_{n+1}]_U^N\in_U [\sigma_n]_U^N$, and $\sigma_n\in N$ for all $n$. As far as possible, we choose each name in the sequence to be an actual element of the domain of the previous one, but by the well-foundedness of $\in$, this can't continue forever, so there is some $n$ such that $\sigma_n\in trcl(\{\dot{a}\})$ but $\sigma_{n+1}$ is not in the equivalence class of any name in $dom(\sigma_n)$. Therefore because $dom(\sigma_n)\subset trcl(\dot{a})$, $\{\lnot\llbracket \sigma_{n+1}\in \sigma_n\rrbracket\}\cup\{\llbracket \sigma_{n+1}=\tau\rrbracket\sbp \tau\in dom(\sigma_n)\}$ is a predense set definable in $N$ of $N$-cardinality at most $|trcl(\{\dot{a}\})|^N$. Thus $U$ must meet it, contradicting the choice of $\sigma_{n+1}$.

    Now we let $\pi$ be the Mostowski collapse isomorphism on the well-founded part of $N^\mathbb{B}/U$ and prove by induction that $\pi([\dot{a}]_U^N)=\dot{a}^{((U))}$. Assume that $\sigma\in trcl(\{\dot{a}\})$ and for all $\tau\in dom(\sigma)$, $\pi([\tau]_U)=\tau^{((U))}$. Since $\sigma^{((U))}=\{\tau^{((U))}\sbp \tau\in dom(\dot{a})\land \llbracket \tau\in \dot{a}\rrbracket\in U\}=\{\pi([\tau]_U)\sbp \tau\in dom(\dot{a})\land \llbracket \tau\in \dot{a}\rrbracket\in U\}$, certainly $\sigma^{((U))}\subseteq \pi([\sigma]_U)$, and to show the reverse inclusion it is sufficient to show that if $\tau'\in_U \sigma$ for any $\mathbb{B}$-name $\tau'\in N$, then there is a $\tau\in dom(\sigma)\cap [\tau']_U$. However, since by definition $\llbracket \tau'\in \sigma\rrbracket=\bigvee\limits_{\tau\in dom(\sigma)}(\llbracket \tau'=\tau\rrbracket\land \sigma(\tau))$, by essentially the same argument as in the previous paragraph $U$ must include one of the $\llbracket \tau'=\tau\rrbracket$. 
\end{proof}

To prove that the standard and recursive quasi-interpretations are equivalent given enough genericity, the following elementary lemma is helpful:

\begin{lemma}
If $\mathbb{B}$ is a Boolean algebra, $p\in\mathbb{B}$, and $A\subset\mathbb{B}$ is an antichain maximal below $p$, then for any $q\in\mathbb{B}$, $A\land q:=\{a\land q\hspace{2pt}|\hspace{2pt} a\in A,\hspace{4pt} a\land q>0\}$ is an antichain maximal below $p\land q$ (where we consider the empty set to be the unique antichain maximal below 0).
\label{lemma:macConj}
\end{lemma}
\begin{proof}
$A\land q$ is an antichain because for all $a, b\in A$ such that $a\land q, b\land q\in A\land q$, any common lower bound of $a\land q$ and $b\land q$ is a common lower bound of $a$ and $b$, so because $A$ is an antichain it must be 0. $A\land q$ is maximal below $p\land q$ because for any nonzero $r\leq p\land q$, since $r\leq p$ there is some $a\in A$ such that $r\land a>0$ by the maximality of $A$; then since $r\land a\leq r\leq p\land q\leq q$, $r\land a\land q=r\land a>0$, so $r$ is compatible with $a\land q\in A\land q$.
\end{proof}

\begin{lemma}
    If $\mathbb{B}$ is a complete Boolean algebra, $\kappa$ is a regular cardinal, and $\dot{a}\in H_\kappa$ is a $\mathbb{B}$-name, then there is a set of fewer than $\kappa$ predense subsets of $\mathbb{B}$, each of size less than $\kappa$, such that if $F\subseteq \mathbb{B}$ is a filter meeting all of them, $\dot{a}^F=\dot{a}^{((F))}$. Furthermore, this set of predense sets is an element of any transitive $ZFC^-$ model containing $\dot{a}$ and $\mathbb{B}$ and is definable over any such model using $\dot{a}$ and $\mathbb{B}$ as parameters.
    \label{lemma:standeqrqi}
\end{lemma}
\begin{proof}
    We proceed by $\in$-induction on $\dot{a}$. Assume that $\sigma\in trcl(\{\dot{a}\})$ is such that the lemma holds for all $\tau\in dom(\sigma)$. Then for each (of the fewer than $\kappa$) $\tau\in dom(\sigma)$, we have a cardinal $\lambda_\tau<\kappa$ and a collection of predense sets $\{A_\alpha^\tau\sbp \alpha<\lambda_\tau\}$, each smaller than $\kappa$, such that if $F$ is a filter meeting all of the $A_\alpha^\tau$, then $\tau^F=\tau^{((F))}$.

    Using the notation of Lemma \ref{lemma:macConj}, for each such $\tau$ and each $\alpha<\lambda_\tau$, we can form the predense set $A_\alpha^\tau\land \sigma(\tau)\cup\{\lnot\sigma(\tau)\}$. If $F$ is any filter meeting all of these predense sets, then every element of $\sigma^F$ must have a name $\tau\in dom(\sigma)$ such that $\sigma(\tau)\in F$, so $F$ must meet all the $A_\alpha^\tau$ for this $\tau$ and thus $\tau^F=\tau^{((F))}$. Since $\llbracket \tau\in\sigma\rrbracket\geq \sigma(\tau)$, $\tau^{((F))}\in\sigma^{((F))}$, so $\sigma^F\subseteq \sigma^{((F))}$.
    
    For the reverse inclusion, we require that for each $\tau\in dom(\sigma)$ $F$ also meet the predense set $B_\tau:=\{\llbracket\tau=\rho\rrbracket\land\sigma(\rho)\sbp \rho\in dom(\sigma)\}\cup \{\lnot\llbracket \tau\in\sigma\rrbracket\}$. Then whenever $\llbracket\tau\in\sigma\rrbracket\in F$, putting $\tau^{((F))}$ in $\sigma^{((F))}$, there is some $\rho\in dom(\sigma)$ such that $\rho^F\in \sigma^F$ and $\llbracket \tau=\rho\rrbracket\in F$. Assuming $F$ meets all the $A_\alpha^\tau$ and $A_\alpha^\rho$, $\tau^F=\tau^{((F))}$ and $\rho^F=\rho^{((F))}$. Each $x\in \tau^F$ has a name $\dot{x}\in dom(\tau)$ such that $\tau(\dot{x})\in F$, which by the definition of $\llbracket \tau=\rho\rrbracket$ implies that $\llbracket \dot{x}\in\rho\rrbracket\in F$, so $\tau^F\subseteq \rho^{((F))}$. The same argument shows that $\rho^F\subseteq\tau^{((F))}$, so both interpretations of each name are equal to both interpretations of the other. Therefore $\tau^{((F))}\in \sigma^F$, so $\sigma^{((F))}\subseteq \sigma^F$ and the two interpretations of $\sigma$ by $F$ are the same, provided $F$ meets $A^\tau_\alpha$ and $B_\tau$ for all $\tau\in dom(\sigma)$ and $\alpha<\lambda_\tau$.

    To see that the set of relevant predense sets is definable in $ZFC^-$ with $\dot{a}$ and $\mathbb{B}$ as parameters, observe that each must initially arise as a $B_\tau$ for some name $\tau\in trcl(\dot{a})$ and then is modified only in very simple ways later on in the induction. Since the definition of $B_\tau$ only involves Boolean values of atomic formulas, which are $\Delta_1$-definable and therefore absolute to arbitrary transitive structures, each predense set we require $F$ to meet and the set of all of them can be formed from subsets of $trcl(\dot{a})$ by applications of the collection axioms. This resulting set is then definable as the set of everything constructed by the end of the induction described above. 
\end{proof}

Going forward, we will usually make assumptions which allow us to satisfy the hypotheses of the preceding lemmas at no additional cost, so we will make no further use of the quasi-interpretation and use the other three interpretations interchangeably as is convenient.

\section{Set-Theoretic Geology}
\label{section:geology}

Set-theoretic geology, launched by Fuchs, Hamkins, and Reitz in \cite{FHRgeology}, is the study of the grounds of the universe, that is the inner models $W$ such that $V=W[G]$ for some $W$-generic filter $G\in V$ on a (set) forcing notion in $W$. Since it will prove useful in establishing certain separation results, we review the basics here. First, some preliminary facts and definitions:

\begin{definition}
    (Hamkins \cite{Hamkinsapproxcover}) For $M\subseteq N$ transitive classes satisfying some sufficient fragment of $ZFC$ and $\delta$ a cardinal of both, $M$ has the $\delta$-cover property in $N$ (or the extension $M\subseteq N$ has the $\delta$-cover property) iff whenever $A\in N$, $A\subset M$, and $|A|^N<\delta$, there is a $B\in M$ such that $A\subseteq B$ and $|B|^M<\delta$.
\end{definition}

\begin{definition}
    (Hamkins \cite{Hamkinsapproxcover}) For $M$, $N$, and $\delta$ as above, $M$ has the $\delta$-approximation property in $N$ iff whenever $A\in N$, $A\subset M$, and for all $b\in M$ with $|b|^M<\delta$ we have $A\cap b\in M$, then $A\in M$.
\end{definition}

If $\mathbb{P}$ is a forcing poset in $W$, $|\mathbb{P}|^W<\delta$, and $G\subseteq\mathbb{P}$ is $W$-generic, then $W$ has the $\delta$-cover and $\delta$-approximation properties in $W[G]$. The same holds for some posets larger than $\delta$; see \cite{Hamkinsapproxcover} for more details.

\begin{definition}
\label{def:ZFCdelta}
    (Reitz \cite{ReitzGA}) $ZFC_\delta$ is the theory in the language of set theory with an added constant symbol for $\delta$ consisting of the axioms of Zermelo set theory with foundation and choice, the assertion that $\delta$ is a regular cardinal, the assertion that every set is coded by a set of ordinals, and the restricted replacement scheme asserting that the image of $\delta$ under any first-order-definable mapping is a set.

    If a regular cardinal $\delta$ has already been chosen, to say that $M\models ZFC_\delta$ for some transitive $M$ with $\delta\in M$ is to say that this holds with the constant symbol interpreted as that particular $\delta$.
\end{definition}

\begin{lemma}
    (Reitz) If $\delta$ is a regular cardinal, $\beth_\theta=\theta$ and $cf(\theta)>\delta$, then $V_\theta\models ZFC_\delta$.
\end{lemma}

\begin{lemma}
    (Hamkins)\\ (Uniqueness Lemma) Suppose $M$, $M'$, and $N$ are transitive classes with $M\subseteq N$ and $M'\subseteq N$, $\delta$ is a regular cardinal in $N$, all three satisfy $ZFC_\delta$, $M$ and $M'$ has the $\delta$-cover and $\delta$-approximation properties in $N$, $(^{<\delta}2)^M=(^{<\delta}2)^{M'}$, and $(\delta^+)^M=(\delta^+)^{M'}=(\delta^+)^N$. Then $M=M'$.
\end{lemma}

The preceding ideas enabled Hamkins to prove the foundational theorem of set-theoretic geology, that the grounds are uniformly first-order definable (though Laver and Woodin proved similar results independently).

\begin{theorem}
    \label{thm:grounddef}
    (\cite{FHRgeology}, Theorem 6) \\(Ground Definability Theorem) There is a $\Sigma_2$ formula $\rho$ with two free variables such that for all set forcing grounds $W$, there is a parameter $r\in W$ such that for all $x$, $x\in W\longleftrightarrow \rho(x, r)$. Specifically, if $\mathbb{P}$ is a poset, $G\subseteq\mathbb{P}$ is a $W$-generic filter, and $V=W[G]$, we can set $\delta:=|\mathbb{P}|^+$ and $r:=\mathcal{P}(\mathbb{P})^W$; then $\rho(x, r)$ asserts that there exists a $\theta$ and an $M\subseteq V_\theta$ such that:
    \begin{enumerate}
        \item $\beth_\theta=\theta$
        \item $cf(\theta)>\delta$
        \item $M\models ZFC_\delta$
        \item $M\subseteq V_\theta$ has the $\delta$-cover and $\delta$-approximation properties
        \item $r=\mathcal{P}(\mathbb{P})^M$
        \item $(\delta^+)^M=\delta^+$
        \item $x\in M$
    \end{enumerate}
    $W$ is also definable by a $\Pi_2$ formula $\rho'(x,r)$ asserting that for all $\theta>rank(x)$ and $M\subseteq V_\theta$, if the first six conditions hold, then $x\in M$.
\end{theorem}

The idea of the proof is that for any suitable $\theta$, $W\cap V_\theta$ will satisfy conditions 3-6 for $M$, and by the Uniqueness Lemma it is the only such $M$, so we will have $x\in W$ iff we can find a $\theta$ and an $M\subseteq V_\theta$ satisfying conditions 1-6 with $x\in M$ iff all $\theta$ and $M\subseteq V_\theta$ satisfying 1-6 have $x\in M$. For the formula complexity, note that the definition of $V_\theta$ is $\Pi_1$, while conditions 1-6 can be expressed with quantifiers over $M$ and $V_\theta$ with at most one unbounded universal quantifier (which becomes an existential quantifier in the hypothesis of the conditional in $\rho'$). Note also that $\delta$ and $\mathbb{P}$ are definable within $V_\theta$ from $r$, so there is no need to include them as additional parameters.

\begin{lemma}
    \label{lemma:groundrel}
    If $W_r=\{x\sbp \rho(x, r)\}$ is a ground of $V$, $a\in W_r$, and $\phi$ is a $\Sigma_n$ or $\Pi_n$ formula  with $n\geq 2$, then $W_r\models \phi(a)$ is expressible in $V$ as a $\Sigma_{n}$ or $\Pi_{n}$ formula with $a$ and $r$ as parameters.
\end{lemma}
\begin{proof}
    This proof is due to Farmer Schlutzenberg in an answer to a question I asked on MathOverflow \cite{SchlutzenbergMO}. We proceed by induction on the complexity of $\phi$. If $\phi$ is $\Sigma_2$ or $\Pi_2$, we use the facts that $W_r$ satisfies a $\Sigma_2$ formula iff it holds in some $V^{W_r}_\theta$ where $\theta$ is a beth fixed point, and it satisfies a $\Pi_2$ formula iff it holds in all such $V^{W_r}_\theta$ containing the parameters. (For readers unfamiliar with this fact, it will follow as a special case of Proposition \ref{prop:c1bfp} and Lemma \ref{lemma:Cnaddexist}.) Hence in the $\Sigma_2$ case we can express $W_r\models \phi(a)$ as "there exist $\theta$ and $M\subseteq V_\theta$ such that conditions 1-7 from Theorem \ref{thm:grounddef} hold and $M\models \phi(a)$", while if $\phi$ is $\Pi_2$, we can similarly say "for all $\theta$ and $M\subseteq V_\theta$, if conditions 1-7 hold, then $M\models \phi(a)$".

    Now let $n>2$. If $n$ is odd, we assume for induction that for any $\Pi_{n-1}$ formula $\psi$, $W_r\models \psi(a, y)$ is expressible as $n-3$ layers of quantifiers, followed by the assertion that for all $\theta$ and $M\subseteq V_\theta$ satisfying the appropriate conditions, if $M$ contains $a$, $y$, and all the variables bound by the outer quantifiers, then $M$ satisfies the $\Pi_2$ formula obtained by stripping away $n-3$ layers of quantifiers from $\psi$. We similarly assume that $\Sigma_{n-1}$ formulas are expressible via $n-3$ layers of outer quantifiers, followed by the assertion that there exists a $\theta$ and $M\subseteq V_\theta$ satisfying the conditions and containing the appropriate variables and parameters such that $M$ satisfies the inner $\Sigma_2$ formula. If $n$ is even, our induction hypotheses are the same but with $\Sigma_{n-1}$ and $\Pi_{n-1}$ switched.

    Now if $W_r\models \phi(a)$ for some $\Sigma_n$ formula $\phi$ of the form $\exists y \hspace{3pt}\psi(a, y)$ where $\psi$ is $\Pi_{n-1}$ and $n$ is odd, fix some $b$ such that $W_r\models \psi(a, b)$. Then by the induction hypothesis, this fact is expressible by a string of $n-3$ quantifiers followed by an assertion about all $\theta$ and $M\subseteq V_\theta$ with $a, b\in M$, so $W_r\models\phi(a)$ is expressible as a similar statement preceded by a string of $n-2$ quantifiers, with all their bound variables required to be in $M$. Conversely, if the formula that we wish to capture $W_r\models \phi(a)$ holds, then there is such a $\theta$, $M$, and $b\in M$ so that the formula expressing $W_r\models \psi(a, b)$ holds, so by the induction hypothesis $\psi(a,b)$ in fact holds in $W_r$ and thus $\phi(a)$ holds in $W_r$. The arguments in the cases where $\phi$ is $\Pi_n$ and/or $n$ is even are almost identical.
\end{proof}

Of course, not all parameters $r$ will actually define a ground, so when we wish to quantify over grounds, it is useful to know that:

\begin{lemma}
\label{lemma:groundsuccess}
    (Bagaria-Hamkins-Tsaprounis-Usuba, \cite{BHTUssnever}) The assertion that a parameter $r$ successfully defines a ground $W_r$ is $\Sigma_3$ expressible in $r$.
\end{lemma}
\begin{proof}
    We express this as "there exist a poset $\mathbb{P}=(\bigcup r, \leq_{\mathbb{P}})$ and a filter $G\subseteq\mathbb{P}$ which meets all dense sets in $r$ such that for all beth fixed points $\theta$ of cofinality greater than $\delta:=|\mathbb{P}|^+$ with $\mathbb{P}\in V_\theta$, there is an $M_\theta\subseteq V_\theta$ with the $\delta$-cover and $\delta$-approximation properties in $V_\theta$ such that $M_\theta\models ZFC_\delta$, $\mathbb{P}\in M_\theta$, $r=\mathcal{P}(\mathbb{P})^{M_\theta}$, $(\delta^+)^{M_\theta}=\delta^+$, and $V_\theta=M_\theta[G]$". If $W_r$ is in fact a ground, we can take $M_\theta=W_r\cap V_\theta$. Conversely, if the quoted statement holds, then the $M_\theta$ are uniquely defined because of the Uniqueness Lemma and $W_r$ as defined in Theorem \ref{thm:grounddef} can be seen to be equal to their union. Jech's Theorem 13.9 \cite{jech} asserts that a transitive class is an inner model of ZF iff it is closed under G\"odel operations and every subset of the class is covered by an element of the class, which can straightforwardly and tediously be verified for $W_r$. Choice and $V=W_r[G]$ follow because each set in $W_r$ has a well-ordering in some $M_\theta$ and each set in $V$ has a $\mathbb{P}$-name in some $M_\theta$.

    Bagaria, Hamkins, Tsaprounis, and Usuba observe (several paragraphs after Lemma 5 in \cite{BHTUssnever}) that, given $\mathbb{P}$ and $G$, any counterexample to the rest of the statement can be witnessed by any sufficiently large $V_\alpha$. Consequently, everything after the initial existential quantifiers can be expressed as the $\Pi_2$ statement that every $V_\alpha$ containing $\mathbb{P}$, $r$, and $G$ does not think there is a counterexample. Hence, the overall statement is $\Sigma_3$.
\end{proof}

Thus in statements like $\exists r\hspace{3pt} W_r\models \phi(a)$ or $\forall r\hspace{3pt} W_r\models \phi(a)$, specifying that $r$ actually defines a ground only increases the complexity of the formula beyond what would be obtained from adding an extra quantifier to the formula produced in Lemma \ref{lemma:groundrel} if $\phi$ is at most $\Pi_2$ or $\Sigma_2$.

The following fact predates the explicit study of set-theoretic geology, but is frequently useful in geological arguments:

\begin{lemma}
\label{lemma:intermodel}
    (see e.g. Jech Lemma 15.43 \cite{jech}) \\ (Intermediate Model Lemma) If $V[G]$ is a generic extension by some poset $\mathbb{P}$ and $V\subseteq W\subseteq V[G]$ for some transitive $W\models ZFC$, then $W$ is a generic extension of $V$ by some complete subalgebra $\mathbb{B}$ of the Boolean completion of $\mathbb{P}$, with $W=V[G\cap\mathbb{B}]$, and $W$ is a ground of $V[G]$.
\end{lemma}

With the groundwork laid, we now survey some basic concepts of set-theoretic geology. Hamkins and Reitz used the first-order definability of grounds to formulate the following axioms:

\begin{definition}
    The Ground Axiom is the statement that there are no nontrivial set forcing grounds.
\end{definition}

\begin{definition}
    The Bedrock Axiom is the statement that there is a minimal ground, or in other words a ground which satisfies the Ground Axiom.
\end{definition}

Reitz \cite{ReitzGA} showed that both axioms can hold or fail independently of a wide range of other natural statements.

Fuchs, Hamkins, and Reitz introduced the following basic object of study:

\begin{definition}
    The mantle $\mathbb{M}$ of a model of set theory is the intersection of all its grounds.
\end{definition}

It follows from the above results that the mantle is a $\Pi_3$ definable class (without parameters).

Many basic questions about the mantle were initially open, but were settled six years later by the work of Toshimichi Usuba:

\begin{theorem}
    \label{thm:strongDDG}
    (Usuba \cite{UsubaDDG}, Proposition 5.1) \\ (Strong DDG Theorem) The grounds are strongly downward directed: that is, for every set $X$ of parameters which succeed in defining grounds of the universe, there is a parameter $s$ which defines a ground $W_s$ such that for all $r\in X$, $W_s$ is a ground of $W_r$.
\end{theorem}

\begin{corollary}
\label{cor:ddgapp}
    (Usuba \cite{UsubaDDG}, Corollary 5.5, applying Fuchs, Hamkins, and Reitz, \cite{FHRgeology}, Theorem 22)
    \begin{enumerate}
        \item $\mathbb{M}$ is a transitive model of $ZFC$
        \item $\mathbb{M}$ is a forcing-invariant class
        \item The following are equivalent:
        \begin{enumerate}
            \item $V$ has only set many grounds
            \item $V$ satisfies the Bedrock Axiom
            \item $V$ is a set forcing extension of the mantle
        \end{enumerate}
    \end{enumerate}
\end{corollary}

\begin{definition}
    (Usuba) For $\kappa$ a cardinal, a $\kappa$-ground is a ground $W$ of which $V$ is a forcing extension by a poset of size less than $\kappa$. The $\kappa$-mantle is the intersection of all $\kappa$-grounds.
\end{definition}

\begin{theorem}
\label{thm:extmantle}
   (Usuba \cite{Usubaextendible}, Theorem 1.3) If $\kappa$ is extendible, then the mantle is equal to the $\kappa$-mantle. It thus follows from the Strong DDG Theorem that if there is an extendible cardinal, then $V$ is a set forcing extension of $\mathbb{M}$. 
\end{theorem}

Note however that the $\kappa$-mantle is not necessarily a $\kappa$-ground, since the proof of the Strong DDG Theorem only tells us in that case that the universe is a $\kappa^{++}$-cc extension of some ground of all the $\kappa$-grounds.

\begin{definition}
    The generic multiverse of a model of set theory is the collection of all models which can be obtained from it by repeatedly taking set forcing extensions and grounds.
\end{definition}

\begin{prop}
\label{prop:genmulti}
    \begin{enumerate}
        \item (Fuchs, Hamkins, and Reitz, conditional on DDG) The generic multiverse can alternatively be defined as the collection of all set forcing extensions of grounds.
    \item (Fuchs, Hamkins, and Reitz, conditional on DDG) All models in the generic multiverse share the same mantle.
    \item (Usuba) If any model in the generic multiverse has an extendible cardinal, then the generic multiverse can be defined as the collection of all set forcing extensions of the mantle.
    \end{enumerate}
\end{prop}
\begin{proof}
    (1): Let $W$ be a ground of $V$ and $W[G]$ be a forcing extension of it. We show that all forcing extensions and all grounds of $W[G]$ are in fact forcing extensions of grounds of $V$, so the collection of forcing extensions of grounds satisfies the closure property to be the complete generic multiverse.

    If $W[G][H]$ is a forcing extension of $W[G]$, then it is a forcing extension of $W$ by the generic filter $G*H$ on some two-step iteration of posets. If $U$ is a ground of $W[G]$, then because grounds are downward directed, there is a ground $W'$ of both $W$ and $U$, so $U$ is a forcing extension of the ground $W'$ of $V$, as desired.

    (2): This follows inductively or by (1) from Usuba's result that the mantle is invariant under forcing.

    (3): If any $W$ in the generic multiverse has an extendible cardinal, then it is a set forcing extension of the mantle, which by (2) is the mantle of the entire generic multiverse, and by downward directedness the mantle is thus a ground of every model in the generic multiverse.
\end{proof}

\chapter[\texorpdfstring{Large Cardinals with $\Sigma_n$ Reflection Properties}{Large Cardinals with Sigma\_n Reflection Properties}]{Large Cardinals with $\Sigma_n$ Reflection Properties}

In this chapter, we study cardinals which reflect the truth of $\Sigma_n$ formulas true in $V$. In the same way that the $\Sigma_2$ reflection properties of supercompact cardinals allow us to establish the consistency of forcing axioms involving $\Sigma_2$-definable forcing classes and $\Sigma_2$ formulas like "there exists a filter which interprets a particular name as a stationary set", these cardinals will allow us to prove the consistency of corresponding axioms for more complex classes and formulas.

\section[\texorpdfstring{$\Sigma_n$-correct Cardinals}{Sigma\_n-correct Cardinals}]{$\Sigma_n$-correct Cardinals}
\label{section:Cn}

We start with a simple and well-studied form of reflection: agreement with $V$ about formulas with parameters in $V_\kappa$. Much of the material in this section comes from Section 1 of Joan Bagaria's paper on the subject and related large cardinal notions (\cite{bagariacn}).

\begin{definition}
    The class of $\Sigma_n$-correct cardinals $C^{(n)}$ consists of all $\kappa$ such that $V_\kappa\prec_{\Sigma_n} V$, i.e. for all $\Sigma_n$ formulas $\phi$ and all $a\in V_\kappa$, $V_\kappa\models \phi(a)$ iff $V\models \phi(a)$.
\end{definition}

In defining and in proving results about $\Sigma_n$-correct cardinals, we make heavy use of the well-known fact that, although by Tarski the truth of formulas in $V$ is not uniformly definable, there is a $\Sigma_n$ definition of $\Sigma_n$ truth:

\begin{lemma}
\label{lemma:sntruth}
    (Folklore) For each standard positive integer $n$, there is a $\Sigma_n$ formula $T_{\Sigma_n}$ in the language of set theory such that, for any $\Sigma_n$ formula $\phi$ with G\"odel number $\lceil \phi\rceil\in V$ and parameter $a$,

    $$T_{\Sigma_n}(\lceil \phi\rceil, a)\Longleftrightarrow \phi(a).$$

    Furthermore, there is a $\Pi_n$ formula $T_{\Pi_n}$ which works the same way with $\Pi_n$ formulas.
\end{lemma}
\begin{proof}
    We proceed by induction on $n$. For $n=1$, we use the facts that $\Delta_0$ formulas are absolute to transitive structures and that the satisfaction of a formula in a set-sized structure is a $\Delta_1$ relation between the structure, the G\"odel number of the formula, and the parameters of the formula. Thus given a $\Delta_0$ formula $\psi(x,y)$ we can express $T_{\Sigma_1}(\lceil\exists x \psi\rceil, a)$ as "there exists an $x$ and a transitive structure $S$ such that $a, x\in S$ and $S\models \psi(x, a)$" and $T_{\Pi_1}(\lceil\forall x \psi\rceil, a)$ as "for all $x$ and all transitive structures $S$ such that $a, x\in S$, $S\models \psi(a, x)$".

    Now if $T_{\Sigma_n}$ and $T_{\Pi_n}$ have already been defined and $\chi(x, y)$ is a $\Pi_n$ formula, we can define $T_{\Sigma_{n+1}}(\lceil\exists x \chi\rceil, a)$ as $\exists x T_{\Pi_n}(\lceil \chi\rceil, (x,a))$. If $\chi$ is instead $\Sigma_n$, we can define $T_{\Pi_{n+1}}$ in a corresponding way. 
\end{proof}

\begin{corollary}
\label{cor:Cndef}
   (Bagaria) $C^{(n)}$ is a $\Pi_n$ definable club class.
\end{corollary}
\begin{proof}
    That $C^{(n)}$ is a proper class follows immediately from the Reflection Theorem, applied to the formula $T_{\Sigma_n}$; closure follows from standard model-theoretic arguments about unions of chains of elementary substructures. For $\Pi_n$ definability, we assume that $n\geq 2$; the case where $n=1$ will follow from Proposition \ref{prop:c1bfp} below. We show that $\alpha\in C^{(n)}$ is equivalent to "for all $x$, all $\Pi_n$ formulas $\phi$, and all $b\in x$, if $x=V_\alpha$ and $x\models \phi(b)$, then $T_{\Pi_n}(\lceil\phi\rceil,b )$". This is sufficient to establish the assertion is $\Pi_n$ because $x=V_\alpha$ is $\Pi_1$ in $x$ and $\alpha$, so after converting the implication to a disjunction the formula has the form $\forall x\forall \phi\forall b(\Sigma_1\vee \Pi_n)$, which is $\Pi_n$ as long as $n\geq 2$.

    To see that this formula actually implies that $\alpha\in C^{(n)}$, note that if a $\Sigma_n$ formula fails in $V_\alpha$, its $\Pi_n$ negation holds in $V_\alpha$ and thus in $V$. Conversely, if we have a $\Sigma_n$ formula of the form $\exists x\hspace{2pt} \psi(x, b)$ for $\psi$ $\Pi_{n-1}$ which holds in $V_\alpha$, let $a\in V_\alpha$ be such that $V_\alpha\models \psi(a,b)$; then $V\models \psi(a,b)$ because $\psi$ is at most $\Pi_n$, so $V\models\exists x\hspace{2pt} \psi(x,b)$.
\end{proof}

$C^{(1)}$ turns out to have a nice characterization in terms of cardinal arithmetic.

\begin{prop}
\label{prop:c1bfp}
    $C^{(1)}$ is exactly the class of $\kappa$ such that $\beth_\kappa=\kappa$.
\end{prop}
\begin{proof}
    The axioms of infinity and pairing are $\Sigma_1$ expressible, using the sets to be paired as parameters in the latter case, so if $\kappa\in C^{(1)}$, then $\kappa>\omega$ is a limit ordinal. For any set $x$, both "there exists a bijection $f:\alpha\rightarrow x$" (where $\alpha$ is an ordinal parameter) and "there exists an ordinal $\alpha$ with a bijection $f:\alpha\rightarrow x$" are $\Sigma_1$, so $V_\kappa$ correctly identifies an ordinal with the same cardinality as each of its elements. It follows that $\kappa$ is larger than many uncountable cardinals and certainly larger than $\omega^2$, so $\omega+\alpha<\kappa$ whenever $\alpha<\kappa$. An easy induction shows that $|V_{\omega+\alpha}|=\beth_\alpha$ for any ordinal $\alpha$. Then because $V_{\omega+\alpha}\in V_\kappa$ for all $\alpha<\kappa$, it follows that $\beth_\alpha\in V_\kappa$ and thus $\beth_\alpha<\kappa$ for all such $\alpha$. Therefore we must have $\beth_\kappa=\kappa$.

Conversely, if $\beth_\kappa=\kappa$, then for any $x\in V_\kappa$, there is a $\beta<\kappa$ such that $trcl(x)\subset V_{\omega+\beta}$. It follows that $|trcl(x)|\leq \beth_\beta<\kappa$, so $x\in H_\kappa$. Since $H_\kappa\subseteq V_\kappa$ for all cardinals $\kappa$, this implies that $V_\kappa=H_\kappa$, so $\kappa\in C^{(1)}$ by Lemma \ref{lemma:HS1correct}.
\end{proof}

We can therefore express $\kappa\in C^{(1)}$ as the $\Pi_1$ formula:
$$\forall x\hspace{2pt}\forall f:x\rightarrow\kappa\hspace{2pt}\forall \alpha<\kappa \hspace{2pt}\exists\beta<\kappa\hspace{2pt}\forall y\in x\hspace{2pt}(rank(y)<\alpha\Rightarrow f(y)\neq \beta)$$
(in words, "there is no surjection $x\cap V_\alpha\rightarrow \kappa$ for any set $x$ and any $\alpha<\kappa$"), using the fact that the rank function is $\Delta_1$.

Bagaria proved an alternative characterization of $C^{(1)}$ included in the above proof, that it consists of exactly the uncountable cardinals $\kappa$ such that $V_\kappa=H_\kappa$.

Though by Corollary \ref{cor:Cndef} any model of ZFC will have many $C^{(n)}$ cardinals, the existence of a regular member of $C^{(n)}$ has large cardinal strength.

\begin{corollary}
\label{cor:cnlargecard}
    (Folklore) A regular cardinal is in $C^{(1)}$ if and only if it inaccessible. The axiom scheme that there is a regular $\Sigma_n$-correct cardinal for each $n$ is equivalent to Ord is Mahlo (i.e., that every definable (with parameters) club class of ordinals contains a regular cardinal).
\end{corollary}
\begin{proof}
    The only part that is not immediate from previous results is that the existence of regular $C^{(n)}$ cardinals for all $n$ implies that Ord is Mahlo. For this, we first show that the hypothesis implies that for each $n$, there is a proper class of regular $C^{(n)}$ cardinals. Let $\kappa\in C^{(n+2)}$ be regular. Then for any $\alpha<\kappa$, $V$ thinks there is a regular $C^{(n)}$ cardinal above it (namely $\kappa$), so $V_\kappa$ thinks there is a regular $C^{(n)}$ cardinal above $\alpha$. Thus $V_\kappa$ satisfies the $\Pi_{n+2}$ sentence "for every ordinal $\alpha$, there is a regular $C^{(n)}$ cardinal above $\alpha$", so $V$ believes this as well, as desired.
    
    Now let $\phi$ be a $\Sigma_n$ formula defining a club class and $\kappa$ a regular $C^{(n)}$ cardinal large enough for $V_\kappa$ to contain all the parameters used in $\phi$. Then for any $\beta<\kappa$, $V$ thinks "there is some $\alpha>\beta$ such that $\phi(\alpha)$", which is $\Sigma_n$. Thus there are unboundedly many ordinals $\alpha$ in $V_\kappa$ such that $\phi(\alpha)$ holds (in $V_\kappa$ and thus also in $V$), so because $\phi$ defines a closed class, $\phi(\kappa)$ holds. Thus $\kappa$ is a regular cardinal in the class defined by $\phi$, as desired.
\end{proof}

The regular $C^{(2)}$ cardinals are reasonably well-known: they are exactly the $\Sigma_1$-reflecting cardinals\footnote{Some authors use $\Sigma_n$-reflecting to mean regular $\Sigma_n$-correct. However, since this conflicts with Goldstern and Shelah's usage and the word "reflecting" is somewhat overloaded in set theory, I will avoid this in favor of more descriptive terminology.} introduced by Goldstern and Shelah \cite{GSbpfa} to prove the consistency of BPFA, and later used by Aspero and Bagaria to establish the consistency of symmetric bounded forcing axioms in general (\cite{ABbfacont}, Lemma 2.2).

\begin{prop}
\label{prop:c2eqrefl}
    (Folklore) A regular cardinal $\kappa$ is $\Sigma_2$-correct if and only if for all formulas $\phi$, all regular $\theta\geq \kappa$, and all $a\in H_\kappa$ such that $H_\theta\models \phi(a)$, there is a cardinal $\delta<\kappa$ with $a\in H_\delta$ and $H_\delta\models \phi(a)$.
\end{prop}
\begin{proof}
    For the forward direction, since the assertions that an ordinal is regular and that a set is equal to $H_\theta$ are $\Pi_1$, the statement that there is a regular $\theta$ such that $a\in H_\theta$ and $H_\theta\models \phi(a)$ is $\Sigma_2$. It follows that it is true in $V_\kappa$; let $\delta$ be the cardinal witnessing it.

    For the reverse direction, if $V\models \exists x\psi(a, x)$ where $\psi$ is $\Pi_1$, let $b$ be such that $V \models \psi(a,b)$. Then by Lemma \ref{lemma:HS1correct}, $H_\theta\models \psi(a,b)$ for any regular $\theta$ large enough for $\theta$ to contain $a$ and $b$, and hence $H_\theta\models \exists x\psi(a, x)$. Then applying the reflection property, there is some $\delta<\kappa$ such that $H_\delta\models\exists x\psi(a,x)$. Since by Corollary \ref{cor:cnlargecard} $\kappa$ is inaccessible, the proof of Lemma \ref{lemma:HS1correct} goes through in $V_\kappa$, so $V_\kappa$ agrees with $H_\delta$ about the truth of $\psi$, and thus whichever witness for the statement $H_\delta$ has also works in $V_\kappa$. Hence $V_\kappa\models\exists x \psi(a, x)$, so $\kappa\in C^{(2)}$.
\end{proof}

We can straightforwardly generalize this result to $n>2$ if we require that $\theta$ and $\delta$ are in $C^{(n-1)}$.

The following generalizes a fact about $H_\theta$ that we established in the preceding proof:

\begin{lemma}
    \label{lemma:Cnaddexist}
    (Bagaria) A $\Sigma_{n+1}$ formula holds in $V$ if and only if there is an $\alpha\in C^{(n)}$ such that it holds in $V_\alpha$. Inversely, a $\Pi_{n+1}$ formula holds in $V$ if and only if it holds in $V_\alpha$ for all $\alpha\in C^{(n)}$ such that $V_\alpha$ contains the parameters.
\end{lemma}
\begin{proof}
    Let our formula be $\exists y\psi(a, y)$ where $\psi$ is $\Pi_n$ and $a$ is any parameter. If $V_\alpha\models \exists y\psi(a, y)$ for some $\alpha\in C^{(n)}$ such that $a\in V_\alpha$, then if $b\in V_\alpha$ witnesses this, $V\models \psi(a, b)$. It follows that $V\models \exists y \psi(a, y)$.

Conversely, if $b$ is such that $V\models \psi(a, b)$, let $\alpha\in C^{(n)}$ be big enough that $a, b\in V_\alpha$. Then $V_\alpha\models \psi(a, b)$, so $V_\alpha\models \exists y \psi(a, y)$.
\end{proof}

It is obvious from the transitivity of agreement on $\Sigma_n$ truth that $\Sigma_n$-correct cardinals recognize smaller $\Sigma_n$-correct cardinals. Using the previous lemma, we can in fact conclude something slightly stronger:

\begin{lemma}
\label{lemma:recogCn+1}
    If $n\leq m$, $\alpha<\beta$, $\alpha\in C^{(n+1)}$, and $\beta\in C^{(m)}$, then $V_\beta\models\alpha\in C^{(n+1)}$ (or in other words, $V_\alpha\prec_{\Sigma_{n+1}} V_\beta$).
\end{lemma}
\begin{proof}
    By Corollary \ref{cor:Cndef}, the assertion that $\alpha$ is $\Sigma_{n+1}$-correct is at most $\Pi_{m+1}$, so by Lemma \ref{lemma:Cnaddexist}, it holds in $V_\beta$.
\end{proof}

Thus if a $\Sigma_n$-correct cardinal $\beta$ has a $\Sigma_{n+1}$-correct cardinal $\alpha$ below it, $V_\beta$ will correctly compute the $C^{(n+1)}$ up to and including $\alpha$ (because $V_\alpha$ correctly computes it and $V_\beta$ agrees with $V_\alpha$ on $\Pi_{n+1}$ truth). However, there may be cardinals above $\alpha$ which are $\Sigma_{n+1}$-correct in $V_\beta$ but not in $V$.

Finally, we examine how $\Sigma_n$-correct cardinals interact with small forcing:

\begin{lemma}
    \label{lemma:Cnforce}
    (Folklore) If $\mathbb{P}\in V_\theta$ is a forcing poset and $\theta\in C^{(n)}$, then $\Vdash_\mathbb{P} \theta\in C^{(n)}$.
\end{lemma}
\begin{proof}
    Let $G\subseteq\mathbb{P}$ be $V$-generic, $\phi$ be a $\Sigma_n$ (or $\Pi_n$) formula, and $a\in V_\theta^{V[G]}$ such that $V[G]\models\phi(a)$. By Lemma \ref{lemma:namesize} there is some $\dot{a}\in V_\theta$ such that $a=\dot{a}^G$. Then if $p\in G$ forces $\phi(\dot{a})$, $V_\theta\models p\Vdash_{\mathbb{P}}\phi(\dot{a})$ because the forcing relation for $\Sigma_n$ ($\Pi_n$) formulas is $\Sigma_n$ ($\Pi_n$). Thus $V_\theta[G]=V_\theta^{V[G]}\models \phi(a)$, where the equality follows from Lemma \ref{lemma:namesize} as well.
\end{proof}

We also have the geological converse:

\begin{lemma}
    \label{lemma:Cnground}
    If $\mathbb{P}\in V_\theta$ is a forcing poset, $G\subseteq\mathbb{P}$ is a $V$-generic filter, and $\theta\in (C^{(n)})^{V[G]}$, then $\theta\in (C^{(n)})^V$.
\end{lemma}
\begin{proof}
    If $n=1$, this follows because $C^{(1)}$ is $\Pi_1$-definable and $\Pi_1$ formulas are downward absolute between transitive structures. For $n\geq 2$, let $r=\mathcal{P}(\mathbb{P})^V$ and let $\phi(a)$ be a $\Sigma_{n}$ or $\Pi_n$ formula which holds in $V$ for some $a\in V_\theta$. By Lemma \ref{lemma:groundrel}, $V\models\phi(a)$ is $\Sigma_n$ or $\Pi_n$ expressible in $V[G]$ with $a$ and $r$ as parameters. Since $\theta$ is a limit ordinal, $r\in V_\theta$, so $V_\theta^{V[G]}$ agrees with $V[G]$ that the ground defined by $r$ satisfies $\phi(a)$. Since $\theta$ is at least $\Sigma_2$-correct in $V[G]$, by the Ground Definability Theorem \ref{thm:grounddef} $V_\theta^{V[G]}$ computes the ground defined by $r$ as $V_\theta$. It thus follows that $V_\theta\models \phi(a)$.
\end{proof}

\section[\texorpdfstring{Supercompactness for $C^{(n)}$}{Supercompactness for C\^(n)}]{Supercompactness for $C^{(n)}$}
\label{section:scCn}

We study a strengthening of supercompactness where the elementary embeddings are also required to preserve initial segments of a class, analogous to the notion of strongness for a class $A$ used in some definitions of Woodin cardinals.

\begin{definition}
    \begin{enumerate}
        \item A cardinal $\kappa$ is $\nu$-supercompact for a class $A$ iff there is an inner model $M$ with an elementary embedding $j:V\rightarrow M$ such that:
        \begin{itemize}
            \item $crit(j)=\kappa$
            \item $^\nu M\subset M$
            \item $j(\kappa)>\nu$
            \item $j(A\cap V_\kappa)\cap V_\nu=A\cap V_\nu$
        \end{itemize}
        \item $\kappa$ is supercompact for $A$ if it is $\nu$-supercompact for $A$ for all cardinals $\nu$
        \item A cardinal $\delta$ is Woodin for supercompactness if for all $A\subseteq V_\delta$, there is some $\kappa<\delta$ such that $(V_\delta, \in, A)\models ``\kappa$ is supercompact for $A"$
    \end{enumerate}
\end{definition}

Kanamori's \cite{kanamori} Exercise 24.19 shows that the Woodin for supercompactness cardinals are exactly the Vopenka cardinals; Kentaro Sato (\cite{SatoDH}, Corollary 10.6) and Norman Perlmutter (\cite{Perlmutter}, Theorem 5.10) independently found alternative proofs. Going forward I will not make much use of Vopenka cardinals, but working in $V_\delta$ for $\delta$ Vopenka rather than assuming some version of Vopenka's principle in $V$ can simplify the metamathematical details, so I include them in the above definition for completeness.

The following standard result (see e.g. Kanamori \cite{kanamori} Lemma 22.12 and surrounding discussion, Jech \cite{jech} pg. 377 between Corollary 20.18 and Lemma 20.19) is very useful for analyzing supercompactness for $A$.

\begin{lemma}
    \label{lemma:scfactor}
    If $\kappa\leq \nu$ is a cardinal and $j:V\rightarrow M$ witnesses that $\kappa$ is $\nu$-supercompact, let $U:=\{X\subseteq \mathcal{P}_\kappa(\nu)\sbp j"\nu\in j(X)\}$ be the normal ultrafilter derived from $j$ and $k: Ult(V, U)\rightarrow M$ be defined by $k([f]_U)=j(f)(j"\nu)$. Then $k$ is an elementary embedding with critical point above $\nu$ and $j=k\circ j_U$.
\end{lemma}

Thus, supercompactness for $A$ has a straightforward combinatorial characterization:

\begin{corollary}
    If $\kappa$ is $\nu$-supercompact for a class $A$, there is a normal ultrafilter $U\subset \mathcal{P}(\mathcal{P}_\kappa(\nu))$ witnessing that fact.
\end{corollary}
\begin{proof}
    Let $U$ and $k$ be as in the preceding lemma. Then $j_U$ is a $\nu$-supercompactness embedding, and we have:
    \begin{align*}
        A\cap V_\nu&=j(A\cap V_\kappa)\cap V_\nu\\
        &=k(j_U(A\cap V_\kappa))\cap V_\nu\\
        &=j_U(A\cap V_\kappa)\cap V_\nu
    \end{align*}
    where the final equality holds because $k$ does not move elements of $V_\nu$.
\end{proof}

Laver functions are very useful for many applications of supercompact cardinals, including Baumgartner-style consistency proofs of forcing axioms. They generalize easily to supercompactness for $A$:

\begin{lemma}
If $\kappa$ is supercompact for some class $A$, there exists a Laver function $f$ on $\kappa$ for $A$. That is, $f:\kappa\rightarrow V_\kappa$ is such that for all sets $x$ and all ordinals $\lambda\geq |trcl(x)|$, there is an elementary embedding $j:V\rightarrow M$ witnessing that $\kappa$ is $\lambda$-supercompact for $A$ such that $j(f)(\kappa)=x$.
\label{lemma:laver}
\end{lemma}
\begin{proof}
We carry out the standard proof of the existence of a Laver function (e.g. Jech \cite{jech} Theorem 20.21) and verify that the embeddings we consider can be taken to preserve $A$. We assume toward a contradiction that for all $f:\kappa\rightarrow V_\kappa$ there is some set $x$ and some $\lambda\geq |trcl(x)|$ such that for all elementary embeddings $j$ witnessing that $\kappa$ is $\lambda$-supercompact for $A$, $j(f)(\kappa)\neq x$; let $\lambda_f$ be the least $\lambda$ for which such an $x$ exists. Choose some $\nu$ greater than any possible value of $2^{\lambda_f^{<\kappa}}$ for $f:\kappa\rightarrow V_\kappa$ and let $j: V\rightarrow M$ witness that $\kappa$ is $\nu$-supercompact for $A$.

Let $\phi(g, \beta)$ be the statement that $g$ is a function from some cardinal $\alpha$ to $V_\alpha$ and $\beta\geq \alpha$ is minimal such that for some $x$ with $|trcl(x)|\leq\beta$, there is no normal measure $U\subset \mathcal{P}(\mathcal{P}_\alpha(\beta))$ such that $j_U(A\cap V_\alpha)\cap V_\beta=A\cap V_\beta$ and $j_U(f)(\alpha)=x$. Then for all $f:\kappa\rightarrow V_\kappa$, $M\models \phi(f, \lambda_f)$ since $M$ was chosen so as to contain everything relevant to the truth of $\phi$.

For any suitable function $g$, $\lambda_g$ denotes the unique cardinal such that $\phi(g,\lambda_g)$ holds if one exists. Let $B=\{\alpha<\kappa\hspace{2pt} | \hspace{2pt} \forall g:\alpha\rightarrow V_\alpha \hspace{3pt} \phi(g,\lambda_g)\}$. Then $j(B)$ is the set of all $\alpha<j(\kappa)$ satisfying the same property, except that in the statement of $\phi$ we now say there is no normal measure $U$ such that $j_U(g)(\alpha)=x$ and $j_U(j(A\cap V_\alpha))\cap V_{\lambda_g}=j(A\cap V_\alpha)\cap V_{\lambda_g}$. If $crit(j_U)=\kappa$, then since $j(A\cap V_\kappa)\cap V_\nu=A\cap V_\nu$, $j_U(j(A\cap V_\kappa))\cap V_{\lambda_g}=j_U(j(A\cap V_\kappa)\cap V_\kappa)\cap V_{\lambda_g}=j_U(A\cap V_\kappa)\cap V_{\lambda_g}$ and, since $\nu>\lambda_g$ for all $g$, $j(A\cap V_\kappa)\cap V_{\lambda_g}=A\cap V_{\lambda_g}$. Hence for any function $g:\kappa\rightarrow V_\kappa$ and normal measure $U\subset \mathcal{P}(\mathcal{P}_\kappa(\lambda_g))$, $j_U(j(A\cap V_\kappa))\cap V_{\lambda_g}=j(A\cap V_\kappa)\cap V_{\lambda_g}$ iff $j_U(A\cap V_\kappa)\cap V_{\lambda_g}=A\cap V_{\lambda_g}$, so $\kappa\in j(B)$ by our hypothesis about the nonexistence of Laver functions.

Now we inductively define $f:\kappa\rightarrow V_\kappa$ so that if $\alpha\in B$ then $f(\alpha)$ witnesses $\phi(f\upharpoonright \alpha, \lambda_{f\upharpoonright \alpha})$, and if $\alpha\not\in B$, $f(\alpha)=\emptyset$. Let $x=j(f)(\kappa)$. $j(f)$ is constructed in the same way as $f$, except with $j(A)$ used in place of $A$ in $\phi$, but as shown in the previous paragraph this doesn't matter for functions with domain $\kappa$ like $j(f)\upharpoonright \kappa=f$, so $x$ witnesses the truth of $\phi(f,\lambda_f)$ in $M$. By the choice of $\nu$, $M$ and $V$ have exactly the same ultrafilters on $\mathcal{P}_\kappa(\lambda_f)$ and compute their ultrapowers identically, so $x$ also witnesses the truth of $\phi(f, \lambda_f)$ in $V$.

Now we let $U\subset \mathcal{P}(\mathcal{P}_\kappa(\lambda_f))$ be the normal measure derived from $j$ and obtain a contradiction by showing that $j_U(A\cap V_\kappa)\cap V_{\lambda_f}=A\cap V_{\lambda_f}$ and $j_U(f)(\kappa)=x$. Let $k: Ult(V, U)\rightarrow V$ be the elementary embedding such that $j=k\circ j_U$. Then $k\upharpoonright V_{\lambda_f}=id$, so for any set $y$, $y\in j_U(A\cap V_\kappa)\cap V_{\lambda_f}$ iff $k(y)\in k(j_U(A\cap V_\kappa))\cap V_{\lambda_f}$ iff $y\in j(A\cap V_\kappa)\cap V_{\lambda_f}$ iff $y\in A\cap V_{\lambda_f}$. Furthermore, $x\in H_{\lambda_f^+}$ and $k(\lambda_f)=\lambda_f$, so it can be shown inductively that $k(x)=x$. However, we also have that
$$k(j_U(f)(\kappa))=k(j_U(f))(k(\kappa))=j(f)(\kappa)=x$$
so by the injectivity of $k$, $j_U(f)(\kappa)=x$. This contradicts the fact that $x$ is a counterexample to $f$ being a Laver function for $A$, completing the proof.
\end{proof}

Most of interest to us is the case where $A=C^{(n)}$ for some $n$. (These are not to be confused with the $C^{(n)}$-supercompact cardinals Bagaria considered in \cite{bagariacn}.) In this case we get the following alternative characterization:

\begin{lemma}
\label{lemma:scfCnalt}
   The following are equivalent for all cardinals $\kappa$ and all positive integers $n$:
   \begin{enumerate}
       \item $\kappa$ is supercompact for $C^{(n)}$
       \item For every cardinal $\nu$, $\Sigma_{n+1}$ formula $\phi$, and set $a$ such that $V\models \phi(a)$, there is an elementary embedding $j:V\rightarrow M$ such that:
       \begin{itemize}
            \item $crit(j)=\kappa$
            \item $^\nu M\subset M$
            \item $j(\kappa)>\nu$
            \item $M\models\phi(a)$
        \end{itemize}
   \end{enumerate}
\end{lemma}
\begin{proof}
    $(1\Rightarrow 2):$ Given $\nu$, $\phi$, and $a$, let $\theta\geq \nu$ be a $C^{(n)}$ cardinal such that $V_\theta\models \phi(a)$ (using Lemma \ref{lemma:Cnaddexist}). Then if $\lambda>\theta=|V_\theta|$ and $j:V\rightarrow M$ witnesses that $\kappa$ is $\lambda$-supercompact for $C^{(n)}$, the first three bullet points in (2) are immediate, so we show the fourth. First, since there is a $\lambda$-sequence in $V$ with range $V_\theta$, the closure conditions on $M$ imply that $V_\theta^M=V_\theta$. By elementarity, $j(C^{(n)}\cap V_\kappa)$ is exactly the $C^{(n)}$ cardinals of $M$ below $j(\kappa)$, so $\theta\in C^{(n)}\cap V_\lambda=j(C^{(n)}\cap V_\kappa)\cap V_\lambda= (C^{(n)})^M\cap\lambda$. Since $V_\theta^M=V_\theta\models \phi(a)$, applying Lemma \ref{lemma:Cnaddexist} in $M$ yields $M\models\phi(a)$.

    $(2\Rightarrow 1):$ Given $\nu$, let $a=C^{(n)}\cap\nu$. Then $V\models\forall\theta\hspace{-3pt}<\hspace{-3pt}\nu \hspace{6pt}\theta\in a\longleftrightarrow \theta\in C^{(n)}$, where $\theta\in C^{(n)}$ abbreviates the formula from Corollary \ref{cor:Cndef}. This is a conjunction of a $\Pi_n$ formula and a $\Sigma_n$ formula with an added bounded quantifier, so it can be written in $\Sigma_{n+1}$ form. Let $j:V\rightarrow M$ be a $\nu$-supercompactness embedding with critical point $\kappa$ and $M\models \forall\theta<\nu \hspace{6pt}\theta\in a\longleftrightarrow \theta\in C^{(n)}$. Then $j(C^{(n)}\cap V_\kappa)\cap V_\nu=(C^{(n)})^M\cap\nu=a=C^{(n)}\cap\nu$, so $j$ witnesses that $\kappa$ is $\nu$-supercompact for $C^{(n)}$.
\end{proof}

Thus, for example if $\phi$ is a $\Pi_2$ large cardinal property like strongness or supercompactness, any given cardinal with that property can retain it in the codomain of a supercompactness for $C^{(2)}$ embedding. Similar arguments show that any cardinal supercompact for $C^{(n)}$ has many cardinals supercompact for $C^{(n-1)}$ below it.

By examining the proofs of Lemma \ref{lemma:laver} and Lemma \ref{lemma:scfCnalt}, it is easy to see that if $f$ is a Laver function for $\kappa$, we can simultaneously ensure that $M\models\phi(a)$ and $j(f)(\kappa)=b$ for any desired $a$ and $b$ such that $V\models \phi(a)$.

Clarifying the bottom of our hierarchy, analogously to the first sentence of Corollary \ref{cor:cnlargecard}:

\begin{lemma}
\label{lemma:scfC1}
    Every supercompact cardinal is supercompact for $C^{(1)}$.
\end{lemma}
\begin{proof}
    Since $j(C^{(1)}\cap V_\kappa)\cap V_\nu$ is simply the set of $C^{(1)}$ cardinals of $M$ below $\nu$, by Proposition \ref{prop:c1bfp}, it is sufficient to show that for every beth fixed point $\nu$, if $M$ is the codomain of any $\nu$-supercompactness embedding, then $M$ correctly computes the beth function below $\nu$. If $\alpha<\nu$, then $|V_{\omega+\alpha}|=\beth_\alpha<\nu$, so $^\nu M\subset M$ implies that $V_{\omega+\alpha}^M=V_{\omega+\alpha}$ and that $M$ correctly computes the cardinality of this set. Thus $\beth_\alpha^M=|V_{\omega+\alpha}|^M=\beth_\alpha$ as desired, so $\alpha$ is $C^{(1)}$ in $M$ if and only if it is $C^{(1)}$ in $V$.
\end{proof}

The following lemmas generalize well-known facts about supercompact cardinals:

\begin{lemma}
\label{lemma:scCn}
    If $\kappa$ is supercompact for $C^{(n)}$ (or merely strong for $C^{(n)}$), then $\kappa\in C^{(n+1)}$
\end{lemma}
\begin{proof}
    First note that $\kappa$ is a limit point of $C^{(n)}$, since bounded subsets of $\kappa$ are fixed by all elementary embeddings with critical point $\kappa$, whereas $j(C^{(n)}\cap\kappa)$ can be an arbitrarily large initial segment of a proper class. Since $C^{(n)}$ is closed, $\kappa\in C^{(n)}$. Thus by Lemma \ref{lemma:Cnaddexist} any $\Sigma_{n+1}$ formula true in $V_\kappa$ must also be true in $V$.

    If $\phi$ is a $\Sigma_{n+1}$ formula and $a\in V_\kappa$ is such that $V\models\phi(a)$, again by Lemma \ref{lemma:Cnaddexist} there is some $\theta\in C^{(n)}$ such that $V_\theta\models \phi(a)$. Then there is an elementary embedding $j:V\rightarrow M$ with $j(\kappa)>\theta$, $V_\theta^M=V_\theta$, and $M\models \theta\in C^{(n)}$. Thus $M$ thinks there is a $C^{(n)}$ cardinal less than $j(\kappa)$ witnessing $\phi(j(a))$ ($j(a)=a$ because $a\in V_\kappa$), so by elementarity there is in $V$ a $\bar{\theta}<\kappa$ in $C^{(n)}$ with $V_{\bar{\theta}}\models\phi(a)$. By Lemma \ref{lemma:recogCn+1}, $V_\kappa\models \bar{\theta}\in C^{(n)}$, and since $V_\kappa\models ZFC$, we can apply Lemma \ref{lemma:Cnaddexist} again inside of it to obtain $V_\kappa\models \phi(a)$, as desired.
\end{proof}

\begin{lemma}
\label{lemma:scfCntrunc}
    If $\kappa<\lambda$, $\kappa$ is $\nu$-supercompact for $C^{(n)}$ for all $\nu<\lambda$, and $\lambda\in C^{(n)}$ is regular, then $V_\lambda\models \text{"}\kappa$ is supercompact for $C^{(n)}$".
\end{lemma}
\begin{proof}
    Since $\lambda$ is inaccessible, $V_\lambda$ contains all the ultrafilters on $\mathcal{P}_\kappa(\nu)$ witnessing that $\kappa$ is supercompact for $C^{(n)}$ up to $\lambda$, as well as all the functions $\mathcal{P}_\kappa(\nu)\rightarrow V_\lambda$ which represent elements of $V_{j_U(\lambda)}^{Ult(V, U)}$ for any such normal ultrafilter $U$, so we can construct the ultrapower within $V_\lambda$ and get exactly the same results below $\lambda$. Furthermore, all functions $\mathcal{P}_\kappa(\nu)\rightarrow \lambda$ are bounded, so the order type of the predecessors of any given one in any ultrapower is less than $\lambda$, from which it follows that $j_U(\lambda)=\lambda$.

    Hence for each $\nu<\lambda$, we have a $\nu$-supercompactness embedding $V_\lambda\rightarrow V_\lambda^M$ which arises as a restriction of a $\nu$-supercompactness embedding $V\rightarrow M$ with $(C^{(n)})^M\cap \nu=C^{(n)}\cap\nu$. Since $\lambda$ is $C^{(n)}$, it agrees with $V$ on which cardinals below it are $C^{(n)}$ by Corollary \ref{cor:Cndef}, so by elementarity $V_\lambda^M$ agrees with $M$ on which cardinals are $C^{(n)}$ below $\lambda$ as well. Thus $V_\lambda$ and $V_\lambda^M$ agree on $C^{(n)}$ cardinals below $\nu$, as desired.
\end{proof}

We conclude this section with some results on the relationship between supercompactness for $C^{(n)}$ and other large cardinal axioms:

\begin{prop}
\label{prop:scC2eqext}
    $\kappa$ is supercompact for $C^{(2)}$ iff $\kappa$ is extendible.
\end{prop}
\begin{proof}
    For the forward direction, given an ordinal $\eta$, we show that $\kappa$ is $\eta$-extendible. Let $\theta\in C^{(2)}$ be greater than $\kappa+\eta$; then there is an elementary embedding $j:V\rightarrow M$ with $crit(j)=\kappa$, $\theta\in (C^{(2)})^M$, $V_\theta^M=V_\theta$, and $j\upharpoonright V_{\kappa+\eta}\in M$. By the last condition and Lemma \ref{lemma:restrictembed}, $M$ satisfies the $\Sigma_2$ assertion that $\kappa$ is $\eta$-extendible, so by $\Sigma_2$-correctness $V_\theta$ contains an elementary embedding $\sigma:V_{\kappa+\eta}\rightarrow V_\beta$ for some $\beta<\theta$ with $crit(\sigma)=\kappa$. $\sigma$ then witnesses that $\kappa$ is $\eta$-extendible in $V$.

    For the converse, given an ordinal $\lambda>\kappa$, we show that $\kappa$ is $\lambda$-supercompact for $C^{(2)}$. Let $\theta\in C^{(2)}$ be greater than $\lambda$; then by extendibility, for some ordinal $\beta>\theta$ there is an elementary embedding $\sigma: V_\theta\rightarrow V_\beta$ with critical point $\kappa$. By Lemma \ref{prop:c1bfp}, the ordinals below $\theta$ are closed under the beth function, so by elementarity those below $\beta$ are as well, and thus $\beta\in C^{(1)}$. It follows from Lemma \ref{lemma:recogCn+1} that $V_\beta$ recognizes the $\Sigma_2$-correctness of $\theta$, and from the remarks after that lemma that it correctly computes $C^{(2)}\cap\lambda$.
    
    By elementarity $\sigma(C^{(2)}\cap \kappa)$ is the set of cardinals $V_\beta$ thinks are $\Sigma_2$-correct, so $\sigma(C^{(2)}\cap\kappa)\cap \lambda= C^{(2)}\cap\lambda$. By standard arguments, if $U\subset \mathcal{P}(\mathcal{P}_\kappa(\lambda))$ is the ultrafilter derived from $\sigma$, then $U$ is a fine, normal, and $\kappa$-complete. Furthermore, if $j_U$ is the associated ultrapower embedding, there is an elementary embedding $k: Ult(V_\theta, U)\rightarrow V_\beta$ such that $\sigma=k\circ j_U$ and $crit(k)\geq \lambda$, so in particular $j_U(C^{(2)}\cap\kappa)\cap\lambda=\sigma(C^{(2)}\cap \kappa)\cap\lambda=C^{(2)}\cap \lambda$. Thus $U$ witnesses that $\kappa$ is $\lambda$-supercompact for $C^{(2)}$, as desired.
\end{proof}

\begin{prop}
\label{prop:scCneqvopenka}
    The existence of cardinals supercompact for $C^{(n)}$ for each standard $n$ is equivalent to the first-order Vopenka scheme (i.e., that for every definable proper class of structures of the form $(A, \in, R)$ for $A$ a transitive set and $R\subseteq A$, there is an elementary embedding between two distinct elements of the class).
\end{prop}
\begin{proof}
    For the forward direction, let $\phi$ be a $\Sigma_n$ formula defining the class and $\kappa$ a supercompact cardinal for $C^{(n-1)}$. Then if $(A,\in, R)$ is such that $\phi((A,\in, R))$ and $|A|\geq \kappa$, we apply Lemma \ref{lemma:scfCnalt} to obtain an elementary embedding $j:V\rightarrow M$ witnessing the $|A|$-supercompactness of $\kappa$ such that $M\models\phi((A,\in, R))$. Then $M\models\phi((j(A), \in, j(R)))$ by elementarity, and by Lemma \ref{lemma:restrictembed} $j\upharpoonright A\in M$ is an elementary embedding $(A, \in, R) \rightarrow (j(A), \in, j(R))$. Since $\kappa\leq |A|^M<j(\kappa)\leq |j(A)|^M$, $A\neq j(A)$, so $M$ thinks there are distinct structures in the class defined by $\phi$ with an elementary embedding between them. By elementarity, the same holds in $V$.

    For the converse, we generalize Bagaria's proof of Theorem 4.3 in \cite{bagariacn}. Given $n$, let $\mathcal{A}$ consist of all structures of the form $(V_{\gamma}, \in, \{\{\alpha, \lambda\}\}\cup (C^{(n)}\hspace{-1pt}\cap\hspace{-1pt}\gamma))$ where $\lambda$ is the least limit ordinal greater than $\alpha$ such that there are no $\kappa\leq\alpha$ that are $<\hspace{-3pt}\lambda$-supercompact for $C^{(n)}$ and $\gamma$ is the least element of $C^{(n+1)}$ above $\lambda$ with uncountable cofinality. Note that there is at most one structure in $\mathcal{A}$ for any given value of $\alpha$, which we denote $A_\alpha$. We assume toward a contradiction that there are no cardinals supercompact for $C^{(n)}$; then $A_\alpha$ exists for all $\alpha$, so in particular $\mathcal{A}$ is a proper class. Since $\mathcal{A}$ is definable, there are $\alpha,\beta$ such that there is an elementary embedding $j:A_\alpha\rightarrow A_\beta$ for $\alpha\neq\beta$.
    
    Set $\gamma$, $\delta$, $\lambda$, and $\mu$ such that $A_\alpha= (V_{\gamma}, \in, \{\{\alpha, \lambda\}\}\cup (C^{(n)}\hspace{-1pt}\cap\hspace{-1pt}\gamma))$ and $A_\beta= (V_{\delta}, \in, \{\{\beta, \mu\}\}\cup (C^{(n)}\cap\delta))$. Since $\{\alpha, \lambda\}$ and $\{\beta, \mu\}$ are the only two-elements sets included in each unary predicate, $\alpha<\lambda$, and $\beta<\mu$, we must have $j(\alpha)=\beta$. As $\alpha\neq\beta$ and elementary embeddings of transitive structures never map ordinals to smaller ordinals, $\alpha<\beta$ and so $\kappa:=crit(j)\leq \alpha$.

To complete the proof, we show that $\kappa$ is $<\hspace{-3pt}\lambda$-supercompact for $C^{(n)}$, contradicting $A_\alpha\in\mathcal{A}$. The argument is a straightforward generalization of the standard proof that sufficient partial extendibility, even without any hypotheses on the size of $j(\kappa)$, implies partial supercompactness (see e.g. Kanamori \cite{kanamori} Propositions 23.6 and 23.15(b)). We first observe that if $j^i(\kappa)$ is defined for all natural numbers $i$, then the supremum $\tau$ of this sequence must be a fixed point of $j$ strictly between $\kappa$ and $\gamma$ (as $\gamma$ was chosen to have uncountable cofinality); however this would mean that $j$ restricts to a nontrivial elementary embedding $V_{\tau+2}\rightarrow V_{\tau+2}$, contrary to Kunen's inconsistency theorem. Thus finitely iterating $j$ sufficiently must carry us above $\gamma$, so in particular we can find some $m$ such that $j^m(\kappa)\leq\lambda<j^{m+1}(\kappa)$.

Following Kanamori, let $P(i)$ denote the assertion that there is an elementary embedding
$$\sigma: (V_\lambda, \in, \lambda\cap C^{(n)})\rightarrow (V_\theta, \in, \theta\cap C^{(n)})$$
for some $\theta$ with $crit(\sigma)=\kappa$ and $\sigma(\kappa)=j^{i+1}(\kappa)$. $P(i)$ can easily be verified to be a $\Sigma_{n+1}$ statement with $\kappa$, $j^{i+1}(\kappa)$, and $(V_\lambda, \in, \lambda\cap C^{(n)})$ as parameters.

We verify $P(m)$ by induction. $P(0)$ holds with $\theta=\mu$ and $\sigma=j\upharpoonright V_\lambda$. If $P(i)$ holds and $i<m$, then because $\gamma\in C^{(n+1)}$ and all the parameters are in $V_\gamma$, it holds in $V_\gamma$. Thus we can find a $\theta<\gamma$ and $\sigma\in V_\gamma$ witnessing it. Then in $V_\delta$ $j(\sigma)$ is an elementary embedding $(V_{j(\lambda)}, \in, j(\lambda)\cap C^{(n)})\rightarrow (V_{j(\theta)}, \in, j(\theta)\cap C^{(n)})$ with $crit(j(\sigma))=j(\kappa)$ and $j(\sigma)(j(\kappa))=j(j^{i+1}(\kappa))=j^{i+2}(\kappa)$. Thus $j(\sigma)\circ\sigma: (V_\lambda, \in, \lambda\cap C^{(n)})\rightarrow (V_{j(\theta)}, \in, j(\theta)\cap C^{(n)})$ witnesses that $V_\delta\models P(i+1)$, so because $\delta\in C^{(n+1)}$, $P(i+1)$ holds in $V$. It follows that $P(m)$ holds and thus there is an elementary embedding 
$$\bar{j}: (V_\lambda, \in, \lambda\cap C^{(n)})\rightarrow (V_\theta, \in, \theta\cap C^{(n)})$$
with critical point $\kappa$ and $\bar{j}(\kappa)>\lambda$.

Thus for any $\nu<\lambda$, we can derive a normal ultrafilter $U_\nu:=\{X\subseteq \mathcal{P}_\kappa(\nu)\sbp \bar{j}"\nu\in \bar{j}(X)\}$. By standard arguments $U_\nu$ is a fine normal $\kappa$-complete ultrafilter, and by a simple adaptation of Lemma \ref{lemma:scfactor} $\bar{j}$ factors through $Ult(V_\lambda, U_\nu)$ with the second factor embedding having critical point above $\nu$. Since $\bar{j}$ preserves $C^{(n)}$, the embedding generated by $U_\nu$ preserves $C^{(n)}$ up to $\nu$. Thus $\kappa$ is $<\hspace{-3pt}\lambda$-supercompact for $C^{(n)}$, as desired.
\end{proof}

Since Proposition \ref{prop:scC2eqext} can be generalized to show that Bagaria's notion of $C^{(n)}$-extendibility is equivalent to supercompactness for $C^{(n+1)}$ for all $n$, Proposition \ref{prop:scCneqvopenka} can be viewed as a corollary and unification of Bagaria's results on the equivalence of fragments of the Vopenka scheme with the existence of supercompact or $C^{(n)}$-extendible cardinals.

\section[\texorpdfstring{$\Sigma_n$-correctly $H_\lambda$-reflecting Cardinals}{Sigma\_n-correctly H\_lambda-reflecting Cardinals}]{$\Sigma_n$-correctly $H_\lambda$-reflecting Cardinals}
\label{section:SnHlref}

Tadatoshi Miyamoto introduced the $H_\lambda$-reflecting cardinals (\cite{miyamotosegments}, Definition 1.1) to extend the work of Goldstern and Shelah to asymmetric versions of BPFA. The natural generalization of Miyamoto's definition to formulas of a given complexity true in $V$ will form the basis of consistency proofs for $\Sigma_n$-correct bounded forcing axioms.

\begin{definition}
For cardinals $\kappa\leq\lambda$ and $n\geq 2$ an integer, we say that $\kappa$ is  $\Sigma_n$-correctly $H_\lambda$-reflecting iff $\kappa$ is regular and for every $\Sigma_n$ formula $\phi$ and $a\in H_\lambda$, if $\phi(a)$ holds, then the set of $Z\prec H_\lambda$ of size less than $\kappa$ and containing $a$ such that $V_\kappa\models\phi(\pi_Z(a))$ is stationary in $[H_\lambda]^{<\kappa}$ (recall that $\pi_Z$ is the Mostowski collapse map for $Z$). If $\lambda=\kappa^{+\alpha}$, we say that $\kappa$ is $\Sigma_n$-correctly $+\alpha$ reflecting.\footnote{The $+\alpha$-reflecting terminology is due to Fuchs (\cite{fuchshierachies}, Definition 3.9).}
\end{definition}

The $n\geq 2$ assumption is necessary to make this a genuine large cardinal notion because every regular $\kappa$ is $\Sigma_1$-correctly $H_\lambda$ reflecting for all $\lambda\geq \kappa$: if $\phi$ is $\Sigma_1$ then $V\models\phi(a)$ implies $H_\lambda\models\phi(a)$ implies $Z\models\phi(a)$ for all $Z\prec H_\lambda$; it then follows that $\pi_Z"Z\models\phi(\pi_Z(a))$, and since $\Sigma_1$ statements are upward absolute between transitive structures, $V_\kappa\models\phi(\pi_Z(a))$ whenever $|Z|<\kappa$. The $n=2$ case is exactly Miyamoto's $H_\lambda$-reflecting cardinals, by an argument very similar to Proposition \ref{prop:c2eqrefl}.

We begin with the simple observation that $\Sigma_n$-correct $H_\lambda$-reflection is a strengthening of $\Sigma_n$-correctness, so in particular it is also the case that $V\models\phi(\pi_Z(a))$ for suitable $Z$, $\phi$, and $a$:

\begin{prop}
If $\kappa\leq \lambda$ is $\Sigma_n$-correctly $H_\lambda$-reflecting, then $\kappa\in C^{(n)}$.
\label{prop:refcorr}
\end{prop}

\begin{proof}
First we show that $\kappa$ is a strong limit and thus inaccessible. For any $\alpha<\kappa$, there are club many $Z\prec H_\lambda$ of size less than $\kappa$ with $\alpha+1\subset Z$ and thus $\pi_Z(\alpha)=\alpha$, so we can reflect the $\Sigma_2$ statement "there exists a surjection from some ordinal onto the power set of $\alpha$" and get that $(2^{|\alpha|})^{V_\kappa}$ exists (if we drop the prevailing assumption that $n\geq 2$, the proposition is of course false). Since $V_\kappa$ contains the full power set of $\alpha$, $2^{|\alpha|}=(2^{|\alpha|})^{V_\kappa}<\kappa$, as desired.

Now fix $a\in V_\kappa$. If $\phi$ is a $\Sigma_n$ formula such that $V\models \phi(a)$, then we can find a $Z\prec H_\lambda$ large enough so that $\pi_Z(a)=a$ (since $V_\kappa=H_\kappa$ when $\kappa$ is inaccessible, so $|trcl(a)|<\kappa$) and such that $V_\kappa\models\phi(\pi_Z(a))$, so $V_\kappa\models\phi(a)$. By the contrapositive of the argument given in the proof of Corollary \ref{cor:Cndef}, this is sufficient to establish $\kappa\in C^{(n)}$.
\end{proof}

We have the following analogue to Corollary \ref{cor:Cndef}:

\begin{lemma}
    \label{lemma:correfdef}
    "$\kappa$ is $\Sigma_n$-correctly $+\alpha$-reflecting" is a $\Pi_n$-definable relation between $\kappa$ and $\alpha$.
\end{lemma}
\begin{proof}
    All objects relevant to the definition of $\Sigma_n$-correct $H_\lambda$ reflection can be found in $V_{\lambda+3}$, with the exception of those needed to verify that $V\models\phi(a)$. Thus we can say "$\kappa$ is regular and for all $\lambda=\kappa^{+\alpha}$, all $x=V_{\lambda+3}$, all $a\in H_\lambda$, and all $\Sigma_n$ formulas $\phi$, $\lnot\phi(a)$ or there is a set $S\in x$ satisfying the definition of stationarity as evaluated by functions in $x$ such that every $Z\in S$ is an elementary substructure of $H_\lambda$, has cardinality less than $\kappa$, contains $a$, and has a function $\pi_Z\in x$ satisfying the definition of the Mostowski collapse such that $V_\kappa\models\phi(\pi_Z(a))$." $\lnot\phi(a)$ is $\Pi_n$, regularity is $\Pi_1$, and as shown in the appendix, the definitions of $\kappa^{+\alpha}$, $H_\lambda$, and $V_{\lambda+3}$ are at worst $\Delta_2$, so (given our prevailing assumption that $n\geq 2$) the overall definition is $\Pi_n$.
\end{proof}

Using this, we analyze the hierarchy of $\Sigma_n$-correctly $+\alpha$-reflecting cardinals as $n$ and $\alpha$ vary.

\begin{prop}
\label{prop:refhierarchy}
    Let $\kappa$ be $\Sigma_n$-correctly $+\alpha$ reflecting. Then:
    \begin{enumerate}
        \item $\kappa$ is $\Sigma_m$ correctly $+\beta$-reflecting for all $m\leq n$ and $\beta\leq \alpha$
        \item If $0<\alpha<\kappa$, then for all $m<n$ there are stationarily many $\bar{\kappa}<\kappa$ which are $\Sigma_m$-correctly $+\alpha$-reflecting (cf. Lemma 4.10 in \cite{fuchshierachies})
    \end{enumerate}
\end{prop}

\begin{proof}
    (1): If $\phi$ is a $\Sigma_m$ formula and $a\in H_{\kappa^{+\beta}}$ is such that $\phi(a)$ holds, let $f:[H_{\kappa^{+\beta}}]^{<\omega}\rightarrow H_{\kappa^{+\beta}}$ be a function; to simplify things, we will assume without loss of generality than $f$ encodes the Skolem functions necessary to guarantee that any set closed under it is an elementary substructure of $H_{\kappa^{+\beta}}$. We extend this to a function $f':[H_{\kappa^{+\alpha}}]^{<\omega}\rightarrow H_{\kappa^{+\alpha}}$ by defining $f'(x)=f(x\cap H_{\kappa^{+\beta}})$.
    
    Then because $a\in H_{\kappa^{+\beta}}\subseteq H_{\kappa^{+\alpha}}$, every $\Sigma_m$ formula is $\Sigma_n$, and $\kappa$ is $\Sigma_n$-correctly $+\alpha$ reflecting, we can find a $Z'\prec H_{\kappa^{+\alpha}}$ of size less than $\kappa$, containing $a$, and closed under $f'$ with $Z'\cap H_\kappa$ transitive and $V_\kappa\models \phi(\pi_{Z'}(a))$. Let $Z:=Z'\cap H_{\kappa^{+\beta}}$. By the definition of $f'$, $Z$ is closed under $f$, and clearly $a\in Z$, $Z\cap H_\kappa=Z'\cap H_\kappa$ is transitive, $|Z|<\kappa$, and because we put Skolem functions into $f$ $Z\prec H_{\kappa^{+\beta}}$. Finally, $\pi_Z=\pi_{Z'}\upharpoonright Z$ because if $x\in Z'$ and $x\in y\in Z$ then $y\in H_{\kappa^{+\beta}}$, so by transitivity $x\in H_{\kappa^{+\beta}}$, and thus $x\in Z$, so the recursive definitions of $\pi_Z(y)$ and $\pi_{Z'}(y)$ will agree.

    (2): By Lemma \ref{lemma:correfdef}, the assertion that $\kappa$ is $\Sigma_m$-correctly $+\alpha$-reflecting is $\Pi_m$. Thus if $2\leq m <n$\footnote{or even if $m=1$, though by earlier remarks the proof in this case would reduce to a needlessly circuitous proof that $\kappa$ is Mahlo}, the assertion that $\kappa$ is $\Sigma_m$-correctly $+\alpha$-reflecting is $\Sigma_n$, so we can find stationarily many $Z\prec{H_{\kappa^{+\alpha}}}$ of size less than $\kappa$ with $Z\cap H_\kappa$ transitive, $\alpha,\kappa\in Z$, and $\pi_Z(\kappa)$ $\Sigma_m$-correctly $+\alpha$-reflecting (since the transitivity hypothesis implies that $\pi_Z(\alpha)=\alpha$). It then follows from Lemma \ref{lemma:clubequiv} that there are stationarily many $\Sigma_m$-correctly $+\alpha$-reflecting cardinals below $\kappa$, as desired.
\end{proof}

Now we compare them to other large cardinals. $\Sigma_n$-correctly $+0$-reflecting can easily be seen to be equivalent to regular $C^{(n)}$, which as we have seen lies between inaccessible and Mahlo in the consistency strength hierarchy. $+1$-reflecting, even without added correctness, is already a fair bit stronger:

\begin{prop}
    If $\kappa$ is $\Sigma_2$-correctly $+1$-reflecting, then $\kappa$ is weakly compact.
\end{prop}
\begin{proof}
    We have already shown that $\kappa$ is inaccessible, so we verify the tree property. Let $<_T\subset \kappa\times\kappa$ be a $\kappa$-tree ordering and $T=(\kappa, <_T)\in H_{\kappa^+}$. If $T$ has no cofinal branch, the statement "$T$ is a $\kappa$-tree with no cofinal branch" is $\Pi_1$. Thus we can find a $Z\prec H_{\kappa^+}$ of size less than $\kappa$ such that $\kappa, T\in Z$, $Z\cap H_\kappa$ is transitive, $Z$ closed under the mapping $\alpha\mapsto T_\alpha$, and $\pi_Z(T)$ is a $\pi_Z(\kappa)$-tree with no cofinal branch.

    By the closure and transitivity conditions, every node of $T$ which appears below level $\pi_Z(\kappa)=Z\cap\kappa$ is in $Z$ and not moved by $\pi_Z$. By elementarity, they have the same ordering relations to each other in $\pi_Z(T)$ as they did in $T$. Since $\pi_Z(T)$ has height $\pi_Z(\kappa)$, it can have no nodes on higher levels, and since all of its nodes arise from nodes of $T$, none of whose heights in the tree are moved by $\pi_Z$, it must have exactly the same nodes as $T$ on all levels below $\pi_Z(\kappa)$. Thus $\pi_Z(T)=T\upharpoonright\pi_Z(\kappa)$. However, taking the predecessors of any node in $T_{\pi_Z(\kappa)}$ gives a cofinal branch of $T\upharpoonright\pi_Z(\kappa)$, a contradiction. Thus $T$ has have had a cofinal branch to begin with, so $\kappa$ is weakly compact.
\end{proof}

It can further be shown that ($\Sigma_2$-correct) $+1$-reflection is exactly equivalent to the concept of strong unfoldability introduced by Villaveces (\cite{VillavecesChains}, which in turn implies total indescribability, so $\Sigma_n$-correct $+1$-reflection is in fact a fair bit stronger than weak compactness.

However, as the next two lemmas show, $\Sigma_n$-correct $+1$-reflection is still consistent with $V=L$ and thus below $0^\sharp$ for arbitrarily large $n$:

\begin{lemma}
\label{lemma:alt+1ref}
(adapted and generalized from Miyamoto \cite{miyamotosegments}, proof of Theorem 4.2) A cardinal $\kappa$ is $\Sigma_n$-correctly $+1$-reflecting in $L$ if and only if for all $A\in \mathcal{P}(\kappa)\cap L$ and $\Sigma_n$ formulas $\phi$ such that $L\models \phi(A, \kappa)$, there are stationarily many $\alpha<\kappa$ such that $L_\kappa\models\phi(A\cap\alpha, \alpha)$.
\end{lemma}
\begin{proof}
Assume $V=L$ throughout. For the forward implication, if $\kappa$ is $\Sigma_n$-correctly $+1$-reflecting and $\phi(A, \kappa)$ holds, there are stationarily many $Z\prec H_{\kappa^+}$ of size less than $\kappa$ and containing $A$ and $\kappa$ such that $V_{\kappa}\models \phi(\pi_Z(A), \pi_Z(\kappa))$ and $Z\cap H_\kappa$ is transitive. By Lemma \ref{lemma:clubequiv}, there are stationarily many $\alpha<\kappa$ such that $\alpha=\pi_Z(\kappa)=Z\cap\kappa$ for some such $Z$, and by elementarity, $\pi_Z(A)=A\cap \pi_Z(\kappa)$. Since even $\Sigma_1$-correctness implies inaccessibility for regular cardinals, $V_\kappa=H_\kappa=L_\kappa$. Thus there are stationarily many $\alpha<\kappa$ such that $L_\kappa\models \phi(A\cap\alpha, \alpha)$, as desired.

For the reverse implication, since both $\mathcal{P}(\kappa)$ and $H_{\kappa^+}=L_{\kappa^+}$ are well-ordered by $<_L$ with order type $\kappa^+$, there is a bijection $f_\kappa$ between them, definable with $\kappa$ as a parameter. Given any $a\in L_{\kappa^+}$ such that $L\models\phi(a)$, setting $A:=f_\kappa^{-1}(a)$, $L\models\phi(f_\kappa(A))$. It follows that there is a stationary $S\subset\kappa$ such that $L_\kappa\models \phi(f_\alpha(A\cap\alpha))$ for all $\alpha\in S$, where $f_\alpha$ has the obvious definition.

Then again by Lemma \ref{lemma:clubequiv} we have a stationary $S^*\subset [L_{\kappa^+}]^{<\kappa}$ consisting of sets $Z\prec L_{\kappa^+}$ containing $a$, $A$, and $\kappa$ such that $Z\cap\kappa\in S$. For all $Z\in S^*$, by elementarity we have $\pi_Z(A)=A\cap\pi_Z(\kappa)$ and thus $\pi_Z(a)=\pi_Z(f_\kappa(A))=f_{\pi_Z(\kappa)}(\pi_Z(A))=f_{\pi_Z(\kappa)}(A\cap\pi_Z(\kappa))$, so since $\pi_Z(\kappa)\in S$, $L_\kappa\models \phi(\pi_Z(a))$. Thus $\kappa$ is $\Sigma_n$-correctly $+1$-reflecting.
\end{proof}

\begin{corollary}
    Every $\Sigma_n$-correctly $+1$-reflecting cardinal is $\Sigma_n$-correctly $+1$-reflecting in $L$.
\end{corollary}
\begin{proof}
    Let $\kappa$ be $\Sigma_n$-correctly $+1$-reflecting. If $A\subseteq\kappa$ is constructible and $\phi$ is $\Sigma_n$ such that $L\models\phi(A, \kappa)$, then the $\Sigma_n$ formula $\phi^L(A, \kappa)$ is true in $V$. Thus there are stationarily many $Z\prec H_{\kappa^+}$ of size less than $\kappa$ containing $A$ and $\kappa$ with $Z\cap H_\kappa$ transitive and $V_\kappa\models \phi^L(\pi_Z(A), \pi_Z(\kappa))$, so by arguments similar to those given for the previous lemma there are stationarily many $\alpha<\kappa$ such that $V_\kappa\models\phi^L(A\cap\alpha, \alpha)$. Since $L^{V_\kappa}=L_\kappa$ and any club in $L$ is still a club in $V$, $\{\alpha<\kappa\sbp L_\kappa\models \phi(A\cap\alpha, \alpha)\}$ is constructible and stationary in $L$. It follows that $L\models\text{"}\kappa$ is $\Sigma_n$-correctly $+1$-reflecting".
\end{proof}

However, if there is a $\kappa$ which is even $\Sigma_2$-correctly $+2$-reflecting, Miyamoto \cite{miyamotosegments} shows that $BPFA^{<\omega_4}+2^{\aleph_0}=\aleph_2$ is consistent (Theorem 3.1; note that Miyamoto writes the axiom as $\Sigma(\omega_3)$ and that the hypothesis of the theorem can be forced over any model with a $+2$-reflecting cardinal). Schimmerling (\cite{SCHIMMERLINGcoherent}, two paragraphs following Proposition 1.2), combining his own work with various results from Todorcevic, Velickovic, and Steel, observes that under the forcing axiom for proper posets of size at most $(2^{\aleph_0})^+$ (which certainly holds in Miyamoto's model), $\square(\aleph_2)$ and $\square(\aleph_3)$ simultaneously fail, so the axiom of determinacy holds in $L(\mathbb{R})$. Thus the consistency strength of $\Sigma_n$-correctly $+2$-reflecting cardinals is at least as high as infinitely many Woodin cardinals.

With an added cardinal arithmetic hypothesis, we can get even more consistency strength more straightforwardly (recall that every 1-extendible cardinal is superstrong and a limit of superstrongs, and thus of greater consistency strength than any number of Woodin cardinals; see e.g. Kanamori's Proposition 26.11(a) \cite{kanamori}).

\begin{prop}
\label{prop:+2ref}
    If $\kappa$ is $\Sigma_2$-correctly $+2$-reflecting and $2^\kappa=\kappa^+$, then there are stationarily many 1-extendible cardinals below $\kappa$. If $\kappa$ is additionally $\Sigma_3$-correctly $+2$-reflecting, then it is itself 1-extendible.
\end{prop}
\begin{proof}
    Since $\kappa$ is inaccessible, $|V_\kappa|=\kappa$, so $|V_{\kappa+1}|=2^\kappa=\kappa^+$ and hence $V_{\kappa+1}\in H_{\kappa^{++}}$. If $\phi(V_{\kappa+1}, \kappa)$ is the $(\Pi_1)$ assertion that $V_{\kappa+1}$ satisfies its own definition, then there are stationarily many $Z\prec H_{\kappa^{++}}$ of size less than $\kappa$ containing $\kappa$ and $V_{\kappa+1}$ such that $\pi_Z(V_{\kappa+1})=V_{\pi_Z(\kappa)+1}$ and $Z\cap\kappa$ is transitive. For any such $Z$, $\pi_Z^{-1}\upharpoonright V_{\pi_Z(\kappa)+1}$ witnesses that $\pi_Z(\kappa)$ is 1-extendible. By Lemma \ref{lemma:clubequiv}, there are stationarily many possible values of $\pi_Z(\kappa)$.

    If $\kappa$ is $\Sigma_3$-correctly $+2$-reflecting, assume toward a contradiction that $\kappa$ is not 1-extendible. Then we can carry out the above argument while simultaneously reflecting the $\Pi_2$ assertion that $\kappa$ is not 1-extendible, obtaining a $\pi_Z(\kappa)$ which both is and is not 1-extendible.
\end{proof}

We get a natural upper bound on the strength of correctly reflecting cardinals (analogous to Miyamoto's Proposition 1.2(3)); in doing so, it is convenient to also prove an analogue of Magidor's characterization of supercompactness \cite{MagidorSC}:

\begin{prop}
\label{prop:sceqcorref}
    The following are equivalent for any integer $n>1$ and any cardinal $\kappa$:
    \begin{enumerate}
        \item $\kappa$ is supercompact for $C^{(n-1)}$
        \item $\kappa$ is $\Sigma_n$-correctly $H_\lambda$-reflecting for all cardinals $\lambda\geq\kappa$
        \item For every $\eta>\kappa$ there is an $\alpha<\kappa$ with a nontrivial elementary embedding $\sigma: (V_\alpha, \in, C^{(n-1)}\cap\alpha)\rightarrow (V_\eta,\in C^{(n-1)}\cap \eta)$ such that $\sigma(crit(\sigma))=\kappa$
    \end{enumerate}
\end{prop}
\begin{proof}
    $(1\Rightarrow 2):$ Given a $\lambda\geq \kappa$, $\Sigma_n$ formula $\phi$, an $a\in H_\lambda$ such that $\phi(a)$, and a club $C\subseteq [H_\lambda]^{<\kappa}$, we show that there is a $Z\in C$ containing $a$ such that $V_\kappa\models \phi(\pi_Z(a))$. Setting $\nu:=|[H_\lambda]^{<\kappa}|$, since $\kappa$ is supercompact for $C^{(n-1)}$ there is by Lemma \ref{lemma:scfCnalt} an embedding $j:V\rightarrow M$ witnessing the $\nu$-supercompactness of $\kappa$ such that $M\models \phi(a)$.

    Then $j"H_\lambda, j"C\in M$, and by the unboundedness of $C$ $j"C$ is a directed subset of $j(C)$ whose union is $j"H_\lambda$; since $|j"C|\leq \nu<j(\kappa)$ and $j(C)$ is a club of $([H_{j(\lambda)}]^{<j(\kappa)})^M$, $j"H_\lambda\in j(C)$. Since $\pi_{j"H_\lambda}$ is simply the inverse of the restriction of $j$, $\pi_{j"H_\lambda}(j(a))=a$. Furthermore, by elementarity $j(\kappa)$ is supercompact for $C^{(n-1)}$ in M, so by Lemma \ref{lemma:scCn} $V_{j(\kappa)}^M\models \phi(a)$. Hence in $M$ there is a set $j"H_\lambda\in j(C)$ containing $j(a)$ such that $V_{j(\kappa)}^M\models \phi(\pi_{j"H_\lambda}(j(a)))$. Therefore in $V$ there is some $Z\in C$ containing $a$ such that $V_\kappa\models \phi(\pi_Z(a))$, as desired.

    $(2\Rightarrow 3):$ Let $\phi(\beta, x, y)$ denote the assertion that $\beta$ is an ordinal, $x=V_\beta$, and for all $\theta<\beta$, $\theta$ is $\Sigma_{n-1}$-correct if and only if $\theta\in y$. Then $\phi$ can be written in $\Sigma_n$ form, and for any $\eta>\kappa$, $\phi(\eta, V_\eta, C^{(n-1)}\cap \eta)$ holds. It follows that for any $\lambda$ large enough so that $V_\eta\in H_\lambda$, there is a $Z\prec H_\lambda$ of size less than $\kappa$ containing $\eta$, $V_\eta$, and $C^{(n-1)}\cap\eta$ such that $\phi(\pi_Z(\eta), \pi_Z(V_\eta), \pi_Z(C^{(n-1)}\cap\eta))$ holds and $Z\cap H_\kappa$ is transitive. Setting $\alpha:=\pi_Z(\eta)$, it is immediate that $\pi_Z(V_\eta)=V_\alpha$ and $\pi_Z(C^{(n-1)}\cap\eta)=C^{(n-1)}\cap \alpha$.

    Then if we define $\sigma:=\pi_Z^{-1}\upharpoonright V_\alpha$, Lemma \ref{lemma:restrictembed} implies that $\sigma$ is elementary $(V_\alpha, \in, C^{(n-1)}\cap\alpha)\rightarrow (V_\eta, \in, C^{(n-1)}\cap\eta)$. By the transitivity condition on $Z$, $crit(\sigma)=\pi_Z(\kappa)$, so $\sigma$ maps its critical point to $\kappa$, as desired.

    $(3\Rightarrow 1):$ Given $\nu$, we show that $\kappa$ is $\nu$-supercompact for $C^{(n-1)}$. Let $\eta=\nu+\omega$ and $\sigma: (V_\alpha, \in, C^{(n-1)}\cap \alpha) \rightarrow (V_\eta, \in, C^{(n-1)}\cap\eta)$ be elementary with critical point $\delta$ such that $\sigma(\delta)=\kappa>\alpha$. Then we must have $\alpha=\beta+\omega$ for some $\beta$ such that $\sigma(\beta)=\nu$. By standard arguments, if $U$ is the set of all $X\subseteq\mathcal{P}_\delta(\beta)$ such that $\sigma"\beta\in \sigma(X)$, then $U$ is a fine normal $\delta$-complete ultrafilter. Furthermore, if $j_U$ is the associated ultrapower embedding, $\sigma$ factors as $j_U$ followed by an embedding which fixes all ordinals up to $\beta$, so in particular $j_U(C^{(n-1)}\cap\delta)\cap\beta=\sigma(C^{(n-1)}\cap \delta)\cap\beta=C^{(n-1)}\cap \beta$. Since $U\in V_\alpha$, by elementarity $\sigma(U)\subset\mathcal{P}(\mathcal{P}_\kappa(\nu))$ is a fine normal $\kappa$-complete ultrafilter such that $j_{\sigma(U)}(C^{(n-1)}\cap\kappa)\cap\nu=C^{(n-1)}\cap\nu$, so it witnesses that $\kappa$ is $\nu$-supercompact for $C^{(n-1)}$.
\end{proof}

If $\kappa$ is $\Sigma_n$-correctly $H_\lambda$-reflecting and $\lambda<\theta\in C^{(n-1)}$, then combining Lemma \ref{lemma:correfdef} and Lemma \ref{lemma:Cnaddexist} we get that $V_\theta\models\kappa$ is $\Sigma_n$-correctly $H_\lambda$-reflecting. If $\theta$ is regular and this holds for all $\lambda<\theta$, then $V_\theta\models ZFC+\exists \kappa$ supercompact for $C^{(n-1)}$, so correct reflection up to the next regular $C^{(n-1)}$ cardinal is equiconsistent with supercompactness for $C^{(n-1)}$.

Finally, since Laver functions are highly useful for proving the consistency of forcing axioms, we would like a notion of Laver functions appropriate to $\Sigma_n$-correct $H_\lambda$-reflection:

\begin{definition}
    $g:\kappa\rightarrow V_\kappa$ is a correctly reflecting Laver function for a $\Sigma_n$-correctly $H_\lambda$-reflecting cardinal $\kappa <\lambda$ iff for all $\Sigma_n$ formulas $\phi$ and $a\in H_\lambda$ such that $V\models\phi(a)$, there are stationarily many $Z\prec H_\lambda$ of size less than $\kappa$ such that $V_\kappa\models\phi(\pi_Z(a))$ and $g(\pi_Z(\kappa))=\pi_Z(a)$. 
\end{definition}

If $\kappa$ is in fact supercompact for $C^{(n-1)}$, any standard Laver function $g$ will be a correctly reflecting Laver function, since if $j:V\rightarrow M$ is a supercompactness embedding such that $j(g)(\kappa)=a$ and $M\models \phi(a)$, then setting $Z:=j"H_\lambda$ we get $M\models j(g)(\pi_Z(j(\kappa)))=\pi_Z(j(a))\land \phi(\pi_Z(j(a)))$, which pulled back to $V$ gives the desired properties. In general, however, it is not clear that correctly reflecting Laver functions always exist. Fortunately, they can always be added by the following forcing, developed by Woodin and best exposited by Hamkins \cite{hamkinslottery}.

\begin{definition}
    If $\kappa$ is a cardinal, the fast function forcing $\mathbb{F}_\kappa$ consists of all partial functions $p:\kappa\rightarrow\kappa$ such that:
    \begin{itemize}
        \item $|dom(p)|<\kappa$
        \item every element of $dom(p)$ is an inaccessible cardinal\footnote{Hamkins later gave a slightly different version of fast function forcing in \cite{Hamkinstall}, where arbitrary ordinals are allowed in the domain but $|p\upharpoonright \gamma|<\gamma$ is only required for inaccessible $\gamma$. This version has some nicer properties, but I have chosen to use the older one because the differences are not relevant here and the presentation in \cite{hamkinslottery} is more thorough.}
        \item for all $\gamma\in dom(p)$, $p"\gamma\subset \gamma$ and $|p\upharpoonright \gamma|<\gamma$
    \end{itemize}
    The ordering is given by $p\leq_{\mathbb{F}_\kappa} q$ iff $q\subseteq p$. If $G\subset \mathbb{P}_\kappa$ is a generic filter, we call the partial function $f=\bigcup G$ a fast function on $\kappa$.
\end{definition}

\begin{lemma}
    \label{lemma:reflaver}
    If $\kappa$ is $\Sigma_n$-correctly $H_\lambda$-reflecting for some $\lambda>\kappa$ and $f$ is a $V$-generic fast function on $\kappa$, then in $V[f]$ there is a correctly reflecting Laver function $g:\kappa\rightarrow V_\kappa$.
\end{lemma}

Since the definition of a correctly reflecting Laver function implies that $\kappa$ is $\Sigma_n$-correctly $H_\lambda$-reflecting, in particular $\mathbb{F}_\kappa$ preserves that $\kappa$ is $\Sigma_n$-correctly $H_\lambda$-reflecting.

\begin{proof}
    The following argument is essentially an adaptation of Hamkins's Generalized Laver Function Theorem 2.2 in \cite{hamkinslottery} to the correctly reflecting setting. Let $e:\kappa\rightarrow V_\kappa$ be any surjection in $V$ and define $g(\gamma)=e(f(\gamma))^{f\upharpoonright\gamma}$ whenever $\gamma\in dom(f)$ and $e(f(\gamma))$ is an $\mathbb{F}_\gamma$-name (where $e(f(\gamma))^{f\upharpoonright\gamma}$ in fact means the interpretation of $e(f(\gamma))$ by the filter corresponding to $f\upharpoonright\gamma$), and $g(\gamma)=\emptyset$ otherwise. Assume toward a contradiction that $g$ is not a correctly reflecting Laver function; then for some $\Sigma_n$ formula $\phi$, suitable names $\dot{f}$ and $\dot{g}$, and other names $\dot{a}$ and $\dot{h}$, there is a $p\in\mathbb{F}_\kappa$ which forces:

"$\phi(\dot{a})$ holds and $\dot{h}:[H_\lambda^{V[\dot{f}]}]^{<\omega}\rightarrow H_\lambda^{V[\dot{f}]}$ is a function such that for all $Z\prec H_\lambda^{V[\dot{f}]}$ of size less than $\kappa$ such that $Z\cap H_\kappa^{V[\dot{f}]}$ is transitive, $V_\kappa^{V[\dot{f}]}\models \phi(\pi_Z(\dot{a}))$, and $\dot{g}(\pi_Z(\kappa))=\pi_Z(\dot{a})$, there is a finite $u\subset Z$ such that $\dot{h}(u)\not\in Z$."

Let $\psi(p,\dot{a}, \kappa, \mathbb{F}_\kappa)$ denote the formula asserting that $\kappa$ is inaccessible and $p\Vdash_{\mathbb{F}_\kappa}\phi(\dot{a})$. Define $\tilde{h}:[H_\lambda]^{<\omega}\rightarrow H_\lambda$ by, for $\dot{x}_1,\dotsc, \dot{x}_k\in H_\lambda$ $\mathbb{F}_\kappa$-names, $\tilde{h}(\{\dot{x}_1,\dotsc, \dot{x}_k\})$ is an $\mathbb{F}_\kappa$-name $\dot{y}\in H_\lambda$ such that $p\Vdash \dot{h}(\{\dot{x}_1\dotsc, \dot{x}_k\})=\dot{y}$; for all other finite sets $u$ in its domain, $\tilde{h}(u)=\emptyset$. (Such a $\dot{y}$ will exist for any finite set of names in the domain by a mixing lemma argument.)

Since $\psi$ is a $\Sigma_n$ formula (inaccessibility is $\Pi_1$ and a condition forcing a $\Sigma_n$ formula is $\Sigma_n$), there is (in $V$) a $Z\prec H_\lambda$ such that:
\begin{enumerate}
    \item $|Z|<\kappa$
    \item $Z\cap H_\kappa=V_\beta$ for some $\beta<\kappa$ (this is possible because the set of all $Z$ with this property can easily be seen to be a club in $[H_\lambda]^{<\kappa}$ for any inaccessible $\kappa\leq \lambda$)
    \item $p, \dot{a}, \kappa, \mathbb{F}_\kappa\in Z$
    \item $Z$ is closed under $\tilde{h}$
    \item $\bar{\kappa}:=\pi_Z(\kappa)$ is inaccessible
    \item $\pi_Z(p)\Vdash_{\pi_Z(\mathbb{F}_{\kappa})}\phi(\pi_Z(\dot{a}))$
\end{enumerate}
Since $\mathbb{F}_\kappa\subset H_\kappa$, condition (2) implies that $\pi_Z(p)=p$ (and in fact this holds for all elements of $Z\cap \mathbb{F}_\kappa$). Since we must have $\beta=\bar{\kappa}$ in condition (2) and $\mathbb{F}_{\bar{\kappa}}\subset V_{\bar{\kappa}}$, $\pi_Z(\mathbb{F}_\kappa)=V_{\bar{\kappa}}\cap\mathbb{F}_\kappa=\mathbb{F}_{\bar{\kappa}}$.

Let $\alpha<\kappa$ be such that $e(\alpha)=\pi_Z(\dot{a})$. By elementarity, $p\in \mathbb{F}_{\bar{\kappa}}$, so $p"\bar{\kappa}\subset \bar{\kappa}$ and thus $p$ can be extended in $\mathbb{F}_\kappa$ to a partial function with $\bar{\kappa}$ in its domain. Let $f^*$ be a $V$-generic fast function including $p\cup\{\langle\bar{\kappa},\alpha\rangle\}$. Then $\pi_Z"f^*=f^*\upharpoonright\bar{\kappa}$, and by Hamkins's Fast Function Factor Lemma, below $\{\langle\bar{\kappa},\alpha\rangle\}$ $\mathbb{F}_\kappa$ factors as $\mathbb{F}_{\bar{\kappa}}\times \mathbb{F}_{\nu,\kappa}$, where $\nu$ is the least inaccessible above $\bar{\kappa}$ and $\alpha$ and $\mathbb{F}_{\nu,\kappa}$ is the subposet of $\mathbb{F}_\kappa$ consisting of partial functions with domains contained in $[\nu,\kappa)$, so $f^*\upharpoonright \bar{\kappa}$ is a $V$-generic fast function on $\bar{\kappa}$.

We can then apply Lemma \ref{lemma:extembed} with $M=rng(\pi_Z)$, $N=H_\lambda$, $j=\pi_Z^{-1}$, $G=f^*\upharpoonright \bar{\kappa}$, and $H=f^*$ to get that $H_\lambda[f^*]$ has an elementary substructure $Z[f^*]:=\{\dot{x}^{f^*}\sbp \dot{x}\in Z\cap H_\lambda^{\mathbb{F}_\kappa}\}$ of size less than $\kappa$ (since by Hamkins Lemma 1.3 fast function forcing does not change cardinalities) containing $\dot{a}^{f^*}$ and $\kappa$ such that $\pi_{Z[f^*]}(\dot{x}^{f^*})=\pi_Z(\dot{x})^{f^*\upharpoonright\bar{\kappa}}$ for all $\mathbb{F}_\kappa$-names $\dot{x}\in Z$.

We now verify that $Z[f^*]$ is a counterexample to the statement forced by $p$. To see that $Z[f^*]\cap H_\kappa^{V[f^*]}$ is transitive, note that $\pi_{Z[f^*]}^{-1}$ agrees with $\pi_Z^{-1}$ on the ordinals, so it sends its critical point $\bar{\kappa}$ to $\kappa$.  $V_\kappa^{V[f^*]}\models\phi(\pi_Z(\dot{a})^{f^*\upharpoonright\bar{\kappa}})$ because $p$ forces that, so it follows from the equation at the end of the previous paragraph that $V_\kappa^{V[f^*]}\models\phi(\pi_{Z[f^*]}(\dot{a}^{f^*}))$. Because $\langle\bar{\kappa},\alpha\rangle\in f^*$, we have
\begin{align*}
    \dot{g}^{f^*}(\bar{\kappa})&=e(f^*(\bar{\kappa}))^{f^*\upharpoonright\bar{\kappa}}\\
    &=e(\alpha)^{f^*\upharpoonright\bar{\kappa}}\\
    &=\pi_Z(\dot{a})^{f^*\upharpoonright\bar{\kappa}}\\
    &=\pi_{Z[f^*]}(\dot{a}^{f^*})
\end{align*}

Finally, if $u=\{u_0,\dotsc u_{k-1}\}\subset Z[f^*]$ and for each $i<k$ $\dot{u}_i\in Z$ is a name for $u_i$, then $p$ forces that $\tilde{h}(\{\dot{u}_0, \dotsc, \dot{u}_k\})\in Z$ is a name for $\dot{h}^{f^*}(u)$, so $Z[f^*]$ is closed under $\dot{h}^{f^*}$, a contradiction. Thus $g$ is a correctly reflecting Laver function.
\end{proof}

\chapter[\texorpdfstring{$\Sigma_n$-correct Forcing Axioms}{Sigma\_n-correct Forcing Axioms}]{$\Sigma_n$-correct Forcing Axioms}

\section{Statement and Consistency}
\label{section:sandcon}

To motivate the statement of $\Sigma_n$-correct forcing axioms, we first consider the following characterization of classical forcing axioms, obtained from generalizing a characterization in Jensen's handwritten notes (\cite{jensenfacch} Lemma 1) beyond the $\kappa=\omega_2$ case:

\begin{lemma}
    \label{lemma:jensenfa}
    For any forcing class $\Gamma$ consisting of separative posets and regular cardinal $\kappa>\omega_1$, the following are equivalent:
    \begin{enumerate}
    \item $FA_{<\kappa}(\Gamma)$
    \item For every $\mathbb{P}\in \Gamma$, every regular $\gamma>\kappa$ such that $\mathbb{P}\in H_\gamma$, and $X\subset H_\gamma$ such that $|X|<\kappa$, there is a transitive structure $\bar{N}$ and an elementary embedding $\sigma:\bar{N}\rightarrow H_\gamma$ such that:
    \begin{enumerate}
        \item $|\bar{N}|<\kappa$
        \item $X\cup\{\mathbb{P}\}\subseteq rng(\sigma)$
        \item there is an $\bar{N}$-generic filter $\bar{G}\subseteq \bar{\mathbb{P}}:=\sigma^{-1}(\mathbb{P})$
    \end{enumerate}
    \end{enumerate}
\end{lemma}
\begin{proof}
    We closely follow Jensen's arguments.

    $(2\Rightarrow 1):$ Let $X$ be any collection of fewer than $\kappa$ dense subsets of $\mathbb{P}\in \Gamma$ and $\gamma$ be an arbitrary regular cardinal large enough for the hypotheses of (2). Given $\bar{N}$, $\sigma$, and $\bar{G}$ from (2), let $G$ be the filter on $\mathbb{P}$ generated by $\sigma"\bar{G}$. Then for any $D\in X$, by (2)(b) there is a $\bar{D}\in \bar{N}$ such that $\sigma(\bar{D})=D$, and by elementarity $\bar{D}$ is a dense subset of $\bar{\mathbb{P}}$. Thus there is a $p\in\bar{G}\cap \bar{D}$ by $\bar{N}$-genericity, so $\sigma(p)\in G\cap D$ and hence $G$ meets every set in $X$.

    $(1\Rightarrow 2):$ Given $\mathbb{P}\in\Gamma$, assume without loss of generality that $\mathbb{P}\in X\prec H_\gamma$ and that $X\cap \kappa$ is transitive. Let $G\subseteq \mathbb{P}$ be a filter meeting all dense subsets $D$ of $\mathbb{P}$ such that $D\in X$ (since $|X|<\kappa$). Now we cannot collapse $X$ to obtain $\bar{N}$, since $G$ might meet some such $D$ at a condition which is not itself an element of $X$ and then the image of $G$ under the collapsing map would not be $\bar{N}$-generic. To avoid this issue, we thicken $X$ before collapsing it. Let 
    $$\Upsilon:=\{\dot{u}\in X\cap H_\gamma^\mathbb{P}\sbp \{p\in \mathbb{P}\sbp \exists x\in H_\gamma \hspace{6pt} p\Vdash_\mathbb{P} \dot{u}=\check{x}\}\text{ is dense in }\mathbb{P}\}$$

    For any such $\dot{u}$, the dense set witnessing $\dot{u}\in\Upsilon$ is in $X$ by elementarity. Then by the choice of $G$, there is a $p\in G$ and a unique $x\in H_\gamma$ such that $p\Vdash_\mathbb{P} \dot{u}=\check{x}$. Define (for the remainder of this proof only) $\dot{u}^G$ to be this $x$ and $Y:=\{\dot{u}^G\sbp \dot{u}\in \Upsilon\}$.

    To see that $Y$ is an elementary substructure of $ H_\gamma$, suppose $H_\gamma\models\exists x \hspace{2pt}\phi(x, a_1,\dotsc a_n)$ where $a_i=\dot{u}_i^G\in Y$ for each $i$. Let $p\in G$ be a common lower bound of the conditions in $G$ forcing $\dot{u}_i=\check{a}_i$. Then $p\Vdash H_\gamma^V\models \exists x\hspace{2pt} \phi(x, \dot{u}_1,\dotsc \dot{u}_n)$, so by the mixing lemma there is a name $\dot{u}_0$ such that $\dot{u}_0$ being equal to some  $a_0\in H_\gamma$ such that $H_\gamma^V\models \phi(\check{a}_0, \dot{u}_1,\dotsc \dot{u}_n)$ is dense below $p$. By $X\prec H_\gamma$, there is such a $\dot{u}_0$ in $X$ and hence in $\Upsilon$, so $\dot{u}_0^G\in Y$ is such that $H_\gamma\models \phi(\dot{u}_0^G, a_1, \dotsc, a_n)$. By the Tarski-Vaught criterion, $Y\prec H_\gamma$. Thus in particular it is extensional, so let $\bar{N}$ be the transitive collapse of $Y$, $\sigma=\pi_Y^{-1}$, and $\bar{G}:=\pi_Y"G$.

    We now verify the conditions in (2). For (a), $|X|<\kappa$ by hypothesis and each element of $Y$ is generated by a name in $X$, so $|Y|\leq |X|$. For (b), $X\subseteq Y=rng(\sigma)$ because for any $x\in X$, $\check{x}\in \Upsilon$ and $\check{x}^G=x$; $\mathbb{P}\in rng(\sigma)$ because we assumed $\mathbb{P}\in X$.

    Finally, for (c), any dense subset of $\bar{\mathbb{P}}:=\pi_Y(\mathbb{P})$ in $\bar{N}$ is of the form $\pi_Y(D)$ for some dense subset of $\mathbb{P}$ $D=\dot{D}^G\in Y$; thus if we can produce a $p\in G\cap D\cap Y$, $\pi_Y(p)\in \bar{G}\cap\pi_Y(D)$, establishing the $\bar{N}$-genericity of $\bar{G}$. Towards this, assume without loss of generality that the name $\dot{D}$ is such that $\Vdash_\mathbb{P}`` \dot{D}$ is a dense subset of $\check{\mathbb{P}}$ in the ground model''. Then it is forced that $\dot{D}$ meets the canonical name for the generic filter $\dot{G}$, so by the mixing lemma let $\dot{p}\in\Upsilon$ be such that $\Vdash_\mathbb{P}\dot{p}\in \dot{D}\cap\dot{G}$ and set $p:=\dot{p}^G\in Y$.

Then by the definition of $\dot{p}^G$, there is a $q\in G$ such that $q\Vdash_\mathbb{P} \dot{p}=\check{p}\land \dot{D}=\check{D}$. It follows immediately from the choice of $\dot{p}$ that $q\Vdash_{\mathbb{P}} \check{p}\in\check{D}$, which can only happen if $p\in D$. Furthermore, $q\Vdash_{\mathbb{P}} \check{p}\in \dot{G}$, so by the definition of $\dot{G}$ the set $\{r\in\mathbb{P}\sbp r\leq p\}$ is dense below $q$. Since $\mathbb{P}$ was assumed to be separative, this implies that $q\leq p$, so $p\in G$. Thus $p\in G\cap D\cap Y$, as desired.
    \end{proof}

We are now ready to state the central axiom. We would like to generalize $FA^+$ by allowing arbitrary names and arbitrary (or at least arbitrary $\Sigma_n$) provably $\Gamma$-persistent formulas rather than simply the stationarity of subsets of $\omega_1$, but as noted toward the end of Section \ref{section:MP}, this is inconsistent, since most nontrivial forcing classes can make arbitrary sets smaller than some definable cardinal (such as $\aleph_2$ or $2^{\aleph_0}$) of the forcing extension. To avoid this issue, we need to shrink our names down to a reasonable size before interpreting them. The previous lemma guides us on how to do that.

\begin{definition}
    For $\Gamma$ a forcing class, $n$ a positive integer, and $\kappa$ a regular uncountable cardinal, the $\Sigma_n$-correct forcing axiom $\Sigma_n\mhyphen CFA_{<\kappa}(\Gamma)$ is the statement that for all posets $\mathbb{P}\in\Gamma$, $\mathbb{P}$-names $\dot{a}$, sets $b$, regular cardinals $\gamma>\kappa$ such that $\mathbb{P}$, $\dot{a}$, $b\in H_\gamma$, $X\subset H_\gamma$ with $|X|< \kappa$, and provably $\Gamma$-persistent $\Sigma_n$ formulas $\phi$ such that $\Vdash_\mathbb{P}\phi(\dot{a},\check{b})$,
there is a transitive structure $N$ with an elementary embedding $\sigma:N\rightarrow H_\gamma$ such that
\begin{itemize}
    \item $|N|<\kappa$
    \item $\dot{a}$, $b$, $\mathbb{P}$, and all elements of $X$ are in the range of $\sigma$
    \item $rng(\sigma)\cap\kappa$ is transitive \item there is an $N$-generic filter $F\subset\sigma^{-1}(\mathbb{P})$ such that $\phi(\sigma^{-1}(\dot{a})^F, \sigma^{-1}(b))$ holds.
    \end{itemize}
    If $\kappa=\delta^+$ for some cardinal $\delta$, $\Sigma_n\mhyphen CFA_\delta$ is synonymous with $\Sigma_n\mhyphen CFA_{<\kappa}$.
\end{definition}

As with $ZFC_\delta$ (Definition \ref{def:ZFCdelta}), formally $\Sigma_n\mhyphen CFA_{<\kappa}(\Gamma)$ is expressed in the language of set theory with a constant symbol for $\kappa$, though for many forcing classes $\Gamma$ the resulting theory will be able to prove a particular value for $\kappa$.

The parameter $b$ can of course be incorporated into the name $\dot{a}$, and as such we will frequently omit it, but it is included here to emphasize that such parameters are permissible.

We refer to $\Sigma_n$-correct forcing axioms for certain common forcing classes by the obvious modifications of the names for the corresponding classical forcing axioms:

\begin{definition}
    $\Sigma_n\mhyphen CMA$ is $\Sigma_n\mhyphen CFA_{<2^{\aleph_0}}(ccc)$. $\Sigma_n\mhyphen CPFA$, $\Sigma_n\mhyphen CMM$, and $\Sigma_n\mhyphen CSCFA$ are $\Sigma_n\mhyphen CFA_{<\omega_2}(\Gamma)$ where $\Gamma$ is the class of proper, stationary set preserving, or subcomplete forcing, respectively. 
\end{definition}

If $\Delta\subseteq \Gamma$ and $\lambda\leq \kappa$, then $FA_{<\kappa}(\Gamma)$ implies $FA_{<\lambda}(\Delta)$. However, these downward implications fail for even $\Sigma_2$-correct forcing axioms. The class of proper forcing contains the classes of countably closed and ccc posets, but as we will see in Section \ref{section:continuum}, the $\Sigma_2$-correct proper forcing axiom, $\Sigma_2$-correct countably closed forcing axiom, and $\Sigma_2$-correct Martin's Axiom all have mutually inconsistent implications for the value of the continuum. Similarly, the statement that $\omega_2$ is the third infinite cardinal can be expressed as a provably ccc-persistent $\Sigma_2$ formula, so $\Sigma_2\mhyphen CFA_{<\omega_2}(ccc)$ implies that it is true of some ordinal in a transitive structure of cardinality at most $\omega_1$ and thus $\Sigma_2\mhyphen CFA_{<\omega_2}(ccc)$ is inconsistent, even though (as we will see shortly) $\Sigma_2\mhyphen CMA$ is consistent (and implies that the continuum is very large).

We do, however, have the following easy implications:

\begin{prop}
    $\Sigma_n\mhyphen CFA_{<\kappa}(\Gamma)$ implies:
    \begin{enumerate}
        \item $FA_{<\kappa}(\Gamma)$
        \item $\Sigma_n\mhyphen MP_\Gamma(H_\kappa)$
        \item $\Sigma_m\mhyphen CFA_{<\kappa}(\Gamma)$ for all $m\leq n$
    \end{enumerate}
\end{prop}
\begin{proof}
    (1): We apply Lemma \ref{lemma:jensenfa}, since the statement of $\Sigma_n\mhyphen CFA_{<\kappa}(\Gamma)$ is an obvious strengthening of (2) in that lemma.

    (2): For any provably $\Gamma$-persistent $\Sigma_n$ formula $\phi$, $b\in H_\kappa$, and $\mathbb{P}\in \Gamma$ such that $\Vdash_\mathbb{P} \phi(\check{b})$, let $X=trcl(\{b\})$. Then if $\sigma$ is as in the statement of the axiom applied to $X$, $\phi$, $b$, and $\dot{a}:=\emptyset$, $X\subseteq rng(\sigma)$ implies that $\sigma(b)=b$, so $\phi(b)$ holds in $V$.

    (3): Every $\Sigma_m$ formula is $\Sigma_n$.
\end{proof}

We now turn to proving the consistency of $\Sigma_n$-correct forcing axioms. First, we characterize the forcing classes for which the consistency proof works correctly (modeled on Asper\'o and Bagaria \cite{ABbfacont}, Lemma 2.2):

\begin{definition}
A forcing class $\Gamma$ is $n$-nice iff:
\begin{itemize}
    \item $\Gamma$ contains the trivial forcing
    \item Each $\mathbb{P}\in \Gamma$ preserves $\omega_1$
    \item $\Gamma$ is closed under restrictions, i.e., for any $\mathbb{P}\in\Gamma$ and $p\in\mathbb{P}$, $\mathbb{P}\upharpoonright p:=\{q\in\mathbb{P}\sbp q\leq p\}\in \Gamma$
    \item If $\mathbb{P}\in \Gamma$ and $\Vdash_\mathbb{P} \dot{\mathbb{Q}}\in \Gamma$, then $\mathbb{P}*\dot{\mathbb{Q}}\in \Gamma$
    \item For every inaccessible cardinal $\kappa$ and every forcing iteration $\langle\langle \mathbb{P}_\alpha, \dot{\mathbb{Q}}_\alpha\rangle | \alpha<\kappa\rangle$ of posets in $V_\kappa\cap \Gamma$ with some suitable support, if $\mathbb{P}_\kappa$ is the corresponding limit, then:
    \begin{itemize}
        \item $\mathbb{P}_\kappa\in \Gamma$
        \item $\Vdash_{\mathbb{P}_\alpha} \mathbb{P}_\kappa/\mathbb{P}_\alpha\in\Gamma$ in for all $\alpha<\kappa$
        \item If $\mathbb{P}_\alpha\in V_\kappa$ for all $\alpha<\kappa$, then $\mathbb{P}_\kappa$ is the direct limit of $\langle \mathbb{P}_\alpha\sbp \alpha<\kappa\rangle$ and satisfies the $\kappa$-cc
    \end{itemize}
    \item $\Gamma$ is $\Sigma_n$ definable
\end{itemize}
\end{definition}

Most forcing classes whose corresponding forcing axioms are commonly studied are 2-nice. Strictly speaking, only the last three conditions, regarding iterability and definability, are necessary for the consistency proof, but the first three are convenient and hold for most reasonable forcing classes.

\begin{theorem}
    If $\kappa$ is supercompact for $C^{(n-1)}$ and $\Gamma$ is an $n$-nice forcing class, then there is a $\kappa$-cc forcing $\mathbb{P}_\kappa\in \Gamma$ of size $\kappa$ such that if $G\subset \mathbb{P}_\kappa$ is $V$-generic, then $V[G]\models \Sigma_n\mhyphen CFA_{<\kappa}(\Gamma)$.
\label{thm:gencon}
\end{theorem}
\begin{proof}
    Let $f$ be a Laver function on $\kappa$ for $C^{(n-1)}$ and let $\mathbb{P}_\kappa$ be the standard Baumgartner iteration of $\Gamma$ of length $\kappa$ derived from $f$. That is, we recursively construct a sequence of names for posets in $\Gamma$ $\langle \dot{\mathbb{Q}}_\alpha\sbp \alpha<\kappa\rangle$ and take $\langle \mathbb{P}_\alpha\sbp \alpha\leq\kappa\rangle$ to be the iteration of it with support suitable to $\Gamma$, where $\dot{\mathbb{Q}}_\alpha=f(\alpha)$ for $\alpha$ such that $f(\alpha)$ is a $\mathbb{P}_\alpha$-name for a forcing in $\Gamma$, and otherwise $\dot{\mathbb{Q}}_\alpha$ is a $\mathbb{P}_\alpha$-name for the trivial forcing. By the definition of $n$-niceness, $\mathbb{P}_\kappa$ satisfies the $\kappa$-cc (so in particular $\kappa$ remains regular in the forcing extension) and $|\mathbb{P}_\kappa|=\kappa$. Let $G\subseteq \mathbb{P}_\kappa$ be $V$-generic.

    In $V[G]$, given a poset $\mathbb{P}\in \Gamma$, a $\mathbb{P}$-name $\dot{a}$, a parameter $b$, a regular cardinal $\gamma>\kappa$ such that $\mathbb{P}, \dot{a}, b\in H_\gamma^{V[G]}$, a set $X\subset H_\gamma^{V[G]}$ of size less than $\kappa$, and a provably $\Gamma$-persistent $\Sigma_n$ formula $\phi$ such that $\Vdash_{\mathbb{P}} \phi(\dot{a},\check{b})$, let $\dot{\mathbb{P}}, \ddot{a}, \dot{b}\in H_\gamma^V$ (applying Lemma \ref{lemma:namesize}) be $\mathbb{P}_\kappa$-names for the corresponding objects. Then there is some $p\in G$ such that $p\Vdash_{\mathbb{P}_\kappa} \dot{\mathbb{P}}\in \Gamma$ and $(p, 1_{\mathbb{P}})\Vdash_{\mathbb{P}_\kappa*\dot{\mathbb{P}}} \phi(\ddot{a},\dot{b})$ (interpreting $\ddot{a}$ and $\dot{b}$ as $\mathbb{P}_\kappa*\dot{\mathbb{P}}$-names in the obvious way). Applying Lemma \ref{lemma:scfCnalt}, there is an elementary embedding $j:V\rightarrow M$ witnessing that $\kappa$ is $2^\gamma$-supercompact such that $j(f)(\kappa)=\dot{\mathbb{P}}$ and $M\models ``p\Vdash_{\mathbb{P}_\kappa} \dot{\mathbb{P}}\in \Gamma\land (p, 1_{\mathbb{P}})\Vdash_{\mathbb{P}_\kappa*\dot{\mathbb{P}}} \phi(\ddot{a},\dot{b})"$.

    Then because $j(\mathbb{P}_\kappa)$ is constructed from $j(f)$ in the same way that $\mathbb{P}_\kappa$ is from $f$, the $\kappa$th poset used in the iteration is $\dot{\mathbb{P}}$ (and since it is forced to be in $\Gamma$, it is in fact used). Furthermore, for all $\alpha<\kappa$, $\mathbb{P}_\alpha\in V_\kappa$ and thus $j(\mathbb{P}_\alpha)=\mathbb{P}_\alpha$, so $j(\mathbb{P}_\kappa)=\mathbb{P}_\kappa*\dot{\mathbb{P}}*\dot{\mathbb{R}}$ for some $\mathbb{P}_\kappa*\dot{\mathbb{P}}$-name for a poset $\dot{\mathbb{R}}$. Let $H*K\subseteq \mathbb{P}*\dot{\mathbb{R}}$ be $V[G]$-generic.

    From the choice of $M$, $M[G][H]\models \phi(\dot{a}^H, b)$, and by one of the iteration conditions in the definition of $n$-nice forcing classes, $\dot{\mathbb{R}}^{G*H}=j(\mathbb{P}_\kappa)/(\mathbb{P}_\kappa*\dot{\mathbb{P}})\in \Gamma^{M[G][H]}$. Since $\phi$ is provably $\Gamma$-persistent and $M[G][H]\models ZFC$, $M[G][H][K]\models \phi(\dot{a}^H, b)$. (It is here that we need $\phi$ to be provably persistent rather than merely forceably necessary over $V[G]$, although see Lemma \ref{lemma:neccompl}.)

    By Lemma \ref{lemma:extembed}, $j$ extends (in $V[G][H][K]$) to an elementary embedding $j^*:V[G]\rightarrow M[G][H][K]$ defined by
    $$j^*(\dot{x}^G)=j(\dot{x})^{G*H*K}$$
    for all $\mathbb{P}_\kappa$-names $\dot{x}$. By the closure condition on $M$, $j\upharpoonright H_\gamma^V\in M$. By Lemma \ref{lemma:namesize}, every element of $H_\gamma^{V[G]}$ is of the form $\dot{x}^G$ for some $\dot{x}\in H_\gamma^V$. It follows that $\sigma':=j^*\upharpoonright H_\gamma^{V[G]}\in M[G][H][K]$, since it is definable from $G*H*K$ and the values of $j$ on $\mathbb{P}_\kappa$-names in $H_\gamma^V$. By Lemma \ref{lemma:restrictembed}, $\sigma':H_\gamma^{V[G]}\rightarrow H_{j(\gamma)}^{M[G][H][K]}$ is an elementary embedding such that $\sigma'^{-1}(j^*(\dot{a}))=\dot{a}$, $\sigma'^{-1}(j^*(b))=b$, and $\sigma'^{-1}(j^*(\mathbb{P}))=\mathbb{P}$. Since $|X|<\kappa=crit(j^*)$, $j^*(X)=j^*"X$, so in particular $j^*(X)\subseteq rng(\sigma')$. Finally, since $\sigma'$ maps its critical point $\kappa$ to $j(\kappa)$, $rng(\sigma)\cap j(\kappa)$ is transitive. We have thus shown that $M[G][H][K]$ satisfies the statement:

    "There is a transitive structure $H_\gamma^{V[G]}$ of size less than $j(\kappa)$ with an elementary embedding $\sigma':H_\gamma^{V[G]}\rightarrow H_{j(\gamma)}^{M[G][H][K]}$ such that $\{j^*(\dot{a}), j^*(b), j^*(\mathbb{P})\}\cup j^*(X)\subseteq rng(\sigma')$, $rng(\sigma')\cap j(\kappa)$ is transitive, and there is an $H_\gamma^{V[G]}$-generic filter $H\subseteq \sigma'^{-1}(j^*(\mathbb{P}))$ such that $\phi(\sigma'^{-1}(j^*(\dot{a}))^H, \sigma'^{-1}(j^*(b)))$ holds."

    By the elementarity of $j^*$, in $V[G]$ there must be a transitive structure $N$, an elementary embedding $\sigma:N\rightarrow H_\gamma^{V[G]}$, and an $N$-generic filter $F\subseteq \sigma^{-1}(\mathbb{P})$ witnessing the truth of the desired instance of the axiom.
\end{proof}

It is straightforward to check that if $\Gamma$ can ($\Gamma$-necessarily) collapse arbitrary cardinals to $\omega_1$, then $\kappa=\aleph_2^{V[G]}$, and if $\Gamma$ can add arbitrarily many reals, then $\kappa=(2^{\aleph_0})^{V[G]}$.

In the case where $n=2$, $\mathbb{P}_\kappa$ is exactly the poset used to force $FA_{<\kappa}(\Gamma)$, so by Lemma \ref{lemma:scfC1} we have:

\begin{corollary}
\label{cor:stdfamodel}
   For any 2-nice forcing class $\Gamma$, the standard model of $FA_{<\kappa}(\Gamma)$ (constructed as a forcing extension of a model where $\kappa$ is supercompact) in fact satisfies $\Sigma_2\mhyphen CFA_{<\kappa}(\Gamma)$.
\end{corollary}

We can also obtain models of what one might call the fully correct forcing axiom:

\begin{corollary}
    If the Vopenka scheme is consistent and $\Gamma$ is an $m$-nice forcing class for some $m$, then it is consistent that there is a regular $\kappa>\omega_1$ such that $\Sigma_n\mhyphen CFA_{<\kappa}(\Gamma)$ holds for all (metatheoretic) natural numbers $n$.
\end{corollary}
\begin{proof}
    By Lemma \ref{prop:scCneqvopenka}, for each $n$, any model of the Vopenka scheme has a cardinal $\kappa_n$ supercompact for $C^{(n-1)}$ and thus by Theorem \ref{thm:gencon} a forcing extension where $\Sigma_n\mhyphen CFA_{<\kappa_n}(\Gamma)$ holds and $\kappa_n$ is regular. By the compactness theorem, there is a model with a single regular $\kappa$ such that $\Sigma_n\mhyphen CFA_{<\kappa}(\Gamma)$ holds for each $n$.
\end{proof}

For the plus versions of forcing axioms, we can obtain filters which interpret not just a single name for a stationary subset of $\omega_1$ as an actual stationary set, but all names in a small transitive structure (see e.g. Remark 42 of \cite{coxfa}). This easily generalizes to any other single formula in the $\Sigma_n$-correct case:

\begin{prop}
If $\Sigma_n\mhyphen CFA_{<\kappa}(\Gamma)$ holds for some $\Sigma_n$-definable $\Gamma$, then for all $\mathbb{P}\in\Gamma$, cardinals $\gamma$ such that $\mathbb{P}\in H_\gamma$, provably $\Gamma$-persistent $\Sigma_n$ formulas $\phi$, and $X\subseteq H_\gamma$ such that $|X|<\kappa$, there is a transitive structure $N$ with an elementary embedding $\sigma: N\rightarrow H_\gamma$ such that $X\cup\{\mathbb{P}\}\subseteq rng(\sigma)$ and an $N$-generic filter $F$ such that for all $\sigma^{-1}(\mathbb{P})$-names $\dot{a}\in N$ such that $\Vdash_{\mathbb{P}} \phi(\sigma(\dot{a}))$, $\phi(\dot{a}^F)$ holds.
    \label{prop:singleformula}
\end{prop}

We call such an $F$ a $\phi$-correct $N$-generic filter.

\begin{proof}
    Given $\mathbb{P}$, $\gamma$, $\phi$, and $X$, let $\gamma'>2^{<\gamma}$ and $S=\{\dot{x}\in H_\gamma\sbp \Vdash_\mathbb{P} \phi(\dot{a})\}$. If we take $\varsigma:=\{\langle\dot{x}, 1_\mathbb{P}\rangle\sbp \dot{x}\in S\}$ to be the canonical name for the set of interpretations of elements of $S$ (not to be confused with $\check{S}$), then $\Vdash_\mathbb{P} \forall x\in \varsigma \hspace{2pt} \phi(x)$, and $\forall x\in y\hspace{2pt}\phi(x)$ is a provably $\Gamma$-persistent $\Sigma_n$ formula in $y$ because $\phi$ is a provably $\Gamma$-persistent $\Sigma_n$ formula. $\varsigma, H_\gamma\in H_{\gamma'}$ by the choice of $\gamma'$, so we can find an embedding $\sigma':N'\rightarrow H_{\gamma'}$ with $X\cup\{\mathbb{P}, \varsigma, H_\gamma\}\subseteq rng(\sigma')$ and an $N'$-generic filter $F\subseteq \bar{\mathbb{P}}:=\sigma'^{-1}(\mathbb{P})$ such that $\phi(a)$ holds for all $a\in \sigma'^{-1}(\varsigma)^F$.

    Observe that $\sigma'^{-1}(\varsigma)^F=\{\dot{a}^F\sbp \dot{a}\in N'\land \sigma'(\dot{a})\in S\}$. Therefore if we set $N:=\sigma'^{-1}(H_\gamma)$ and $\sigma=\sigma'\upharpoonright N$, then $\sigma$ is elementary by Lemma \ref{lemma:restrictembed}, $\sigma(\bar{\mathbb{P}})=\mathbb{P}$, $F$ is $N$-generic, and for every $\dot{a}\in N$ such that $\Vdash_\mathbb{P}\phi(\sigma(\dot{a}))$, $\sigma(\dot{a})\in S$, so $\dot{a}^F\in \sigma'^{-1}(\varsigma)^F$ and thus $\phi(\dot{a}^F)$ holds.
\end{proof}

We might wish to obtain filters which are fully $\Sigma_n$-correct; that is, they correctly interpret all names in $N$ with respect to all provably $\Gamma$-persistent $\Sigma_n$ formulas. However, difficulties arise with finding a formula we can apply the $\Sigma_n$-correct forcing axiom to in order to obtain such a filter. We might attempt to take the fragment $T$ of the $\Sigma_n$ forcing relation for $\mathbb{P}$ involving only names in $H_\gamma$ and formulas which are provably $\Gamma$-persistent. The problems appear when we attempt to choose a particular forceable property of $T$ to reflect, since we need to include some information about $T$ in order for ZFC to prove that the formulas in it are preserved by further forcing. If we try to use "for all $\langle \phi, \dot{a}\rangle\in T$, $\phi$ is provably $\Gamma$-persistent and $\phi(\dot{a})$ holds", then we can't prove in ZFC that the $\phi(\dot{a})$ actually continue to hold in further forcing extensions, since ZFC does not prove its own soundness. We can avoid this issue by instead reflecting "for all $\langle \phi, \dot{a}\rangle\in T$, $\phi(\dot{a})$ holds and is preserved by all forcing in $\Gamma$"; however, as noted in Lemma \ref{lemma:neccompl}, this adds additional formula complexity.

\begin{prop}
     \label{prop:correctfilter}
         If $\Sigma_{n+2}\mhyphen CFA_{<\kappa}(\Gamma)$ holds for some $\Sigma_n$-definable $\Gamma$, then for all $\mathbb{P}\in\Gamma$, regular cardinals $\gamma$ such that $\mathbb{P}\in H_\gamma$, and $X\subseteq H_\gamma$ such that $|X|<\kappa$, there is a transitive structure $N$ with an elementary embedding $\sigma: N\rightarrow H_\gamma$ such that $X\cup\{\mathbb{P}\}\subseteq rng(\sigma)$ and a $\Sigma_n$-correct $N$-generic filter $F$, i.e. for all $\sigma^{-1}(\mathbb{P})$-names $\dot{a}\in N$ and all $\Sigma_n$ formulas $\phi$ such that $\Vdash_{\mathbb{P}} ``\phi(\sigma(\dot{a}))$ holds and is preserved by all further forcing in $\Gamma"$, $\phi(\dot{a}^F)$ holds. 
         
         Consequently, if $\Sigma_n\mhyphen CFA_{<\kappa}(\Gamma)$ holds for all $n\in \omega$, for all $n$ we can find a structure with a $\Sigma_n$-correct generic filter.
\end{prop}

\begin{proof}
     Given $\mathbb{P}$, $\gamma$, and $X$, let $T=\{\langle \phi, \dot{x}\rangle \sbp \dot{x}\in H_\gamma^\mathbb{P}\land\Vdash_\mathbb{P}\square_\Gamma\phi(\dot{x})\}$, and let $\tau$ be the $\mathbb{P}$-name such that for any filter $F$, $\tau^F=\{\langle \phi, \dot{x}^F\rangle\sbp \langle \phi, \dot{x}\rangle\in T\}$. Since $\square_\Gamma \phi$ is $\Pi_{n+1}$, the statement "$\square_\Gamma \phi(x)$ holds for all $\langle \phi, x\rangle\in\tau$" can be expressed as a $\Pi_{n+1}$ formula $\psi(\tau)$, which is forced by $\mathbb{P}$. Because $\square_\Gamma \phi$ is provably $\Gamma$-persistent, $\psi$ is as well. If $\gamma'$ is sufficiently large that $\tau, H_\gamma\in H_{\gamma'}$, then there is an elementary embedding $\sigma': N'\rightarrow H_{\gamma'}$ such that $X\cup\{\mathbb{P}, \tau, H_\gamma\}\subseteq rng(\sigma')$ and an $N'$-generic filter $F\subseteq \bar{\mathbb{P}}:=\sigma'^{-1}(\mathbb{P})$ such that $\psi(\sigma'^{-1}(\tau)^F)$ holds.

    Let $N:=\sigma'^{-1}(H_\gamma)$. Then as in the proof of Proposition \ref{prop:singleformula}, $\sigma'^{-1}(\tau)^F$ consists of all pairs of $\Sigma_n$ formulas $\phi$ and $\dot{a}^F$ where $\dot{a}$ is a $\bar{\mathbb{P}}$-name $\dot{a}\in N$ such that $\mathbb{P}$ forces that $\phi(\sigma'(\dot{a}))$ holds and continues to hold after all further $\Gamma$-forcing. By $\psi(\sigma'^{-1}(\tau)^F)$, $\phi(\dot{a}^F)$ holds for each such pair. Thus if we set $\sigma=\sigma'\upharpoonright N$, $\sigma:N\rightarrow H_\gamma$ is an elementary embedding with all the desired properties and $F$ is a $\Sigma_n$-correct $N$-generic filter.
\end{proof}

\section{Equivalent Formulations}
\label{section:equivforms}

The official (Jensen-style) formulation of $\Sigma_n$-correct forcing axioms given in the previous section underscores the fact that they are generalizations of $FA^+(\Gamma)$, but can be somewhat unwieldy. In this section, we explore more streamlined presentations.

First, Philipp Schlicht and Christopher Turner \cite{stNP} have shown that classical forcing axioms are equivalent to "name principles" asserting the existence of filters interpreting a name to have a certain property. As such, the $N$-genericity conditions may be omitted:

\begin{definition}
    The Schlicht-Turner\footnote{Schlicht and Turner would perhaps be less likely to recognize this axiom than the namesakes of the other formulations would be to recognize theirs, since their work mainly focused on identifying the names which can consistently be interpreted to have certain properties, whereas my approach is to reflect arbitrary names to names small enough to have no issues. However, their work was very helpful to me in clarifying the relationship between genericity and interpretations of names, and I had to call this formulation something.} (S-T) formulation of $\Sigma_n\mhyphen CFA_{<\kappa}(\Gamma)$ is the assertion that for all provably $\Gamma$-persistent $\Sigma_n$ formulas $\phi$, posets $\mathbb{P}\in \Gamma$, $\mathbb{P}$-names $\dot{a}$ such that $\Vdash_\mathbb{P} \phi(\dot{a})$, regular cardinals $\gamma>\kappa$ such that $H_\gamma$ contains $\mathbb{P}$ and $\dot{a}$, and sets $X\subset H_\gamma$ of size less than $\kappa$, there is a $Z\prec H_\gamma$ of size less than $\kappa$ containing $\mathbb{P}$, $\dot{a}$, and all elements of $X$ such that $Z\cap H_\kappa$ is transitive and a filter $F\subseteq \pi_Z(\mathbb{P})$ such that $\phi(\pi_Z(\dot{a})^F)$ holds.
\end{definition}

In fact, we can go even further and dispense with filters altogether, yielding the following elegant principle:

\begin{definition}
    The Miyamoto-Asper\'o (M-A) formulation of $\Sigma_n\mhyphen CFA_{<\kappa}(\Gamma)$ is the assertion that for all provably $\Gamma$-persistent $\Sigma_n$ formulas $\phi$, all cardinals $\lambda\geq \kappa$, and all $b\in H_\lambda$ such that $\Vdash_\mathbb{P} \phi(\check{b})$ for some $\mathbb{P}\in \Gamma$, there are stationarily many $Z\prec H_\lambda$ of size less than $\kappa$ containing $b$ such that $\phi(\pi_Z(b))$ holds.
\end{definition}

Miyamoto's Theorem 2.5(2) in \cite{miyamotosegments} is essentially a bounded version of this principle in the special case where $n=1$, $\kappa=\omega_2$, and $\Gamma$ is the class of proper forcing. David Asper\'o, possibly inspired by Miyamoto, stated the bounded version in generality (\cite{asperomax}, Definitions 1.3 and 1.5).

Finally, we have a generic elementary embedding characterization, reminiscent of Lemma \ref{lemma:scfCnalt}. The following definition is helpful in stating it and related results:

\begin{definition}
\label{def:witnessfa}
    If $\mathbb{Q}$ is a forcing poset, $\phi$ is a formula in the language of set theory, $\dot{a}$ is a $\mathbb{Q}$-name, and $\kappa<\gamma$ are regular cardinals, we say that a generic elementary embedding $j:V\rightarrow M$ (where $M$ is a possibly ill-founded class model with transitive well-founded part) witnesses the $<\hspace{-2pt}\kappa$-forcing axiom for $(\mathbb{Q}, \phi, \dot{a}, \gamma)$ iff:
    \begin{enumerate}[(a)]
        \item $H_\gamma^V$ is in the wellfounded part of $M$
        \item $|H_\gamma^V|^M<j(\kappa)$
        \item $j\upharpoonright H_\gamma^V\in M$
        \item $crit(j)=\kappa$
        \item $M$ contains a $V$-generic filter $H\subseteq\mathbb{Q}$ such that $M\models \phi(\dot{a}^H)$
    \end{enumerate}
\end{definition}

\begin{definition}
\label{def:wcform}
    The Woodin-Cox (W-C) formulation of $\Sigma_n\mhyphen CFA_{<\kappa}(\Gamma)$ is the assertion that for every provably $\Gamma$-persistent $\Sigma_n$ formula $\phi$, $\mathbb{Q}\in \Gamma$, $\mathbb{Q}$-name $\dot{a}$ such that $\Vdash_\mathbb{Q} \phi(\dot{a})$, and regular $\gamma>|\mathcal{P}(\mathbb{Q})\cup trcl(\dot{a})\cup\kappa|$, there is a generic elementary embedding $j:V\rightarrow M$ which witnesses the $<\hspace{-2pt}\kappa$-forcing axiom for $(\mathbb{Q}, \phi, \dot{a}, \gamma)$.
\end{definition}

Woodin showed the equivalence for classical forcing axioms with the existence of suitable generic embeddings under the additional assumption of a proper class of Woodin cardinals, in which case the codomain $M$ can be taken to be wellfounded, with the generic elementary embedding produced by stationary tower forcing (\cite{Woodinbook}, Theorem 2.53). Sean Cox generalized Woodin's result by proving that even without the Woodin cardinals we can still get an equivalence with embeddings into illfounded models, and that this also holds for $FA^{+\nu}$ for any $\nu\leq \omega_1$ (\cite{coxfa}, Theorem 43). (Note that since there is a $\Sigma_2$ characterization of Woodin cardinals and there are unboundedly many below any supercompact $\kappa$, "there is a proper class of Woodin cardinals" is a $\Pi_3$ sentence which holds in $V_\kappa$. Thus if $\kappa$ is in fact supercompact for $C^{(2)}$ (or merely supercompact and $\Sigma_3$-correct), then there must be a proper class of Woodins in $V$ and any set forcing extension of $V$, so in particular for $n\geq 3$ our consistency proof of $\Sigma_n\mhyphen CFA_{<\kappa}(\Gamma)$ in fact produces a model of the hypotheses of Woodin's Theorem 2.53.)

\begin{theorem}
\label{thm:equivforms}
    The following are equivalent for all positive integers $n$, regular cardinals $\kappa>\omega_1$, and forcing classes $\Gamma$:
    \begin{enumerate}
        \item The official (Jensen-style) formulation of $\Sigma_n\mhyphen CFA_{<\kappa}(\Gamma)$ given in Section \ref{section:sandcon}
        \item The Miyamoto-Asper\'o formulation of $\Sigma_n\mhyphen CFA_{<\kappa}(\Gamma)$
        \item The Woodin-Cox formulation of $\Sigma_n\mhyphen CFA_{<\kappa}(\Gamma)$
        \item The Schlicht-Turner formulation of $\Sigma_n\mhyphen CFA_{<\kappa}(\Gamma)$
    \end{enumerate}
\end{theorem}
\begin{proof}
    $(1\Rightarrow 2):$ Fix $\phi$, $\lambda$, $b$, and $\mathbb{P}$ as in the M-A formulation. We plan to apply the Jensen formulation to the formula $\phi$ and the name $\check{b}$; the only difficulty is ensuring that the set of possible ranges of the embedding $\sigma$ is stationary in $[H_\lambda]^{<\kappa}$. For this, let $h:[H_\lambda]^{<\omega}\rightarrow H_\lambda$; we will produce a $Z\prec H_\lambda$ of size less than $\kappa$, containing $b$, and closed under $h$ such that $Z\cap H_\kappa$ is transitive and $\phi(\pi_Z(b))$ holds.
    
    Let $\gamma$ be a regular cardinal large enough that $H_\lambda\in H_\gamma$. Let $X:=\{H_\lambda, h, \mathbb{P}, b\}$; then if $N$ is transitive and smaller than $\kappa$ and $\sigma:N\rightarrow H_\gamma$ is elementary with $X\subset rng(\sigma)$, $rng(\sigma)\cap H_\kappa$ transitive, and $\phi(\sigma^{-1}(b))$, define $Z:=rng(\sigma)\cap H_\lambda$. That $Z$ is smaller than $\kappa$ and transitive below $\kappa$ is immediate. That $\phi(\pi_Z(b))$ holds follows from the fact that $\pi_Z$ is a restriction of $\sigma^{-1}$, while $Z\prec H_\lambda$ is a consequence of Lemma \ref{lemma:restrictembed}. Finally, if $\bar{Z}$ and $g$ are such that $\sigma(\bar{Z})=H_\lambda$ and $\sigma(g)=h$, then by elementarity $\bar{Z}$ is closed under $g$. Since every finite subset of $Z$ is of the form $\{\sigma(x_1),\dotsc \sigma(x_k)\}$ for some $x_1,\dotsc, x_k\in \bar{Z}$ and $g(\{x_1,\dotsc, x_k\})\in \bar{Z}$, by elementarity $h(\{\sigma(x_1),\dotsc, \sigma(x_k)\})=\sigma(g(\{x_1,\dotsc, x_k\}))\in Z$. Thus $Z$ has all the desired properties, so the M-A formulation holds.

    $(2\Rightarrow 3):$ We closely follow Cox's proof of Theorem 43 (\cite{coxfa}). Let $\phi$, $\mathbb{Q}$, $\dot{a}$, and $\gamma$ be as in the statement of the W-C formulation. First, observe that for any filter $F$, $\dot{a}^F$ is $\Delta_1$-definable from $\dot{a}$ and $F$. It follows that, if $\mathcal{D}\in H_\gamma$ is the set of all dense subsets of $\mathbb{Q}$ in $V$, the statement $\psi(\dot{a},\mathbb{Q}, \mathcal{D})$:="there exists a $\mathcal{D}$-generic filter $F\subseteq \mathbb{Q}$ such that $\phi(\dot{a}^F)$ holds" is a $\Sigma_n$ formula which is forced by $\mathbb{Q}$; furthermore, since $\phi$ is provably $\Gamma$-persistent and forcing does not destroy $F$ or change the interpretation of names, $\psi$ is provably $\Gamma$-persistent as well.

    Then if $R$ is the set of all $Z\prec H_\gamma$ of size less than $\kappa$ such that $\mathbb{Q}, \dot{a}\in Z$, $Z\cap H_\kappa$ transitive, and $\psi(\pi_Z(\dot{a}), \pi_Z(\mathbb{Q}), \pi_Z(\mathcal{D}))$ holds, by the M-A formulation $R$ is stationary in $[H_\gamma]^{<\kappa}$. Let $\mathbb{B}$ be the power set of $R$ modulo the restriction of the nonstationary ideal and let $U\subset\mathbb{B}$ be a $V$-generic ultrafilter.

    Taking $j: V\rightarrow M$ to be the generic ultrapower embedding derived from $U$, that $j$ witnesses the $\lk$-forcing axiom for $(\mathbb{Q}, \phi, \dot{a}, \gamma)$ follows easily from the basic theory of generic ultrapowers, but the arguments involved will be briefly indicated for the benefit of readers unfamiliar with that theory. $j"H_\gamma^V\subseteq [id_R]_U$ because any $x\in H_\gamma^V$ is in all but nonstationarily many $Z\in R$, so by \L os's Theorem $j(x)=[Z\mapsto x]_U\in [id_R]_U$ for all such $x$; furthermore the reverse inclusion follows from the normality of $U$ (which in turn follows from the normality of the club filter and the genericity of $U$), so $[id_R]_U=j"H_\gamma^V$. Similarly, the function $Z\mapsto \pi_Z"Z$ on $R$ is the coordinatewise transitive collapse of $id_R$, so $[Z\mapsto \pi_Z"Z]_U\cong H_\gamma^V$, the coordinatewise transitive collapse of $j"H_\gamma^V$; since $H_\gamma^V$ is a transitive set from a well-founded model, this transitive isomorphic copy must lie in the well-founded part of $M$, and hence if we take that well-founded part to be transitive the isomorphic copy will be equal to $H_\gamma^V$, and so $H_\gamma^V\in wfp(M)$.
    
    Furthermore, since $Z\mapsto \pi_Z"Z$ is coordinatewise smaller than $\kappa$, in $M$ $H_\gamma^V$ is smaller than $j(\kappa)=[Z\mapsto \kappa]_U$. For condition (c) in the definition of witnessing a forcing axiom, observe that coordinatewise $Z\mapsto \pi_Z^{-1}$ is the inverse of the Mostowski isomorphism of $id_R$, so $[Z\mapsto \pi_Z^{-1}]_U$ is the inverse Mostowski isomorphism of $j"H_\gamma^V$, i.e. $j\upharpoonright H_\gamma^V$. For (d), $Z\mapsto Z\cap\kappa$ is a function on $R$ everywhere less than $\kappa$ but less than any particular $\alpha<\kappa$ only nonstationarily often, so $j$ must be discontinuous at $\kappa$; it fixes each ordinal less than $\kappa$ because the $\kappa$-additivity of the nonstationary ideal and the genericity of $U$ imply that $U$ must be $\kappa$-complete over $V$.

    Setting $\pi:=(j\upharpoonright H_\theta^V)^{-1}=[Z\mapsto \pi_Z]_U$, further invocations of \L os show that $[Z\mapsto \pi_Z(\mathbb{Q})]_U=\pi(j(\mathbb{Q}))=\mathbb{Q}$, $[Z\mapsto \pi_Z(\dot{a})]_U=\pi(j(\dot{a}))=\dot{a}$, $[Z\mapsto \pi_Z(\mathcal{D})]_U=\mathcal{D}$, and $[Z\mapsto G_Z]_U$ (where $G_Z$ is the filter whose existence is asserted by $\psi(\pi_Z(\dot{a}), \pi_Z(\mathbb{Q}), \pi_Z(\mathcal{D}))$) is a $\mathcal{D}$-generic (hence $V$-generic) filter on $\mathbb{Q}$ which interprets $\dot{a}$ correctly. Thus $j$ witnesses the $\lk$-forcing axiom for $(\mathbb{Q}, \phi, \dot{a}, \gamma)$.
    
    $(3\Rightarrow 4):$ Given $\phi$, $\mathbb{P}$, $\dot{a}$, $\gamma$, and $X$ as in the S-T formulation, let $j: V\rightarrow M$ be a generic elementary embedding which witnesses the $\lk$-forcing axiom for $(\mathbb{P}, \phi, \dot{a}, \gamma)$. (If $\gamma$ is not large enough to meet the requirements of the W-C formulation, we can replace it with a larger $\gamma$ and then easily draw all the desired conclusions about our original $\gamma$.) Since $|X|<\kappa=crit(j)$, $j(X)=j"X$, so $Z':= j"H_\gamma^V$ is an elementary substructure of $H_{j(\gamma)}^M$ containing $j(\mathbb{P})$, $j(\dot{a})$, and all elements of $j(X)$. As elements of $H_\kappa$ are not moved by $j$ and sets outside of $H_\kappa$ are mapped to sets outside of $H_{j(\kappa)}^M$, $Z'\cap H_{j(\kappa)}^M$ is transitive. Finally, since $\pi_{Z'}(j(\mathbb{P}))=\mathbb{P}$ and $\pi_{Z'}(j(\dot{a}))=\dot{a}$, there is a filter $H\subseteq \pi_{Z'}(j(\mathbb{P}))$ in $M$ such that $\phi(\pi_{Z'}(j(\dot{a}))^H)$ holds. Pulling all of this back to $V$, there is a $Z\prec H_\gamma^V$ and a filter $F\subseteq\pi_Z(\mathbb{P})$ witnessing the truth of the S-T formulation of the axiom.

    $(4\Rightarrow 1):$ Given $\mathbb{P}$, $\dot{a}$, $\gamma$, $X$, and $\phi$ as in the Jensen formulation, let $\delta$ be large enough that $H_\gamma\in H_\delta$ and let $\chi(x, y, H)$ denote the assertion that $\phi(x)$ holds and $H$ is a $y$-generic filter. Then if $\dot{G}$ is the canonical $\mathbb{P}$-name for the generic filter, $\Vdash_{\mathbb{P}}\chi(\dot{a}, \check{H}_\gamma, \dot{G})$ and $\chi$ is a provably $\Gamma$-persistent $\Sigma_n$ formula (as asserting that a filter is generic over a transitive structure only requires quantifying over the filter and the structure).
    
    It follows from the S-T formulation that there is a $Z\prec H_\delta$ of size less than $\kappa$ containing $\mathbb{P}$, $\dot{a}$, $H_\gamma$, $\dot{G}$, and all elements of $X$ such that $Z\cap H_\kappa$ is transitive and a filter $F\subset \pi_Z(\mathbb{P})$ such that $\chi(\pi_Z(\dot{a})^F, \pi_Z(H_\gamma), \pi_Z(\dot{G})^F)$ holds. Since by elementarity $\pi_Z(\dot{G})$ is the canonical $\pi_Z(\mathbb{P})$-name for the generic filter, $\pi_Z(\dot{G})^F=F$, so setting $N:=\pi_Z(H_\gamma)$ and $\sigma:=\pi_Z^{-1}\upharpoonright N$, $F$ is $N$-generic and $\phi(\sigma^{-1}(\dot{a})^F)$ holds. By Lemma \ref{lemma:restrictembed} $\sigma$ is an elementary embedding, and all other desired properties of $N$ and $\sigma$ are immediate from the choice of $Z$, so the Jensen formulation holds.
\end{proof}

\section[\texorpdfstring{Do $\Sigma_n$-correct Forcing Axioms Form a Strict Hierarchy in $n$?}{Do Sigma\_n-correct Forcing Axioms Form a Strict Hierarchy in n?}]{Do $\Sigma_n$-correct Forcing Axioms Form a Strict Hierarchy in $n$?}
\label{section:hierarchy}

It is natural to ask whether increasing the value of $n$ in the $\Sigma_n$-correct forcing axioms produces a strictly stronger axiom, or if $\Sigma_n\mhyphen CFA_{<\kappa}$ can ever imply $\Sigma_{n+1}\mhyphen CFA_{<\kappa}$. In the $\Sigma_1$ vs $\Sigma_2$ case, this is difficult to answer in general, but for most specific forcing classes of interest, results in later chapters will imply that we do get a separation. At $n=2$, we can prove that moving one more level up produces strictly stronger axioms for a wide range of forcing classes, including for example all classes which can add arbitrarily many reals or collapse arbitrarily large cardinals.

\begin{prop}
    \label{prop:S2nimpS3}
    Let $\Gamma$ be an $2$-nice forcing class and $\Delta$ a forcing class which can destroy arbitrarily many inaccessibles, i.e. for any set $X\subset Ord$, there is a poset in $\Delta$ which forces "$\check{X}$ does not contain any inaccessible cardinals". Then if it is consistent that there is a supercompact cardinal with an inaccessible above it, $\Sigma_2\mhyphen CFA_{<\kappa}(\Gamma)$ does not imply $\Sigma_3\mhyphen MP_\Delta(\emptyset)$. In particular, if $\Gamma$ is a 2-nice forcing class which can destroy arbitrarily many inaccessibles, then $\Sigma_2\mhyphen CFA_{<\kappa}(\Gamma)$ does not imply $\Sigma_3\mhyphen CFA_{<\kappa}(\Gamma)$.
\end{prop}
\begin{proof}
    By truncating the universe if necessary, we can assume that there is a model $V$ with a supercompact $\kappa$ with exactly one inaccessible $\lambda$ above it. Applying Corollary \ref{cor:stdfamodel}, there is a forcing extension $V[G]\models \Sigma_2\mhyphen CFA_{<\kappa}(\Gamma)$. Since the forcing in question has cardinality $\kappa<\lambda$, $\lambda$ will remain inaccessible in $V[G]$, and since forcing does not add inaccessibles it will be the largest inaccessible in $V[G]$\footnote{In the main case of interest where $\Gamma$ can destroy inaccessibles, this fact will reflect to $V_\kappa$, the inaccessible-destroying posets will be included in the Baumgartner iteration, and so $\lambda$ will in fact be the only inaccessible in $V[G]$, but that is not relevant to the proof.}.

    We now show that $\Sigma_3\mhyphen MP_\Delta(\emptyset)$ fails in $V[G]$ by the argument of Hamkins in Theorem 3.9 of \cite{hamkinsMP}. Since inaccessibility is a $\Pi_1$ property, the assertion "all ordinals are not inaccessible cardinals" is a $\Pi_2$ sentence. If there are only set many inaccessibles, there is some forcing in $\Delta$ which destroys them all (without adding new ones, of course) and so makes this statement true. Since the statement is preserved by all further forcing, it holds in any model of the $\Sigma_3$ maximality principle for $\Delta$ without a proper class of inaccessibles. As $V[G]$ has a nonempty set of inaccessibles, it cannot satisfy $\Sigma_3\mhyphen MP_\Delta(\emptyset)$.
\end{proof}

It is tempting to try to generalize the above argument by replacing "inaccessible" with "regular $C^{(n-1)}$". However, the proof made essential use of the fact that forcing cannot add inaccessible cardinals, which need not hold for $\Sigma_{n-1}$-correctness when $n>2$. We could define $\diamondsuit C^{(n)}$ to be those cardinals which are $\Sigma_n$-correct in some forcing extension, and then $\Sigma_{n+2}\mhyphen MP_\Delta(\emptyset)$ will imply that the intersection of $\diamondsuit C^{(n)}$ with the class of regular cardinals is empty or a proper class, but then it is no longer clear how to arrange that there is a model of $\Sigma_{n+1}\mhyphen CFA_{<\kappa}(\Gamma)$ with $\diamondsuit C^{(n)}\cap Reg$ a nonempty set, since truncations at forceably $\Sigma_n$-correct cardinals are not necessarily well-behaved.

We can, however, at least show that $\Sigma_n$-correct forcing axioms become stronger when increasing $n$ by at least two with a more elaborate geological argument.

\begin{lemma}
    \label{lemma:n+2reflCn+1}
    For any forcing class $\Gamma$ which contains the trivial forcing, $\Sigma_{n+2}\mhyphen CFA_{<\kappa}(\Gamma)$ implies that there are unboundedly many ordinals less than $\kappa$ which are $\Sigma_{n+1}$-correct in some ground.
\end{lemma}
\begin{proof}
    Fix any $\theta\in C^{(n+1)}$. Combining Corollary \ref{cor:Cndef} and Lemma \ref{lemma:groundrel}, the formula $\theta\in (C^{(n+1)})^{W_r}$ is $\Pi_{n+1}$ in $\theta$ and $r$, so $\exists r\hspace{3pt} \theta\in (C^{(n+1)})^{W_r}$ is $\Sigma_{n+2}$. Since it is forceable by the trivial forcing and can easily be seen to be provably preserved by all forcing, applying the Miyamoto-Asper\'o form of the axiom, there are stationarily many $Z\prec H_{\theta^+}$ of size less than $\kappa$ such that $\pi_Z(\theta)$ is $\Sigma_{n+1}$-correct in some ground. By considering $Z$ which contain sufficiently many ordinals below $\kappa$, we can force $\pi_Z(\theta)$ to be arbitrarily large below $\kappa$.
\end{proof}

We now need to construct a model of $\Sigma_n\mhyphen CFA_{<\kappa}(\Gamma)$ with no $\theta<\kappa$ which is $\Sigma_{n+1}$-correct in any ground. First, we need the following generalization of the downward direction of the Levy-Solovay Theorem:

\begin{theorem}
\label{thm:downlevysolovay}
    (special case of Hamkins \cite{Hamkinsgap}, Gap Forcing Theorem) If $\delta$ is a regular cardinal, $\mathbb{P}$ is a forcing poset with $\mathbb{P}\in H_\delta$, $G\subseteq\mathbb{P}$ is a $V$-generic filter, and $j:V[G]\rightarrow M'$ is an elementary embedding definable from a parameter $u$ with $crit(j)>\delta$ and $M'$ closed under $\delta$-sequences in $V[G]$, then setting $M:=\bigcup\limits_{\alpha\in Ord} j(V_\alpha^V)$:
    \begin{enumerate}
        \item $M=M'\cap V$
        \item $M'=M[G]$
        \item $j\upharpoonright V: V\rightarrow M$ is an elementary embedding
        \item $j\upharpoonright V$ is definable in $V$ from a name for $u$
    \end{enumerate}
\end{theorem}

\begin{corollary}
    \label{cor:scfCnLS}
    Suppose that $\kappa$ is a cardinal, $\mathbb{P}\in H_\kappa$ is a forcing poset, and $G\subseteq \mathbb{P}$ is a $V$-generic filter. Then $\kappa$ is supercompact for $C^{(n)}$ in $V$ if and only if it is supercompact for $C^{(n)}$ in $V[G]$.
\end{corollary}
\begin{proof}
    First, assume that $\kappa$ is supercompact for $C^{(n)}$ in $V$. Given $\lambda>\kappa$, we show that $\kappa$ is $\lambda$-supercompact for $C^{(n)}$ in $V[G]$. Let $\dot{x}$ be a $\mathbb{P}$-name for $(C^{(n)})^{V[G]}\cap\lambda$; then there is a $p\in G$ which forces the $\Delta_{n+1}$ property that $\dot{x}$ consists of the $\Sigma_n$-correct cardinals below $\lambda$. By Lemma \ref{lemma:scfCnalt}, there is an elementary embedding $j:V\rightarrow M$ witnessing the $\lambda$-supercompactness of $\kappa$ such that $M$ thinks that $p$ forces $\dot{x}$ to be the set of $\Sigma_n$-correct cardinals below $\lambda$. Since $j"G=G$, by Lemma \ref{lemma:extembed} $j$ extends to an elementary embedding $j^*:V[G]\rightarrow M[G]$. Then $crit(j^*)=\kappa$ and $j(\kappa)>\lambda$ because $j^*$ agrees with $j$ on the ordinals, $j^*(C^{(n)}\cap\kappa)\cap\lambda=(C^{(n)})^{M[G]}\cap\lambda=\dot{x}^G=(C^{(n)})^{V[G]}\cap\lambda$, and by Lemma \ref{lemma:closurepreserved} $M[G]$ is closed under $\lambda$-sequences in $V[G]$. Thus $\kappa$ remains supercompact for $C^{(n)}$ in $V[G]$.

    Conversely, given any $\theta>\kappa$ in $(C^{(n)})^{V[G]}$, let $j: V[G]\rightarrow M'$ be an elementary embedding with $crit(j)=\kappa$, $j(\kappa)>\theta$, $(^\theta M')^{V[G]}\subset M'$, and $(C^{(n)})^{V[G]}\cap(\theta+1)=(C^{(n)})^{M'}\cap(\theta+1)$. Then by Theorem \ref{thm:downlevysolovay}, there is an inner model $M$ of $V$ such that $M'=M[G]$ and $\bar{j}:=j\upharpoonright V: V\rightarrow M$ is elementary. $crit(\bar{j})=\kappa$ and $\bar{j}(\kappa)>\theta$ are again immediate. For the closure condition, if $f:\theta\rightarrow M$ is in $V$, then it is in $V[G]$ and thus in $M[G]$ by the closure condition there, so $f\in M[G]\cap V=M$. Finally, since $\theta$ is $\Sigma_n$-correct in both $V[G]$ and $M[G]$, by Lemma \ref{lemma:Cnground} it is $\Sigma_{n}$-correct in both $V$ and $M$ as well. Since $\theta$ is a beth fixed point, $|V_\theta|=\theta$, so $V_\theta=V_\theta^M$ and thus both $V$ and $M$ agree with $V_\theta$'s computation of the $\Sigma_{n}$ cardinals below $\theta$. Hence $\kappa$ is supercompact for $C^{(n)}$ in $V$.
\end{proof}

\begin{lemma}
\label{lemma:noCn+1}
    If $n\geq 3$, $\Gamma$ is an $n$-nice forcing class, and there are two cardinals supercompact for $C^{(n-1)}$, then there is a model of $\Sigma_n\mhyphen CFA_{<\kappa}(\Gamma)$ such that no ordinal less than $\kappa$ is $\Sigma_{n+1}$-correct in any ground.
\end{lemma}
\begin{proof}
    Let $\delta<\lambda$ be supercompact for $C^{(n-1)}$. First, by Proposition \ref{prop:scC2eqext}, $\delta$ is extendible, so Usuba's Theorem \ref{thm:extmantle} implies that the mantle $\mathbb{M}$ is a ground of $V$. Usuba's arguments in fact show that $V$ is a forcing extension of $\mathbb{M}$ by a poset $\mathbb{P}\in\mathbb{M}$ such that $|\mathbb{P}|^{\mathbb{M}}\leq (2^{2^{\delta^{++}}})^V$ (see the discussion following Definition 2.6 in \cite{Usubaextendible}); since $\lambda$ is inaccessible in both $\mathbb{M}$ and $V$, $\mathbb{P}\in V_\lambda^\mathbb{M}$.

    Let $\kappa$ be the least ordinal which is supercompact for $C^{(n-1)}$ in some forcing extension of $\mathbb{M}$ by a poset of size less than $\lambda$. Since $V$ is such a forcing extension and $\delta$ is supercompact for $C^{(n-1)}$ there, such an ordinal exists and $\kappa\leq\delta<\lambda$. Let $\mathbb{M}[G]$ be a forcing extension by a poset $\mathbb{Q}$ smaller than $\lambda$ in which $\kappa$ is supercompact for $C^{(n-1)}$ and $\mathbb{M}[G][H]$ be a further forcing extension by the poset $\mathbb{P}_\kappa$ from Theorem \ref{thm:gencon} in which $\Sigma_n\mhyphen CFA_{<\kappa}(\Gamma)$ holds.

    Let $\theta<\kappa$ and let $W$ be a ground of $\mathbb{M}[G][H]$; we show that $\theta$ is not $\Sigma_{n+1}$-correct in $W$. As $W$ is a ground of a forcing extension of a ground of $V$, it is in the generic multiverse, so it is a forcing extension of $\mathbb{M}$ by Proposition \ref{prop:genmulti}. Hence by the Intermediate Model Lemma \ref{lemma:intermodel}, $W$ is a generic extension of the mantle by some complete subalgebra of the Boolean completion of $\mathbb{Q}*\dot{\mathbb{P}}_\kappa$, which has size less than $\lambda$. Thus by the minimality of $\kappa$, there are no cardinals supercompact for $C^{(n-1)}$ below $\theta$ in $W$.
    
    However, by the downward direction of Corollary \ref{cor:scfCnLS}, $\lambda$ is supercompact for $C^{(n-1)}$ in $\mathbb{M}$, so by the upward direction the same is true in $W$. Thus $V_\theta^W$ either incorrectly identifies a cardinal as supercompact for $C^{(n-1)}$ or it disagrees with $W$ on the sentence "there exists a cardinal supercompact for $C^{(n-1)}$". Combining Lemma \ref{lemma:correfdef} and Proposition \ref{prop:sceqcorref}, being supercompact for $C^{(n-1)}$ is a $\Pi_n$ property and the existence of such a cardinal is $\Sigma_{n+1}$. It follows that in either case, $\theta\not\in (C^{(n+1)})^W$.
\end{proof}

Combining Proposition \ref{prop:S2nimpS3}, Lemma \ref{lemma:n+2reflCn+1}, and Lemma \ref{lemma:noCn+1}, we have:

\begin{theorem}
    \label{thm:+2hierarchy}
    For all positive integers $n$ and forcing classes $\Gamma$ (where if $n<3$ we need the assumption that $\Gamma$ can destroy arbitrarily many inaccessibles), if it is consistent that there is are two cardinals supercompact for $C^{(n-1)}$, then $\Sigma_n\mhyphen CFA_{<\kappa}(\Gamma)$ does not imply $\Sigma_{n+2}\mhyphen CFA_{<\kappa}(\Gamma)$.
\end{theorem}

We are left with the following open questions, where a positive answer to the first would easily yield a positive answer to the second:

\begin{question}
    Is it possible to produce a model of $\Sigma_n\mhyphen CFA_{<\kappa}(\Gamma)$ where $\diamondsuit C^{(n-1)}\cap Reg$ is neither empty nor a proper class when $n>2$?
\end{question}

\begin{question}
    Is $\Sigma_{n+1}\mhyphen CFA_{<\kappa}(\Gamma)$ a strictly stronger axiom than $\Sigma_n\mhyphen CFA_{<\kappa}(\Gamma)$ when $n>2$?
\end{question}

\section[\texorpdfstring{Forcing Axioms for Classes which Collapse $\omega_1$}{Forcing Axioms for Classes which Collapse omega\_1}]{Forcing Axioms for Classes which Collapse $\omega_1$}
\label{section:collo1}

Since the original motivation of this work was generalizing $FA^+$, most of the focus so far has been on classes which preserve $\omega_1$, as those are the classes for which classical and "plus" forcing axioms make sense and have been previously studied. However, this restriction is not necessary. One could call a forcing class \textit{weakly n-nice} if it satisfies all the conditions of $n$-niceness except possibly preservation of $\omega_1$, and then the proof of Theorem \ref{thm:gencon} will yield models of $\Sigma_n\mhyphen CFA_{<\kappa}(\Gamma)$ for any weakly $n$-nice $\Gamma$, where if $\Gamma$ can necessarily collapse $\omega_1$ we will get $\kappa=\omega_1^{V[G]}$. The classical forcing axiom content of such axioms will be trivial, of course, but this isn't really an issue; the classical forcing axiom content of $\Sigma_n$-correct forcing axioms for countably closed forcing, is after all, similarly provable in $ZFC$.

One could even consider forcing axioms for the class of all forcing. By Theorem \ref{thm:cbfafactor}, this will turn out to be the conjunction of the $\Sigma_n$-maximality principle for the class of all forcing together with a reflection principle for provably forcing-persistent properties. However, it is somewhat difficult to identify interesting consequences of this axiom beyond the consequences of the maximality principle. Lemma \ref{lemma:+1equicon} at least yields the implication that $\omega_1$ is $\Sigma_n$-correctly $+1$-reflecting in $L$, so it has noticeably greater consistency strength than the maximality principle. However, even the answer to the following question is unclear.

\begin{question}
    Does $\Sigma_n\mhyphen CFA_{<\omega_1}(all)$ imply that $0^\sharp$ exists?
\end{question}

Further exploration of this topic will be left for future work.

\section{Internal vs External Provable Persistence}
\label{section:intvext}

The statements of all formulations of the $\Sigma_n$-correct forcing axioms given so far have been somewhat ambiguous. Both possible interpretations have slight but easily manageable drawbacks.

First, $\Sigma_n\mhyphen CFA_{<\kappa}(\Gamma)$ could be read as an axiom scheme, with one axiom for each (external) provably $\Gamma$-persistent $\Sigma_n$ formula. This has the disadvantage that the most natural form of the axiom scheme is undecidable, since there is no way to determine whether ZFC proves a formula $\Gamma$-persistent except to wait for a proof to be found. However, the resulting axiom set is at least computably enumerable, and Craig's theorem states that every computably enumerable set of axioms is equivalent to a computable set of axioms\footnote{Thanks to Russell Miller for relating this fact to me.}.

Alternatively, $\Sigma_n\mhyphen CFA_{<\kappa}(\Gamma)$ could be read as a single sentence quantifying over all (internal) $\Sigma_n$ formulas of the model which the model's ZFC proves to be $\Gamma$-persistent, making use of Lemma \ref{lemma:sntruth} to handle assertions about the truth of the formulas. This eliminates concerns about decidability because a single sentence is of course decidable, and in fact even the scheme consisting of $\Sigma_n\mhyphen CFA_{<\kappa}(\Gamma)$ for each $n$ is decidable. New difficulties arise from the fact that ZFC does not prove its own soundness (or equivalently, the ZFC of a nonstandard model need not be sound), so the mere fact that a model believes a formula to be provably $\Gamma$-persistent does not mean that it actually continues to hold in $\Gamma$-extensions of the model. We can address this issue by slightly strengthening the hypotheses of our consistency proof.

\begin{prop}
    If $\Gamma$ is an $n$-nice forcing class, $\kappa$ is supercompact for $C^{(n-1)}$ and there is a regular $\zeta\in C^{(n)}$ above $\kappa$, then there is a model in which the internal version of $\Sigma_n\mhyphen CFA_{<\kappa}(\Gamma)$ holds.
\end{prop}
\begin{proof}
    We construct $V[G]$ as in the proof of Theorem \ref{thm:gencon} and follow that proof up until the point where we need to show that $M[G][H][K]\models \phi(\dot{a}^H, b)$. By elementarity, $j(\zeta)\in (C^{(n)})^M$, so by Lemma \ref{lemma:Cnforce} it is in $(C^{(n)})^{M[G][H]}$ as well. Similarly, it is regular in $M$, so since $|\mathbb{P}_\kappa*\mathbb{P}|<j(\kappa)<j(\zeta)$, it remains regular (and hence inaccessible) in $M[G][H]$.

    For the purposes of this proof, let $ZFC$ denote the external theory and $ZFC^V$ denote the theory defined within the models under consideration (since inner models and forcing extensions do not alter arithmetic truth, all of them will have the same ZFC). Now because $M[G][H]$ is a model of $ZFC$, it believes that $V_{j(\zeta)}^{M[G][H]}\models ZFC^V$ (since ZFC proves that $V_\alpha$ satisfies internal ZFC whenever $\alpha$ is inaccessible). By $\Sigma_n$-correctness, $V_{j(\zeta)}^{M[G][H]}\models \phi(\dot{a}^H, b)\land \dot{\mathbb{R}}^{G*H}\in \Gamma$ (observing that all parameters have size at most $j(\kappa)$ and hence are certainly contained in $V_{j(\zeta)}^{M[G][H]}$). Then because $ZFC^V$ proves that $\phi$ is preserved by $\Gamma$-forcing, $V_{j(\zeta)}^{M[G][H]}[K]=V_{j(\zeta)}^{M[G][H][K]}\models\phi(\dot{a}^H, b)$. Since $\mathbb{P}_\kappa\subset V_\kappa$ and $\dot{\mathbb{R}}$ is a factor of $j(\mathbb{P}_\kappa)$, $|\dot{\mathbb{R}}^{G*H}|^{M[G][H]}=j(\kappa)<j(\zeta)$. Applying Lemma \ref{lemma:Cnforce} again, $j(\zeta)\in (C^{(n)})^{M[G][H][K]}$, so $M[G][H][K]\models \phi(\dot{a}^H, b)$. The rest of the proof proceeds as for Theorem \ref{thm:gencon}.
\end{proof}

Since the distinction between the internal and external versions of the axiom is fairly technical and not particularly relevant, we will largely ignore it outside this section.

\chapter[\texorpdfstring{Bounded $\Sigma_n$-correct Forcing axioms}{Bounded Sigma\_n-correct Forcing axioms}]{Bounded $\Sigma_n$-correct Forcing axioms}

We now turn our attention to the bounded versions of $\Sigma_n$-correct forcing axioms. In order to transition from the unbounded to bounded forms, we need to add two restrictions. First, we can no longer ask that our filter $F$ be fully $N$-generic, since some of the maximal antichains in $N$ may get mapped to excessively large antichains by $\sigma$. To accommodate this, we will use the following natural restricted form of $N$-genericity:

\begin{definition}
    If $\beta$ is an ordinal, $\mathbb{B}$ is a Boolean algebra, and $N$ is a transitive $ZFC^-$ model containing $\beta$ and $\mathbb{B}$ such that $N\models ``\beta\text{ is a cardinal and }\mathbb{B}\text{ is a}$ \\ $\text{complete Boolean algebra}"$, then a filter $F\subset\mathbb{B}$ is $<\hspace{-3pt}\beta$-weakly $N$-generic iff $F$ meets all maximal antichains $A\in N$ of $\mathbb{B}$ such that $|A|^N<\beta$.
\end{definition}

Second, we need to limit the size of $\dot{a}$ to ensure that it does not encode information about excessively large antichains. (This occurs most blatantly in situations involving formulas like "$\dot{G}$ is a $\check{H}_\gamma$-generic filter" as in the proof of Theorem \ref{thm:equivforms}, but can also happen in more subtle ways.) The most precise smallness condition would be that if $\dot{x}$ is a $\mathbb{B}$-name in the transitive closure of $\{\dot{a}\}$, $|\dot{x}|<\lambda$, where $\lambda$ is our (strict) bound. However, we will instead require that $\dot{a}\in H_\lambda$, since this is easier to state and work with than the more precise condition, clearly implies it, and can be made to hold whenever the precise smallness condition does by replacing $\mathbb{B}$ we an isomorphic algebra such that all forcing conditions which occur in the transitive closure of $\dot{a}$ are contained in $H_\lambda$.

With these restrictions added, our statement of the axiom becomes:

\begin{definition}
    If $\kappa>\omega_1$ and $\lambda\geq \kappa$ are cardinals and $\Gamma$ is a forcing class, $\Sigma_n\mhyphen CBFA_{<\kappa}^{<\lambda}(\Gamma)$ is the statement that for all complete Boolean algebras $\mathbb{B}\in\Gamma$, $\mathbb{B}$-names $\dot{a}\in H_{\lambda}$, sets $b\in H_{\lambda}$, regular cardinals $\gamma\geq\lambda$ such that $\mathbb{B}\in H_\gamma$, $X\subset H_\gamma$ with $|X|< \kappa$, and provably $\Gamma$-persistent $\Sigma_n$ formulas $\phi$ such that $\Vdash_\mathbb{B}\phi(\dot{a},\check{b})$,
there is a transitive structure $N$ with an elementary embedding $\sigma:N\rightarrow H_\gamma$ such that $\dot{a}$, $b$, $\mathbb{B}$, $\lambda$, and all elements of $X$ are in the range of $\sigma$, $rng(\sigma)\cap\kappa$ is transitive, and there is a $<\hspace{-2pt}\sigma^{-1}(\lambda)$-weakly $N$-generic filter $F\subset\bar{\mathbb{B}}:=\sigma^{-1}(\mathbb{B})$ such that $\phi(\sigma^{-1}(\dot{a})^F, \sigma^{-1}(b))$ holds.

As before, if $\kappa$ and/or $\lambda$ are successor cardinals, we may write their predecessors in place of $<\kappa$ or $<\lambda$.
\end{definition}

David Asper\'o stated an equivalent axiom, using what we previously called the Miyamoto-Asper\'o formulation (Definitions 1.3 and 1.5 of \cite{asperomax}). However, he appears to have only published a consistency proof for the case where $\kappa=\lambda$.

As we did in Section \ref{section:sandcon}, we note some obvious implications:

\begin{prop}
\label{prop:cbfaeasy}
    \begin{enumerate}
        \item $\Sigma_n\mhyphen CBFA_{<\kappa}^{<\lambda}(\Gamma)$ implies $BFA_{<\kappa}^{<\lambda}(\Gamma)$.
        \item $\Sigma_n\mhyphen CBFA_{<\kappa}^{<\lambda}(\Gamma)$ implies $\Sigma_n\mhyphen MP_\Gamma(H_\kappa)$.
        \item $\Sigma_n\mhyphen CBFA_{<\kappa}^{<\lambda}(\Gamma)$ implies $\Sigma_m\mhyphen CBFA_{<\kappa}^{<\lambda}(\Gamma)$ for any $m\leq n$.
        \item $\Sigma_n\mhyphen CBFA_{<\kappa}^{<\lambda}(\Gamma)$ implies $\Sigma_n\mhyphen CBFA_{<\kappa}^{<\nu}(\Gamma)$ for any $\nu\leq \lambda$.
        \item $\Sigma_n\mhyphen CFA_{<\kappa}(\Gamma)$ is equivalent to the assertion that $\Sigma_n\mhyphen CBFA_{<\kappa}^{<\lambda}(\Gamma)$ holds for all $\lambda$.
    \end{enumerate}
\end{prop}
\begin{proof}
    $(1):$ For any complete Boolean algebra $\mathbb{B}\in\Gamma$, let $X$ consist of any desired collection of fewer than $\kappa$ maximal antichains of $\mathbb{B}$, each of size less than $\lambda$. Let $N$ be a transitive structure with an elementary embedding $\sigma: N\rightarrow H_\gamma$ for some sufficiently large $\gamma$ such that $X\cup\{\lambda, \mathbb{B}\}\subset rng(\sigma)$. Setting $\bar{\lambda}:=\sigma^{-1}(\lambda)$ and $\bar{\mathbb{B}}:=\sigma^{-1}(\mathbb{B})$, let $F\subset\bar{\mathbb{B}}$ be a $<\hspace{-3pt}\bar{\lambda}$-weakly $N$-generic filter and $G\subset\mathbb{B}$ be the filter generated by $\sigma" F$. Then for any maximal antichain $A\in X$, if $\bar{A}:=\sigma^{-1}(A)$, by elementarity $\bar{A}$ is a maximal antichain of $\mathbb{B}$ of $N$-cardinality less than $\bar{\lambda}$, so there is some $p\in \bar{A}\cap F$. It follows that $\sigma(p)\in A\cap G$, so $G$ witnesses the truth of $BFA_{<\kappa}^{<\lambda}(\Gamma)$.

    $(2):$ For any provably $\Gamma$-persistent $\Sigma_n$ formula $\phi$, $b\in H_\kappa$, and $\mathbb{P}\in \Gamma$ such that $\Vdash_\mathbb{P} \phi(\check{b})$, let $X=trcl(\{b\})$. Then if $N$, $F$, and $\sigma$ are as in the statement of the axiom applied to $X$, $\phi$, $b$, and $\dot{a}:=\emptyset$, $X\subseteq rng(\sigma)$ implies that $\sigma(b)=b$, so $\phi(b)$ holds in $V$.

    $(3):$ Every $\Sigma_m$ formula is $\Sigma_n$.

    $(4):$ Any $\dot{a}$ and $b$ in $H_\nu$ must be in $H_\lambda$, and since $\sigma^{-1}(\nu)\leq \sigma^{-1}(\lambda)$ for any elementary embedding $\sigma$, any $<\hspace{-2pt}\sigma^{-1}(\lambda)$-weakly $N$-generic filter is $<\hspace{-2pt}\sigma^{-1}(\nu)$-weakly $N$-generic.

    $(5):$ The forward direction is immediate, observing that any $N$-generic filter is $<\hspace{-4pt}\beta$-weakly $N$-generic for all $N$-cardinals $\beta$. For the converse, given $\mathbb{P}$, $\dot{a}$, and $b$, let $\mathbb{B}$ be the Boolean completion of $\mathbb{P}$ and choose $\lambda$ large enough that $\dot{a}, b, \mathbb{B}\in H_\lambda$. Then we can apply $\Sigma_n\mhyphen CBFA_{<\kappa}^{<\lambda}(\Gamma)$, and since every maximal antichain of $\sigma^{-1}(\mathbb{B})$ must have size less than $\sigma^{-1}(\lambda)$, $<\hspace{-2pt}\sigma^{-1}(\lambda)$-weak $N$-genericity implies full $N$-genericity.
\end{proof}

\section{Consistency Proofs}
\label{section:cbfacon}

Now we show that $\Sigma_n$-correct bounded forcing axioms are consistent relative to the appropriate $\Sigma_n$-correctly $H_\lambda$-reflecting cardinals. We start with the symmetric case, since this is simpler but does not allow us to use correctly reflecting Laver functions (as Lemma \ref{lemma:reflaver} requires that $\lambda>\kappa$). Asper\'o (\cite{asperomax}, Theorem 2.6) and, as we will later see, Hamkins (\cite{hamkinsMP}, Lemma 3.3) proved the consistency of principles equivalent to $\Sigma_n$-correct symmetrically bounded forcing axioms; the proof below adapts their arguments to work with our official formulation of the axiom.

\begin{theorem}
    \label{thm:symbfa}
    If $\Gamma$ is an $n$-nice forcing class and $\kappa\in C^{(n)}$ is regular, there is a $\kappa$-cc poset $\mathbb{P}_\kappa\in \Gamma$ which forces $\Sigma_n\mhyphen CBFA_{<\kappa}^{<\kappa}(\Gamma)$.
\end{theorem}
\begin{proof}
    Since $\kappa$ is inaccessible, let $f:\kappa\rightarrow \kappa^2\times V_\kappa^4\times \omega$ be a surjection such that for all $\alpha<\kappa$, the first coordinate of $f(\alpha)$ is at most $\alpha$. Fix an enumeration $\langle \phi_k\sbp k<\omega\rangle$ of the $\Sigma_n$ formulas in the language of set theory. Using these, we recursively construct a sequence $\langle \dot{\mathbb{Q}}_\alpha\sbp \alpha<\kappa\rangle$ of names for posets in $\Gamma\cap V_\kappa$ and take $\langle \mathbb{P}_\alpha\sbp \alpha\leq\kappa\rangle$ to be the iteration of this sequence with support suitable to $\Gamma$.
    
    If $\mathbb{P}_\alpha$ has already been defined and $f(\alpha)=(\beta, \mu, \dot{M}, \dot{\mathbb{R}}, \ddot{a}, \dot{b}, k)$, where:

    \begin{itemize}
        \item $\dot{M}$ is a $\mathbb{P}_\beta$-name for a transitive set
        \item $\dot{\mathbb{R}}$ is a $\mathbb{P}_\beta$-name for a Boolean algebra (not necessarily in $\Gamma$)
        \item $\ddot{a}$ is a $\mathbb{P}_\beta$-name for a $\dot{\mathbb{R}}$-name
        \item $\dot{b}$ is a $\mathbb{P}_\beta$-name
    \end{itemize}
    then we choose $\dot{\mathbb{Q}}_\alpha$ such that, for any $p\in \mathbb{P}_\alpha$ which forces
    \begin{itemize}
        \item $\check{\mu}$ is a cardinal of $\dot{M}\models ZFC^-$
        \item $\dot{\mathbb{R}}, \ddot{a}, \dot{b}\in \dot{M}$
        \item there is a poset $\mathbb{Q}\in \Gamma\cap V_\kappa$ which forces that there is a $<\hspace{-2pt}\mu$-weakly $\dot{M}$-generic filter $F\subseteq \dot{\mathbb{R}}$ such that $\phi_k(\ddot{a}^F,\dot{b})$ holds
    \end{itemize}
    $p$ forces that $\dot{\mathbb{Q}}_\alpha$ is such a $\mathbb{Q}$, while if $p$ forces any of the conditions on the second list to fail, it forces $\dot{\mathbb{Q}}_\alpha$ to be trivial. If any of the conditions on the first list fail, we choose $\dot{\mathbb{Q}}_\alpha$ to be a canonical name for the trivial forcing.

    Let $G\subseteq \mathbb{P}_\kappa$ be $V$-generic; we will show that $V[G]\models \Sigma_n\mhyphen CBFA_{<\kappa}^{<\kappa}(\Gamma)$. Fix $\gamma>\kappa$, $\mathbb{B}\in \Gamma\cap H_\gamma^{V[G]}$, $\dot{a}, b\in H_\kappa^{V[G]}$ with $\dot{a}$ a $\mathbb{B}$-name, $X\subset H_\gamma^{V[G]}$ of size less than $\kappa$, and $\phi$ a provably $\Gamma$-persistent $\Sigma_n$ formula such that $\Vdash_\mathbb{B} \phi(\dot{a}, \check{b})$. Let $Y\in [H_\gamma^{V[G]}]^{<\kappa}$ be such that $Y\prec H_\gamma^{V[G]}$, $X\cup trcl(\{\dot{a}, b\})\cup \{\mathbb{P}, \kappa\}\subseteq Y$, and $Y\cap \kappa$ is transitive, since the sets with any of those three properties form clubs in $[H_\gamma^{V[G]}]^{<\kappa}$, so the intersection of all three of them is nonempty. Then by Lemma \ref{lemma:macSkolem}, whenever $A\in Y$ is a maximal antichain of $\mathbb{B}$ with $|A|<\kappa$, $A\subset Y$. Define $N$ to be the transitive collapse of $Y$, $\sigma:N\rightarrow H_\gamma^{V[G]}$ the inverse collapse embedding, $\bar{\mathbb{B}}:=\sigma^{-1}(\mathbb{B})$ and $\bar{\kappa}:=\sigma^{-1}(\kappa)$.
    
    Then let $\ddot{a}$, $\dot{b}$, $\dot{\bar{\mathbb{B}}}$, and $\dot{N}$ be suitable $\mathbb{P}_\kappa$-names in $V$; by Lemma \ref{lemma:itername}, we can in fact arrange that for some $\beta<\kappa$, $\ddot{a}$, $\dot{b}$, $\dot{\bar{\mathbb{P}}}$, and $\dot{N}$ are $\mathbb{P}_\beta$-names in $V_\kappa$ with $\ddot{a}^{G_\beta}=\dot{a}$, $\dot{b}^{G_\beta}=b$, $\dot{\bar{\mathbb{B}}}^{G_\beta}=\bar{\mathbb{B}}$, and $\dot{N}^{G_\beta}=N$.
    
    Then because $f$ was taken to be surjective, we can find  $\alpha<\kappa$ such that $f(\alpha)=(\beta, \bar{\kappa}, \dot{N}, \dot{\bar{\mathbb{P}}}, \ddot{a}, \dot{b}, k)$, where $\phi=\phi_k$. It is then immediate from the choices of the parameters involved that all the conditions for $\dot{\mathbb{Q}}_\alpha^{G_\alpha}$ to be nontrivial are met, except possibly the last condition.
    
    To see that the last condition holds as well, observe that $\mathbb{P}_\kappa/\mathbb{P}_\alpha *\dot{\mathbb{B}}^+$ (where $\dot{\mathbb{B}}^+$ is a name for $\mathbb{B}-\{0_\mathbb{B}\}$) is a poset in $\Gamma^{V[G_\alpha]}$ which adds a filter $H\subseteq\mathbb{B}$ which meets all maximal antichains of $\mathbb{B}$ in $Y$. Thus if we take $F:=\sigma^{-1}"H$, whenever $A$ is a maximal antichain of $\bar{\mathbb{B}}$ of size less than $\bar{\kappa}$, $\sigma(A)\in Y$ is a maximal antichain of $\mathbb{B}$ of size less than $\kappa$. By the construction of $Y$, all elements of $\sigma(A)$ are in $Y$, so since $Y=rng(\sigma)$, in particular there is a $\sigma(p)\in \sigma(A)\cap H$ for some $p\in\bar{\mathbb{B}}$. It follows that $p\in A\cap F$, so $F$ is $<\hspace{-2pt}\bar{\kappa}$-weakly $N$-generic. Furthermore, $\phi(\dot{a}^H, b)$ holds because $\mathbb{B}$ forces it to, so because $trcl(\{\dot{a}, b\})\subset Y$, none of the conditions of $\mathbb{B}$ relevant to the interpretation of $\dot{a}$ are moved by $\sigma^{-1}$, so $\dot{a}^F=\dot{a}^H$.

    Hence the last condition holds if we drop the requirement that the poset $\mathbb{Q}$ witnessing it is in $V_\kappa^{V[G_\alpha]}$. However, the existence of such a poset is $\Sigma_n$ expressible, and by Lemma \ref{lemma:Cnforce}, $\kappa$ remains $\Sigma_n$-correct in $V[G_\alpha]$. Thus we can find such a $\mathbb{Q}$ in $\Gamma\cap V_\kappa^{V[G_\alpha]}$, so $\mathbb{Q}_\alpha$ will be such a poset.

    It follows that in $V[G_{\alpha+1}]$, there is a $<\hspace{-2pt}\bar{\kappa}$-weakly $N$-generic filter $F\subseteq\bar{\mathbb{P}}$ such that $\phi(\dot{a}^F, b)$ holds. Since $\mathbb{P}_\kappa/\mathbb{P}_{\alpha+1}\in\Gamma$ by the definition of $n$-niceness, $\phi(\dot{a}^F, b)$ continues to hold in $V[G]$, while $<\hspace{-2pt}\bar{\kappa}$-weak $N$-genericity is a $\Delta_0$ relation between $F$, $N$, and $\bar{\kappa}$, so it is preserved by arbitrary extensions. Thus $V[G]\models \Sigma_n\mhyphen CBFA_{<\kappa}^{<\kappa}(\Gamma)$.
\end{proof}

We now turn to the asymmetric case:

\begin{theorem}
\label{thm:asymbfa}
    If $\kappa$ is $\Sigma_n$-correctly $H_\lambda$-reflecting for $n\geq 2$ and some regular $\lambda>\kappa$, $f$ is a $V$-generic fast function on $\kappa$, and $\Gamma$ is an $n$-nice forcing class, there is a $\kappa$-cc poset $\mathbb{P}_\kappa\in \Gamma^{V[f]}$ such that if $G\subseteq \mathbb{P}_\kappa$ is $V[f]$-generic, then $V[f][G]\models \Sigma_n\mhyphen CBFA_{<\kappa}^{<\lambda}(\Gamma)$.
\end{theorem}
\begin{proof}
    By Lemma \ref{lemma:reflaver}, in $V[f]$ there is a correctly reflecting Laver function $g:\kappa\rightarrow V_\kappa^{V[f]}$. As usual, we wish to construct an iteration $\langle \mathbb{P}_\alpha\sbp \alpha\leq \kappa\rangle$ of (names for) posets $\langle \dot{\mathbb{Q}}_\alpha\sbp \alpha<\kappa\rangle$ in $\Gamma\cap V_\kappa^{V[f]}$ with support suitable for $\Gamma$. We cannot quite let $g$ select $\dot{\mathbb{Q}}_\alpha$ as in the proof of Theorem \ref{thm:gencon}, since we wish to apply the axiom to arbitrarily large posets in $\Gamma$ but $g$ only works nicely with parameters in $H_\lambda^{V[f]}$. Instead, we follow the approach of Theorem \ref{thm:symbfa}, with a more complex reflection argument.

    Assume that $\mathbb{P}_\alpha$ has already been defined. We construct $\dot{\mathbb{Q}}_\alpha$ so that it names the trivial forcing unless $g(\alpha)=(\alpha, \mu, \dot{M}, \dot{\mathbb{R}}, \ddot{a}, \dot{b}, p, \phi)$\footnote{Some of the parameters in this tuple (most obviously $\alpha$) are redundant in the choice of $\dot{\mathbb{Q}}_\alpha$, but will be needed as parameters in the statement which we will eventually want to reflect, so by the definition of a correctly reflecting Laver function we must include them here in order for the reflection to work correctly}, where:
    \begin{itemize}
        \item $\dot{M}$ is a $\mathbb{P}_\alpha$-name for a transitive structure containing $\mu$, $\dot{\mathbb{R}}$, $\ddot{a}$, and $\dot{b}$
        \item $\dot{\mathbb{R}}$ is a $\mathbb{P}_\alpha$-name for a Boolean algebra (not necessarily in $\Gamma$)
        \item $\ddot{a}$ is a $\mathbb{P}_\alpha$-name for a $\dot{\mathbb{R}}$-name
        \item $\dot{b}$ is a $\mathbb{P}_\alpha$ name
        \item $p\in \mathbb{P}_\alpha$ forces $\check{\alpha}$ and $\check{\mu}$ to be cardinals of $\dot{M}$
        \item $\phi$ is a $\Sigma_n$ formula 
        \item $p$ forces that there is a poset $\mathbb{Q}\in \Gamma$ of size less than $\kappa$ which adds a $<\hspace{-5pt}\check{\mu}$-weakly $\dot{M}$-generic filter $F\subset \dot{\mathbb{R}}$ such that $\phi(\ddot{a}^F, b)$ holds
    \end{itemize}
    If all of these hypotheses hold, we arrange that $p$ also forces $\dot{\mathbb{Q}}_\alpha$ to be a poset as in the last item in $V_\kappa$ of the forcing extension by $\mathbb{P}_\alpha$. Let $G\subseteq\mathbb{P}_\kappa$ be $V[f]$-generic.

    In $V[f][G]$, fix a complete Boolean algebra $\mathbb{B}\in \Gamma$, a $\mathbb{B}$-name $\dot{a}\in H_\lambda^{V[f][G]}$, a parameter $b\in H_\lambda^{V[f][G]}$, a cardinal $\gamma\geq\lambda$ such that $\mathbb{B}\in H_\gamma^{V[f][G]}$, a set $X\subset H_\gamma$ of size less than $\kappa$, and a provably $\Gamma$-persistent $\Sigma_n$ formula $\phi$ such that $\Vdash_{\mathbb{B}}\phi(\dot{a}, \check{b})$. As in the previous proof, we can find a $Y\prec H_\gamma^{V[f][G]}$ of size less than $\lambda$ such that
    $$\{\mathbb{B},\lambda\}\cup trcl(\{\dot{a}, b\})\cup (\kappa+1)\cup X\subset Y$$
    and $Y\cap\lambda\in\lambda$. Then by Lemma \ref{lemma:macSkolem}, whenever $A\subset\mathbb{B}$ is a maximal antichain of size less than $\lambda$ and $A\in Y$, we have $A\subset Y$. Let $N'$ be the transitive collapse of $Y$, $\lambda':=\pi_Y(\lambda)$ and $\sigma':N'\rightarrow H_\gamma^{V[f][G]}$ denote $\pi_Y^{-1}$.

    In $V[f]$, let $\dot{N}'$, $\dot{\mathbb{B}}'$, $\ddot{a}$, and $\dot{b}$ be $\mathbb{P}_\kappa$-names for $N'$, $\pi_Y(\mathbb{B})$, $\dot{a}$, and $b$ respectively. Since $|\mathbb{P}_\kappa|=\kappa<\lambda$ and all the parameters are in $H_\lambda^{V[f][G]}$, by Lemma \ref{lemma:namesize} we can arrange that the names for them are in $H_\lambda^{V[f]}$. Let $\psi(\kappa, \lambda', \dot{N}', \dot{\mathbb{B}}', \ddot{a}, \dot{b},p,\phi)$ abbreviate the assertion that $\kappa$ is inaccessible and $p$ forces "there exists a poset in $\Gamma$ (namely $\mathbb{B}-\{0_{\mathbb{B}}\}$) which adds a $<\hspace{-2pt}\lambda'$-weakly $\dot{N}'$-generic filter $F\subset\dot{\mathbb{B}}'$ such that $\phi(\ddot{a}^F, b)$ holds". Inaccessibility is immediate, while the second part holds for some $p\in G$ because if $H\subset\mathbb{B}$ is a $V[f][G]$-generic filter and $A\in N'$ is a maximal antichain of $\mathbb{B}'$ with $|A|^{N'}<\lambda'$, then by elementarity $\sigma'(A)$ is a maximal antichain of $\mathbb{B}$ of size less than $\lambda$, so $\sigma'(A)\subset Y$ by the construction of $Y$. Thus there is some $q\in \sigma'(A)\cap H\cap Y$, so $\pi_Y(q)\in A$ and thus $\pi_Y"H$ is the desired filter.

    Since inaccessibility is a $\Pi_1$ property and the rest of $\psi$ merely adds some existential and bounded quantifiers to the $\Sigma_n$ assertions that a poset is in $\Gamma$ and $\phi$ holds, $\psi$ is overall a $\Sigma_n$ statement (as we assume that $n\geq 2$) about parameters in $H_\lambda^{V[f]}$. Hence by the definition of $g$, we can find a $Z\prec H_\lambda^{V[f]}$ of size less than $\kappa$ containing all parameters of $\psi$, $\mathbb{P}_\kappa$, and a name for every element of $\sigma'^{-1}"X$, which is transitive below $\kappa$ and closed under the map $\alpha\mapsto\mathbb{P}_\alpha$ such that $$g(\pi_Z(\kappa))=(\bar{\kappa},\bar{\lambda}, \dot{N}, \dot{\bar{\mathbb{B}}}, \ddot{\bar{a}},\dot{\bar{b}},\pi_Z(p),\phi)$$
    $$:=(\pi_Z(\kappa),\pi_Z(\lambda'), \pi_Z(\dot{N}'),\pi_Z(\dot{\mathbb{B}}'), \pi_Z(\ddot{a}),\pi_Z(\dot{b}),\pi_Z(p),\phi)$$ and $\psi(\bar{\kappa},\bar{\lambda},\dot{N}, \dot{\bar{\mathbb{B}}}, \ddot{\bar{a}},\dot{\bar{b}},\pi_Z(p),\phi)$ holds in $V_\kappa^{V[f]}$ (and $V[f]$).

    By the definition of $n$-niceness and the fact that $\kappa$ is inaccessible, $\mathbb{P}_\kappa$ is the direct limit of $\langle\mathbb{P}_\alpha\sbp \alpha<\kappa\rangle$ (and this fact is absolute to $H_\lambda^{V[f]}$), so by elementarity and the fact that $Z$ contains $\mathbb{P}_\alpha$ for all $\alpha<\bar{\kappa}=\kappa\cap Z$, $\pi_Z(\mathbb{P}_\kappa)$ is the direct limit of $\langle\mathbb{P}_\alpha\sbp \alpha<\bar{\kappa}\rangle$. Since $\bar{\kappa}$ is inaccessible, this is exactly $\mathbb{P}_{\bar{\kappa}}$. Since $p\in\mathbb{P}_\kappa\subset V_\kappa^{V[f]}$, $p$ is in the transitive part of $Z$ and thus $\pi_Z(p)=p$. Then again by elementarity, $p\in\mathbb{P}_{\bar{\kappa}}$ and $\dot{N}$, $\dot{\bar{\mathbb{B}}}$, $\ddot{\bar{a}}$, and $\dot{\bar{b}}$ are $\mathbb{P}_{\bar{\kappa}}$ names.
    
    Thus $p\in G_{\bar{\kappa}}:=G\cap\mathbb{P}_{\bar{\kappa}}$. Set $N:=\dot{N}^{G_{\bar{\kappa}}}$, $\bar{\mathbb{B}}:=\dot{\bar{\mathbb{B}}}^{G_{\bar{\kappa}}}$, $\dot{\bar{a}}:=\ddot{\bar{a}}^{G_{\bar{\kappa}}}$, and $\bar{b}:=\dot{\bar{b}}^{G_{\bar{\kappa}}}$. By the reflected version of $\psi$, there is a poset in $V_\kappa^{V[f][G_{\bar{\kappa}}]}\cap\Gamma$ which adds a $<\hspace{-2pt}\bar{\lambda}$-weakly $N$-generic filter $F\subset\bar{\mathbb{B}}$ such that $\phi(\dot{\bar{a}}^F,\bar{b})$ holds. By the construction of $\langle\dot{\mathbb{Q}}_\alpha\sbp \alpha<\kappa\rangle$ from $g$, $\dot{\mathbb{Q}}_{\bar{\kappa}}$ is a name for such a poset, so such an $F$ exists in $V[f][G_{\bar{\kappa}+1}]$. Since $\mathbb{P}_\kappa/\mathbb{P}_{\bar{\kappa}+1}\in \Gamma$, it preserves $\phi$, and all forcing preserves the existence of weakly $N$-generic filters, so we have such an $F$ in $V[f][G]$.

    It remains only to show that $V[f][G]$ contains an elementary embedding $\sigma:N\rightarrow H_\gamma$ which maps everything where we want. Define $\bar{\sigma}:N\rightarrow N'$ by, whenever $\dot{x}$ is a $\mathbb{P}_{\bar{\kappa}}$-name for an element of $N$, $\bar{\sigma}(\dot{x}^{G_{\bar{\kappa}}}):=\pi_Z^{-1}(\dot{x})^G$. To see that this is a well-defined elementary embedding, consider any formula $\chi$ and names for elements of $N$ $\dot{x}_1,\dotsc, \dot{x}_m$ such that $N\models \chi(\dot{x}_1^{G_{\bar{\kappa}}},\dotsc, \dot{x}_n^{G_{\bar{\kappa}}})$. Then there is some $q\in G_{\bar{\kappa}}$ such that $q\Vdash_{\mathbb{P}_{\bar{\kappa}}} \pi_Z(\dot{N}')\models \chi(\dot{x}_1,\dotsc, \dot{x}_m)$ is true in the transitive collapse of $Z$, so $q\Vdash_{\mathbb{P}_\kappa} \dot{N}'\models\chi(\pi_Z^{-1}(\dot{x}_1),\dotsc, \pi_Z^{-1}(\dot{x}_m))$ and hence $N'\models \chi(\pi_Z^{-1}(\dot{x}_1)^G,\dotsc, \pi_Z^{-1}(\dot{x}_m)^G)$.

    We can now define $\sigma:=\sigma'\circ\bar{\sigma}$. Then $\sigma$ is elementary $N\rightarrow H_\gamma^{V[f][G]}$ and we verify:

    $$\sigma(\bar{\lambda})=\sigma'(\bar{\sigma}(\bar{\lambda}))=\sigma'(\bar{\sigma}(\pi_Z(\check{\lambda}')^{G_{\bar{\kappa}}}))=\sigma'(\check{\lambda}'^G)=\sigma'(\lambda')=\lambda$$

    $$\sigma(\bar{\mathbb{B}})=\sigma'(\bar{\sigma}(\pi_Z(\dot{\mathbb{B}}')^{G_{\bar{\kappa}}}))=\sigma'(\dot{\mathbb{B}}'^G)=\mathbb{B}$$

    $$\sigma(\dot{\bar{a}})=\sigma'(\bar{\sigma}(\pi_Z(\ddot{a})^{G_{\bar{\kappa}}}))=\sigma'(\ddot{a}^G)=\sigma'(\dot{a})=\dot{a}$$

    $$\sigma(\bar{b})=\sigma'(\bar{\sigma}(\pi_Z(\dot{b})^{G_{\bar{\kappa}}}))=\sigma'(\dot{b}^G)=\sigma'(b)=b$$

    To see that $rng(\sigma)\cap\kappa$ is transitive, first note that  $\bar{\sigma}(\bar{\kappa})=\pi_Z^{-1}(\bar{\kappa})=\kappa$, while for $\alpha<\bar{\kappa}$, since $Z\cap V_\kappa$ was taken to be transitive, $\bar{\sigma}(\alpha)=\alpha$. Since $crit(\sigma')>\kappa$, $\sigma$ fixes all $\alpha<\bar{\kappa}$ while sending $\bar{\kappa}$ to $\kappa$, as desired.

    Finally, for any $x\in X$, $Z$ was chosen to contain some $\mathbb{P}_\kappa$-name $\dot{y}$ such that $x=\sigma'(\dot{y}^G)$. Observe that there must be some $q\in G_{\bar{\kappa}}$ which decides whether $\pi_Z(\dot{y})\in \dot{N}$; since $q\in V_\kappa^{V[f]}$ and thus it is not moved by $\pi_Z$, by elementarity it must decide $\dot{y}\in\dot{N}'$ the same way. Because $G_{\bar{\kappa}}\subset G$ and $\dot{y}^G\in N'$, we must have $q\Vdash \pi_Z(\dot{y})\in\dot{N}$, so $\pi_Z(\dot{y})^{G_{\bar{\kappa}}}\in N$. As 
    $$\sigma(\pi_Z(\dot{y})^{G_{\bar{\kappa}}})=\sigma'(\bar{\sigma}(\pi_Z(\dot{y})^{G_{\bar{\kappa}}}))=\sigma'(\dot{y}^G)=x$$
    it follows that $X\subseteq rng(\sigma)$. Thus $\sigma$ has all the desired properties and $V[f][G]\models \Sigma_n\mhyphen CBFA_{<\kappa}^{<\lambda}(\Gamma)$.
\end{proof}

The construction of $Y$ in the above proof makes essential use of the regularity of $\lambda$; it is unclear whether the consistency of $\Sigma_n\mhyphen CBFA_{<\kappa}^{<\lambda}(\Gamma)$ follows from the existence of a $\Sigma_n$-correct $H_\lambda$-reflecting cardinal when $\lambda$ is singular.

\section{Equiconsistency Results}
\label{section:equicon}

In this section, we derive large cardinal properties of $\kappa$ in $L$ from $\Sigma_n\mhyphen CBFA_{<\kappa}^{<\lambda}(\Gamma)$. Asper\'o (\cite{asperomax}, Lemma 2.3) and Hamkins (\cite{hamkinsMP}, Lemma 3.2) proved results along these lines, making essential use of collapse forcings. These arguments do not work for cardinal-preserving classes like ccc forcing, but for these we can achieve similar results by enlarging power sets rather than collapsing cardinals. The following definition unifies these two styles of argument:

\begin{definition}
    Given a forcing class $\Gamma$, a formula $\psi$ is a $\Gamma$-inflatable definition of a cardinal $\kappa$ iff $\psi(\kappa)$ holds and ZFC proves:
    \begin{enumerate}
        \item there is a unique ordinal $\alpha$ such that $\psi(\alpha)$ holds
        \item for every ordinal $\alpha$, there exists a $\beta>\alpha$ and a poset $\mathbb{P}\in\Gamma$ such that $\Vdash_{\mathbb{P}}\psi(\beta)$
        \item if $\psi(\alpha)$ holds and $\psi(\beta)$ is $\Gamma$-forceable, then $\alpha\leq\beta$
    \end{enumerate}
\end{definition}

Thus for example if $\Gamma$ can collapse arbitrarily large cardinals to $\omega_1$, "a cardinal with exactly two smaller infinite cardinals" is a $\Gamma$-inflatable definition of $\omega_2$. Alternatively, if $\Gamma$ preserves cardinals but can add arbitrarily many reals, then "the cardinality of the continuum" is a $\Gamma$-inflatable definition of $(2^{\aleph_0})^V$.

Note that a $\Gamma$-inflatable definition can never be $\Sigma_1$ or $\Pi_1$, since if any formula of those complexities provably defines a unique object, then by upward or downward absoluteness it is impossible to change which object it defines by forcing.

\begin{lemma}
\label{lemma:symequicon}
    (Generalization of Asper\'o's Lemma 2.3 \cite{asperomax}) If $\Gamma$ is a $\Sigma_n$-definable forcing class and $\kappa$ is a regular cardinal such that $\Sigma_n\mhyphen CBFA_{<\kappa}^{<\kappa}(\Gamma)$ holds, and there is a $\Gamma$-inflatable $\Sigma_n$ definition of $\kappa$ (so $n\geq 2$), then $\kappa$ is a regular $\Sigma_n$-correct cardinal in $L$.
\end{lemma}
\begin{proof}
    It is sufficient to shows that $L_\kappa \prec_{\Sigma_n} L$. For this, we use a version of the Tarski-Vaught criterion. Suppose that $a\in L_\kappa$ and $L\models\exists y\hspace{3pt} \phi(a, y)$, where $\phi$ is a $\Pi_{n-1}$ formula. If $\psi$ is a $\Gamma$-inflatable $\Sigma_n$ definition of $\kappa$, then the formula:
    $$\exists \theta \hspace{3pt}\exists \beta<\theta\hspace{3pt}\exists b\in L_\beta \hspace{3pt} (\psi(\theta)\land L\models \phi(a, b))$$
    is $\Sigma_n$ and $\Gamma$-forceable, since for any $\beta$ containing a witness $b$ of $L\models\exists y\hspace{3pt} \phi(a, y)$ we can enlarge the ordinal defined by $\psi$ to be larger than $\beta$. Furthermore, ZFC proves that forcing preserves truth in $L$ and preserves or enlarges the ordinal defined by $\psi$, so this formula is provably $\Gamma$-persistent.

    Applying the $\Sigma_n$-correct bounded forcing axiom to the above formula with $X=trcl(\{a\})$, we get an elementary embedding $\sigma: N\rightarrow H_\gamma$ for some transitive $N$ and regular $\gamma$ such that the formula already holds in $V$ (as in this situation we will have $\sigma^{-1}(a)=a$). Since $\kappa$ is the unique ordinal such that $V\models \psi(\kappa)$, there is some $\beta<\kappa$ with a $b\in L_\beta$ such that $L\models \phi(a,b)$. Thus $b\in L_\kappa$, so $L_\kappa\prec_{\Sigma_n} L$ by the Tarski-Vaught criterion.
\end{proof}

Miyamoto proved that $BPFA^{<\omega_3}$ is equiconsistent with a +1-reflecting cardinal (\cite{miyamotosegments}, Theorem 4.2); Fuchs showed that the same holds for $BSCFA^{<\omega_3}$ (\cite{fuchshierachies}, Lemma 3.10). We can prove an analogue of these results in the $\Sigma_n$-correct setting:

\begin{lemma}
    \label{lemma:+1equicon} If $\Gamma$ is a $\Sigma_n$-definable forcing class and $\kappa$ is a regular cardinal such that $\Sigma_n\mhyphen CBFA_{<\kappa}^{<\kappa^+}(\Gamma)$ holds, and there is a $\Gamma$-inflatable $\Sigma_n$ definition of $\kappa$, then $\kappa$ is $\Sigma_n$-correctly +1-reflecting in $L$.
\end{lemma}
\begin{proof}
    By Lemma \ref{lemma:alt+1ref} it is sufficient to consider some $A\in\mathcal{P}(\kappa)\cap L$ and $\Sigma_n$ formula $\phi$ such that $L\models\phi(A, \kappa)$ and show that for stationarily many $\alpha<\kappa$, $L_\kappa\models\phi(A\cap\alpha, \alpha)$. Fix a club $C\subset\kappa$; we will find a $\bar{\kappa}\in C$ such that $L_\kappa\models\phi(A\cap\bar{\kappa}, \bar{\kappa})$.

    We apply the M-A formulation of the axiom to the formula $L\models \phi(A, \kappa)$, which is clearly $\Sigma_n$, $\Gamma$-forceable, and provably $\Gamma$-persistent. This yields stationarily many $Z\prec H_{\kappa^+}$ such that $L\models\phi(\pi_Z(A), \pi_Z(\kappa))$, so by Lemma \ref{lemma:clubequiv} we can find one with $\bar{\kappa}:=\kappa\cap Z=\pi_Z(\kappa)\in C$. Then $\pi_Z$ does not move any ordinals less than $\kappa$, so $\pi_Z(A)=A\cap\bar{\kappa}$.

Thus we have obtained a $\bar{\kappa}\in C$ such that $L\models \phi(A\cap\bar{\kappa}, \bar{\kappa})$. Applying Lemma \ref{lemma:symequicon}, we get that $L_\kappa\models \phi(A\cap\bar{\kappa}, \bar{\kappa})$, so $\kappa$ is $\Sigma_n$-correctly $+1$-reflecting in $L$.
\end{proof}

By Proposition \ref{prop:+2ref} and the preceding discussion, +2-reflecting cardinals are too strong to be accommodated by any currently understood canonical inner model, so obtaining good consistency strength lower bounds for $\Sigma_n$-correct bounded forcing axioms when $\lambda\geq \kappa^{++}$ will most likely depend on further developments in inner model theory.

To summarize:
\begin{theorem}
    \label{thm:cbfaequicon}
    Suppose $\Gamma$ is an $n$-nice forcing class with a $\Gamma$-inflatable $\Sigma_n$ definition $\psi$ of a cardinal such that for $\mathbb{P}_\kappa$ as in Theorems \ref{thm:symbfa} or \ref{thm:asymbfa}, $\Vdash_{\mathbb{P}_\kappa} \psi(\kappa)$. Then (1) and (2) are equiconsistent over ZFC, as are (3) and (4):
    \begin{enumerate}[(1)]
        \item There is a regular $\Sigma_n$-correct cardinal
        \item $\Sigma_n\mhyphen CBFA_{<\kappa}^{<\kappa}(\Gamma)\land\psi(\kappa)$
        \item There is a $\Sigma_n$-correctly $+1$-reflecting cardinal
        \item $\Sigma_n\mhyphen CBFA_{<\kappa}^{<\kappa^+}(\Gamma)\land\psi(\kappa)$
    \end{enumerate}
\end{theorem}

As a corollary, we obtain much cleaner hierarchy results in $n$ than seemed to be possible for unbounded $\Sigma_n$-correct forcing axioms in Section \ref{section:hierarchy}:

\begin{corollary}
For any positive integer $n$ and $\Gamma$ as in the previous theorem:
    \begin{enumerate}
        \item $\Sigma_n\mhyphen CBFA_{<\kappa}^{<\kappa}(\Gamma)$ (if it is consistent) does not imply $\Sigma_{n+1}\mhyphen CBFA_{<\kappa}^{<\kappa}(\Gamma)$
        \item $\Sigma_n\mhyphen CBFA_{<\kappa}^{<\kappa^+}(\Gamma)$ (if it is consistent) does not imply $\Sigma_{n+1}\mhyphen CBFA_{<\kappa}^{<\kappa^+}(\Gamma)$
    \end{enumerate}
\end{corollary}
\begin{proof}
    In either case, if the $\Sigma_{n+1}$-correct bounded forcing axiom holds, Theorem \ref{thm:cbfaequicon} implies that $\kappa$ is $\Sigma_{n+1}$-correctly $+i$-reflecting in $L$, where $i$ is 0 or 1 as appropriate. By the inaccessibility of correctly reflecting cardinals and Proposition \ref{prop:refhierarchy}, $L_\kappa$ is thus a model of $ZFC$ with many $\Sigma_n$-correctly $+i$-reflecting cardinals, so it has a forcing extension which is a transitive model of the appropriate $\Sigma_n$-correct bounded forcing axiom. It follows that the $\Sigma_{n+1}$-correct bounded forcing axiom is of significantly greater consistency strength, so in particular it is not implied by the $\Sigma_n$-correct bounded forcing axiom.
\end{proof}

\chapter{Preservation Theorems}
\label{section:preservation}

In this chapter, we study the preservation of $\Sigma_n$-correct forcing axioms under forcing. Sean Cox proved a theorem on the preservation of classical forcing axioms and their "plus versions" in terms of lifting generic elementary embeddings (Theorem 20 of \cite{coxfa}), unifying various preservation results that had been published previously and providing a basis for the approach taken here. 

However, in some cases only a fragment of the original axiom will survive in the forcing extension, particularly since Cox's result makes some use of the fact that the property of preserving the stationarity of subsets of $\omega_1$ is inherited by regular suborders of a forcing poset, which need not hold if we replace stationarity with a property of complexity greater than $\Pi_1$. To account for this possibility we define those fragments of $\Sigma_n$-correct forcing axioms below (using the Woodin-Cox formulation of Definition \ref{def:wcform} for convenience).

\begin{definition}
    For $\Gamma$ a forcing class, $\kappa$ a regular uncountable cardinal, and $\Phi$ a collection of provably $\Gamma$-persistent formulas, the $\Phi$-correct forcing axiom $\Phi\mhyphen CFA_{<\kappa}(\Gamma)$ is the assertion that for every $\phi\in\Phi$, $\mathbb{Q}\in \Gamma$, $\mathbb{Q}$-name $\dot{a}$ such that $\Vdash_{\mathbb{Q}}\phi(\dot{a})$, and regular $\gamma>|\mathcal{P}(\mathbb{Q})\cup trcl(\dot{a})\cup\kappa|$, there is a generic elementary embedding (with possibly ill-founded codomain) which witnesses the $\lk$-forcing axiom for $(\mathbb{Q},\phi, \dot{a}, \gamma)$ (recall Definition \ref{def:witnessfa}).
\end{definition}

To justify the original $\Sigma_n\mhyphen CFA$ notation, we can say that when $\Phi$ contains some formulas which are not provably $\Gamma$-persistent, $\Phi\mhyphen CFA_{<\kappa}(\Gamma)$ uses only those formulas in $\Phi$ which are provably $\Gamma$-persistent.

\begin{theorem}
    Let $\Gamma$ be a forcing class closed under restrictions, $\Phi$ a set of provably $\Gamma$-persistent formulas in the language of set theory, $\kappa$ a regular cardinal such that $\Phi\mhyphen CFA_{<\kappa}(\Gamma)$ holds, and $\mathbb{P}$ a poset. Define $\Phi'$ to be the set of formulas $\phi$ such that:
    \begin{itemize}
        \item $\phi\in \Phi$
        \item $\mathbb{P}$ preserves $\phi$ (note that this preservation need not be provable, only true)
        \item for every $\mathbb{P}$-name $\dot{\mathbb{Q}}$ for a poset in $\Gamma$, and every $\mathbb{P}*\dot{\mathbb{Q}}$-name $\dot{a}$ such that $\Vdash_{\mathbb{P}*\dot{\mathbb{Q}}} \phi(\dot{a})$, there is a $\mathbb{P}*\dot{\mathbb{Q}}$-name $\dot{\mathbb{R}}$ for a poset (possibly depending on $\phi$, $\dot{\mathbb{Q}}$, and $\dot{a}$) such that:
    \begin{enumerate}
        \item $\mathbb{P}*\dot{\mathbb{Q}}*\dot{\mathbb{R}}\in \Gamma$
        \item $\mathbb{P}*\dot{\mathbb{Q}}$ forces that $\dot{\mathbb{R}}$ preserves $\phi(\dot{a})$
        \item If $\theta>|\mathcal{P}(\mathbb{P}*\dot{\mathbb{Q}}*\dot{\mathbb{R}})|+|trcl(\dot{a})|$ is a regular cardinal and $j: V\rightarrow M$ is a generic elementary embedding which witnesses the $<\kappa$-forcing axiom for $(\mathbb{P}*\dot{\mathbb{Q}}*\dot{\mathbb{R}}, \phi, \dot{a}, \theta)$ (where we interpret $\dot{a}$ as a $\mathbb{P}*\dot{\mathbb{Q}}*\dot{\mathbb{R}}$-name in the obvious way) with $V$-generic filter $G*H*K\subseteq \mathbb{P}*\dot{\mathbb{Q}}*\dot{\mathbb{R}}$ in $M$, then $M\models\text{"} j"G\text{ has a lower bound in } j(\mathbb{P})\text{"}$.
    \end{enumerate}
    \end{itemize}
    Then $\Vdash_\mathbb{P} \Phi'\mhyphen CFA_{<\kappa}(\Gamma)$.
      \label{thm:pres}
\end{theorem}

(Note the $\emptyset$-correct forcing axiom is vacuously true, so cases where no such $\dot{\mathbb{R}}$ can be found for any $\phi$ are of no concern.)

\begin{proof}
    Following Cox, we apply the axiom in $V$ to obtain a generic elementary embedding $j:V\rightarrow M$, then show that $j$ lifts to a generic embedding $V[G]\rightarrow M[G']$ for some $M$-generic $G'\subseteq j(\mathbb{P})$ which witnesses the desired instance of the axiom in $V[G]$.

    More specifically, given any $\mathbb{P}$, $\dot{\mathbb{Q}}$, provably $\Gamma$-persistent $\phi\in \Phi'$, and $\dot{a}$ satisfying the standard hypotheses, let $\dot{\mathbb{R}}$ satisfy conditions 1-3 used to define $\Phi'$ in the statement of the theorem. Fix regular $\theta>\kappa$ so that all objects mentioned so far and their power sets are in $H_\theta$; then there is an elementary embedding $j:V\rightarrow M$ in some generic extension $W$ of $V$ such that:

    \begin{enumerate}[(i)]
        \item the wellfounded part of $M$ is transitive and contains $H_\theta^V$
        \item $|H_\theta^V|^M<j(\kappa)$
        \item $j\upharpoonright H_\theta^V\in M$
        \item $crit(j)=\kappa$
        \item $M$ contains a $V$-generic filter $G*H*K\subseteq \mathbb{P}*\dot{\mathbb{Q}}*\dot{\mathbb{R}}$ such that $M\models \phi(\dot{a}^{G*H})$
    \end{enumerate}
    By (iii), (v), and the fact that $\mathbb{P}\subset H_\theta^V$, $j"G\in M$. By condition (3) in the statement of the theorem, $M$ thinks that there is some $p'\in j(\mathbb{P})$ which is below all conditions in $j"G$. Then if $\mathbb{B}\in W$ is the set of $\in^M$-predecessors of the Boolean completion of $j(\mathbb{P})$ in $M$ with the partial ordering inherited from $M$ and $G'\subset\mathbb{B}$ is a $W$-generic ultrafilter containing $p'$, it will also be $M$-generic, and in $W[G']$ we can form the (possibly ill-founded) class model $M[G']$ as the quotient $M^\mathbb{B}/G'$. See Cox for more details on this process.

    Since $j"G\subset G'$ by the upward closure of filters, we can (in $W[G']$, hence in a generic extension of $V[G]$) apply Lemma \ref{lemma:extembed} to define an elementary embedding $j^*:V[G]\rightarrow M[G']$ by $j^*(\dot{x}^G)=j(\dot{x})^{G'}$. We now verify that $j^*$ satisfies the conditions in the definition of witnessing the $<\hspace{-3pt}\kappa$-forcing axiom for $(\dot{\mathbb{Q}}^G, \phi, \dot{a}, \theta)$ in $V[G]$:

    (a): Since $M$ contains $H_\theta^V$ and $G$, it (and thus $M[G']$) contains $H_\theta^V[G]$, which by Lemma \ref{lemma:namesize} is $H_\theta^{V[G]}$.

    (b): $|H_\theta^V[G]|^M=|H_\theta^V|^M<j(\kappa)=j^*(\kappa)$; thus it is sufficient to show that $M$ believes that $j(\mathbb{P})$ does not collapse $j(\kappa)$. By elementarity, this is the same as showing that $V\models\text{"} \mathbb{P}$ does not collapse $\kappa$". Assume toward a contradiction that there is a $\delta<\kappa$ and a $\mathbb{P}$-name $\dot{f}$ such that $\dot{f}^G:\delta\rightarrow\kappa$ is surjective. For each $\alpha<\delta$, fix $p_\alpha\in G$ which decides $\dot{f}(\alpha)$. Since $crit(j)=\kappa>\delta$, $j$ does not move any of these $\alpha$s, so $j(p_\alpha)$ decides $j(\dot{f})(\alpha)$. Since $p'\leq j(p_\alpha)$ for all $\alpha<\delta$, $p'$ alone decides all values of $j(\dot{f})$, so $j(\dot{f})^{G'}\in M$. Therefore $j(\kappa)$ is not a cardinal in $M$, so $\kappa$ is not a cardinal in $V$, contradicting the fact that $V\models\Phi\mhyphen CFA_{<\kappa}(\Gamma)$.

    (c): Since $j\upharpoonright H_\theta^V, G'\in M[G']$, $j^*$ is defined entirely in terms of $G'$ and the action of $j$ on names, and by Lemma \ref{lemma:namesize} every element of $H_\theta^{V[G]}$ has a name in $H_\theta^V$, $j^*\upharpoonright H_\theta^{V[G]}\in M[G']$.

    (d): $j^*$ is an extension of $j$, so in particular it agrees with $j$ on the ordinals.

    (e): $H\in M\subset M[G']$ and $H$ is $V[G]$-generic by the choice of $M$; $M[G']\models\phi(\dot{a}^H)$ because this holds in $M$ and $M$ thinks that  $j(\mathbb{P})$ preserves $\phi$ by elementarity.
\end{proof}

\begin{corollary}
\label{cor:fullpres}
    Suppose $\Gamma$ is an $n$-nice forcing class, $V\models\Sigma_n\mhyphen CFA_{<\kappa}(\Gamma)$, $\mathbb{P}\in \Gamma$, and for all provably $\Gamma$-persistent $\Sigma_n$ formulas $\phi$, $\mathbb{P}$-names $\dot{\mathbb{Q}}$ for a poset in $\Gamma$, and $\mathbb{P}*\dot{\mathbb{Q}}$-names $\dot{a}$ such that $\Vdash_{\mathbb{P}*\dot{\mathbb{Q}}} \phi(\dot{a})$, there is a $\mathbb{P}*\dot{\mathbb{Q}}$-name $\dot{\mathbb{R}}$ for a poset in $\Gamma$ such that condition 3 from the statement of Theorem \ref{thm:pres} holds. Then $\Vdash_\mathbb{P} \Sigma_n\mhyphen CFA_{<\kappa}(\Gamma)$.
\end{corollary}
\begin{proof}
    If $\mathbb{P}\in \Gamma$ and $\Vdash_{\mathbb{P}*\dot{\mathbb{Q}}} \dot{\mathbb{R}}\in\Gamma$, then $\mathbb{P}$ and $\mathbb{R}$ will preserve all provably $\Gamma$-persistent formulas, so we can take $\Phi'$ to consist of all provably $\Gamma$-persistent $\Sigma_n$ formulas.
\end{proof}

The following result generalizes Cox's observation about inheritance of stationary set preservation:

\begin{prop}
    In Theorem \ref{thm:pres}, if $\Phi$ consists exclusively of $\Pi_1$ formulas, then the hypothesis that $\mathbb{P}$ preserves $\phi$ can be omitted from the definition of $\Phi'$, since it will follow from the other conditions.
\end{prop}

\begin{proof}
    If $\phi\in \Phi$ is a provably $\Gamma$-persistent $\Pi_1$ formula such that $V\models \exists x\hspace{2pt}(\phi(x)\land \exists p\in\mathbb{P}\hspace{4pt} p\Vdash_\mathbb{P}\lnot\phi(x))$, then the failure of $\phi(x)$ in any generic extension by a filter containing $p$ is a $\Sigma_1$ statement. It follows that it is preserved by passing to larger models, so $(p, \dot{1}_{\dot{\mathbb{Q}}}, \dot{1}_{\dot{\mathbb{R}}})\Vdash \lnot\phi(x)$. Since $\phi$ is provably $\Gamma$-persistent, this contradicts the fact that $\mathbb{P}*\dot{\mathbb{Q}}*\dot{\mathbb{R}}\in \Gamma$.
\end{proof}

To illustrate the application of Theorem \ref{thm:pres} and procure a useful tool for later separation results, we consider the cases of proper and countably closed forcing:

\begin{theorem}
    \label{thm:properpres}
    (cf K\"onig and Yoshinobu \cite{KYfragMM}, Theorem 6.1) If $\Gamma$ is an $n$-nice forcing class which necessarily contains $Coll(\omega_1, \lambda)$ for all uncountable cardinals $\lambda$ and has the countable covering property (i.e. every countable set in a $\Gamma$-extension whose elements are all in the ground model is a subset of a countable set of the ground model), then $\Sigma_n\mhyphen CFA_{<\omega_2}(\Gamma)$ is preserved by all $<\hspace{-2pt}\omega_2$-closed forcing in $\Gamma$. In particular, $\Sigma_n\mhyphen CPFA$ and $\Sigma_n\mhyphen CFA_{<\omega_2}(<\hspace{-2pt}\omega_1\mhyphen closed)$ are preserved by all $<\hspace{-2pt}\omega_2$-closed forcing.
\end{theorem}
\begin{proof}
    Let $\mathbb{P}\in \Gamma$ be $<\hspace{-2pt}\omega_2$-closed. Given a $\mathbb{P}$-name $\dot{\mathbb{Q}}$ such that $\Vdash_{\mathbb{P}} \dot{\mathbb{Q}}\in \Gamma$, a provably $\Gamma$-persistent $\Sigma_n$ formula $\phi$, and $\mathbb{P}*\dot{\mathbb{Q}}$-name $\dot{a}$ such that $\Vdash_{\mathbb{P}*\dot{\mathbb{Q}}} \phi(\dot{a})$, let $\dot{\mathbb{R}}$ be a $\mathbb{P}*\dot{\mathbb{Q}}$-name such that $\Vdash_{\mathbb{P}*\dot{\mathbb{Q}}}\dot{\mathbb{R}}=Coll(\omega_1, |\dot{G}|)$, where $\dot{G}$ is the canonical name for the generic filter on $\mathbb{P}$. Then $\Vdash_{\mathbb{P}*\dot{\mathbb{Q}}}\dot{\mathbb{R}}\in\Gamma$, so it is sufficient to verify the lower bound of $j"G$ condition.

    To that end, fix any sufficiently large $\theta$ and let $j:V\rightarrow M$ be a generic elementary embedding such that:
    \begin{enumerate}[(i)]
        \item the wellfounded part of $M$ is transitive and contains $H_\theta^V$
        \item $|H_\theta^V|^M<j(\omega_2^V)=\omega_2^M$
        \item $j\upharpoonright H_\theta^V\in M$
        \item $crit(j)=\omega_2^V$
        \item $M$ contains a $V$-generic filter $G*H*K\subseteq \mathbb{P}*\dot{\mathbb{Q}}*\dot{\mathbb{R}}$ such that $M\models \phi(\dot{a}^{G*H})$
    \end{enumerate}
    By our choice of $\dot{\mathbb{R}}$, in $V[G][H][K]$ there is an enumeration $\langle g_\alpha\sbp \alpha<\omega_1\rangle$ of $G$. We will use this to recursively construct a descending sequence $\langle x_\alpha\sbp \alpha<\omega_1\rangle$ of elements of $G$ such that $x_\alpha\leq_{\mathbb{P}} g_\beta$ whenever $\beta<\alpha$.

    Let $x_0=1_\mathbb{P}$. If $x_\alpha$ has already been defined, then $x_\alpha$ and $g_\alpha$ have a common lower bound in $G$ because both are elements of $G$; let $x_{\alpha+1}$ be some such lower bound. At limit stages $\gamma<\omega_1$ where $x_\alpha$ has already been defined for all $\alpha<\gamma$, then since $\mathbb{P}*\dot{\mathbb{Q}}*\dot{\mathbb{R}}\in \Gamma$ and $\Gamma$ has the countable covering property, there is some countable $Y\in V$ such that $\{ x_\alpha\sbp \alpha<\gamma\}\subseteq Y$. Replacing $Y$ with $Y\cap\mathbb{P}$ if necessary, we may assume that $Y\subseteq \mathbb{P}$.

    Working now in $V$, fix an enumeration $\langle y_n\sbp n<\omega\rangle$ of $Y$ and let $D:=\{d\in\mathbb{P}\sbp\forall y\in Y\hspace{3pt} (y\perp d\lor d\leq y)\}$. Then for any $p_0\in\mathbb{P}$, we can recursively choose for each $n<\omega$ a $p_{n+1}\leq p_n$ such that $p_{n+1}\perp y_n$ or $p_{n+1}\leq y_n$. Since $\mathbb{P}$ is $<\omega_2$-closed and hence countably closed, $\langle p_n\sbp n<\omega\rangle$ has a lower bound in $\mathbb{P}$, which must be an element of $D$. It follows that $D$ is dense.

    Returning to $V[G][H][K]$, we can choose $x_\gamma\in G\cap D$. It cannot be the case that $x_\alpha\perp x_\gamma$ for any $\alpha<\gamma$ because they are all elements of the same filter $G$, so by the definitions of $D$ and $Y$ we must have $x_\gamma\leq x_\alpha$ for all such $\alpha$. This completes the construction of $\langle x_\alpha\sbp \alpha<\omega_1\rangle$.

    Let $\dot{x}$ be a name for $\langle x_\alpha\sbp \alpha<\omega_1\rangle$ in $H_\theta^V$; this is possible by Lemma \ref{lemma:namesize} if $\theta$ was chosen to be sufficiently larger than everything of interest. Then since $\dot{x}, G, j\upharpoonright H_\theta^V \in M$, $\langle j(x_\alpha)\sbp \alpha<\omega_1\rangle \in M$. By elementarity, $j\upharpoonright\mathbb{P}$ is order-preserving and $M$ thinks that $j(\mathbb{P})$ is $<\hspace{-2pt}\omega_2$-closed. Thus $j(x_\alpha)\leq_{j(\mathbb{P})} j(g_\beta)$ whenever $\beta<\alpha$ and there is a $p'\in j(\mathbb{P})$ less than all $j(x_\alpha)$, so in particular $p'$ is a lower bound of $j"G$. The result follows from Corollary \ref{cor:fullpres}.
\end{proof}

By strengthening the closure property on $\mathbb{P}$, we can generalize this to other forcing classes:

\begin{theorem}
    \label{thm:dirclosepres}
    (cf Larson \cite{larson}, Theorem 4.3) If $\Gamma$ is any $n$-nice forcing class, and $\mathbb{P}\in\Gamma$ is $\lk$-directed closed (i.e., for all $X\subset\mathbb{P}$ of size less than $\kappa$ such that any two elements of $X$ have a common lower bound in $X$, there is a common lower bound for all elements of $X$ in $\mathbb{P}$), then $\mathbb{P}$ preserves $\Sigma_n\mhyphen CFA_{<\kappa}(\Gamma)$.
\end{theorem}
\begin{proof}
    Given a $\lk$-directed closed $\mathbb{P}\in \Gamma$, $\mathbb{P}$-name $\dot{\mathbb{Q}}$ such that $\Vdash_{\mathbb{P}} \dot{\mathbb{Q}}\in \Gamma$, a provably $\Gamma$-persistent $\Sigma_n$ formula $\phi$, and $\mathbb{P}*\dot{\mathbb{Q}}$-name $\dot{a}$ such that $\Vdash_{\mathbb{P}*\dot{\mathbb{Q}}} \phi(\dot{a})$, let $\dot{\mathbb{R}}$ be a $\mathbb{P}*\dot{\mathbb{Q}}$-name for the trivial forcing, so that $\mathbb{P}*\dot{\mathbb{Q}}*\dot{\mathbb{R}}\equiv \mathbb{P}*\dot{\mathbb{Q}}$. Then if $\theta>|\mathcal{P}(\mathbb{P}*\dot{\mathbb{Q}})|$ is regular and $j:V\rightarrow M$ witnesses the $\lk$-forcing axiom for $(\mathbb{P}*\dot{\mathbb{Q}}, \phi, \dot{a}, \theta)$ with generic filter $G*H$, $M$ thinks that $j"G$ is directed and $|j"G|^M<j(\kappa)$, since $G\subset H_\theta^V$ and $H_\theta^V$ has $M$-cardinality less than $j(\kappa)$. Since by elementarity $M$ thinks that $j(\mathbb{P})$ is $<j(\kappa)$-directed closed, there is a common lower bound for $j"G$ in $j(\mathbb{P})$. Applying Corollary \ref{cor:fullpres} yields the desired result.
\end{proof}

\chapter{Relationships with Other Axioms}

This chapter explores the relations $\Sigma_n$-correct forcing axioms have to various previously-proposed axioms, as well as the relationship between the $\Sigma_n$-correct forcing axioms for different forcing classes.

\section{Equivalences}
\label{section:equivalences}

\begin{theorem}
If $\kappa$ is a regular cardinal and $n\geq 1$, $\Sigma_n\mhyphen CBFA_{<\kappa}^{<\kappa}(\Gamma)$ is equivalent to $\Sigma_n\mhyphen MP_\Gamma(H_\kappa)$.
\label{thm:symbfaMP}
\end{theorem}

\begin{proof}
The forward implication is Proposition \ref{prop:cbfaeasy}(2). For the converse, let $\phi$ be a provably $\Gamma$-persistent $\Sigma_n$ formula, $\mathbb{B}\in \Gamma$ be a complete Boolean algebra, $\dot{a}\in H_\kappa$ a $\mathbb{B}$-name, $b\in H_\kappa$, $X\subset H_\gamma$ with $|X|<\kappa$, and $\Vdash_{\mathbb{B}} \phi(\dot{a},\check{b})$. Then following the proof of Theorem \ref{thm:symbfa} we can construct a transitive $N\in H_\kappa$ with an elementary embedding $\sigma:N\rightarrow H_\gamma$ with $X\cup trcl(\{\dot{a}, b\})\cup\{\mathbb{B}, \kappa\}\subset rng(\sigma)$ and $rng(\sigma)\cap\kappa$ transitive such that whenever $A\subset\mathbb{B}$ is a maximal antichain of size less than $\kappa$ in the range of $\sigma$, all its elements are also in the range of $\sigma$.
Then $\sigma^{-1}(\dot{a})=\dot{a}$ and $\sigma^{-1}(b)=b$, and we can define $\bar{\kappa}=\sigma^{-1}(\kappa)$ and $\bar{\mathbb{B}}=\sigma^{-1}(\mathbb{B})$. As in the relative consistency proofs, if $G\subset\mathbb{B}$ is $V$-generic, then for every maximal antichain of $\bar{\mathbb{B}}$ $A\in N$ with $|A|^N<\bar{\kappa}$, $|\sigma(A)|<\kappa$, so $A\subset rng(\sigma)$ and thus $G\cap A\cap rng(\sigma)$ is non-empty. Hence $V[G]\models$ "$\sigma^{-1}"G$ is a $<\hspace{-3pt}\bar{\kappa}$-weakly $N$-generic filter on $\bar{\mathbb{B}}$ and $\phi(\dot{a}^{\sigma^{-1}"G}, b)$ holds". It follows that the statement that there exists a $<\hspace{-3pt}\bar{\kappa}$-weakly $N$-generic $F\subset \bar{\mathbb{B}}$ such that $\phi(\dot{a}^F, b)$ holds is a $\Sigma_n$ statement forced by $\mathbb{B}$ and provably preserved by further forcing in $\Gamma$ with parameters $\dot{a}, b, N, \bar{\kappa}, \bar{\mathbb{B}}\in H_\kappa$. Applying the maximality principle, it is therefore true.
\end{proof}

We now wish to show that classical forcing axioms are equivalent to $\Sigma_1$-correct forcing axioms. However, there is a slight technical wrinkle. $\Sigma_n$-correct forcing axioms were stated with the requirement that $rng(\sigma)\cap\kappa$ is transitive, but the proof of Lemma \ref{lemma:jensenfa} can be extended to guarantee this condition only in some cases. If $\kappa=\theta^+$ is a successor cardinal, we can simply assume that $\theta+1\subset X$, so that the same will hold of $rng(\sigma)$. Then by elementarity $rng(\sigma)$ correctly identifies a surjection $\theta\rightarrow \alpha$ for each $\alpha\in rng(\sigma)\cap\kappa$ and correctly computes all values of that surjection, so we must have $\alpha\subset rng(\sigma)$ and thus $rng(\sigma)\cap\kappa$ is transitive. When $\kappa$ is a limit cardinal, however, this fails.

To remedy this issue, we could remove the transitivity condition from the statement of $\Sigma_n\mhyphen CFA$, but it is occasionally useful to have that condition. Alternatively, we could redefine $FA_{<\kappa}(\Gamma)$ to mean Lemma \ref{lemma:jensenfa}(2) with the added condition that $rng(\sigma)\cap\kappa$ is transitive (in essence requiring that the collection of elementary substructures of $H_\gamma$ which have generics be stationary in $[H_\gamma]^{<\kappa}$ in Jech's sense rather than merely weakly stationary), since this condition can easily be derived in both the Martin-Solovay consistency proof for Martin's Axiom and Baumgartner-style consistency proofs for other forcing axiom. Whichever option is chosen, the issue can be resolved fairly easily, and we will not worry about it much further.

To show a level-by-level equivalence of bounded forcing axioms, it is helpful to have a bounded version of Lemma \ref{lemma:jensenfa}, proved in essentially the same way:

\begin{lemma}
    The following are equivalent for any forcing class $\Gamma$ and regular cardinals $\lambda\geq \kappa>\omega_1$:
    \begin{enumerate}
        \item $BFA_{<\kappa}^{<\lambda}(\Gamma)$ \item For any cardinal $\gamma>\lambda$, complete Boolean algebra $\mathbb{B}\in \Gamma$ such that $\mathbb{B}\in H_\gamma$, and $X\in [H_\gamma]^{<\kappa}$, there is a transitive structure $N$ and an elementary embedding $\sigma:N\rightarrow H_\gamma$ such that $|N|<\kappa$, $X\cup\{\mathbb{B},\kappa,\lambda\}\subseteq rng(\sigma)$, and there is a $<\sigma^{-1}(\lambda)$-weakly $N$-generic filter $F\subseteq \bar{\mathbb{B}}:=\sigma^{-1}(\mathbb{B})$.
    \end{enumerate}
    \label{lemma:boundedjensen}
\end{lemma}

The following result is a generalization of Bagaria's Theorem \ref{thm:bagariagenabs} to asymmetric bounded forcing axioms and of Miyamoto's Theorem 2.5 in \cite{miyamotosegments} to forcing classes other than proper forcing:

\begin{theorem}
    If $\Gamma$ is a forcing class and $\lambda\geq\kappa>\omega_1$ are regular cardinals, then $BFA^{<\lambda}_{<\kappa}(\Gamma)$ is equivalent to $\Sigma_1\mhyphen CBFA^{<\lambda}_{<\kappa}(\Gamma)$, modulo the requirement that $rng(\sigma)\cap\kappa$ is transitive.
    \label{thm:bfaeqs1cbfa}
\end{theorem}
\begin{proof}
    The reverse implication is Proposition \ref{prop:cbfaeasy}(1). For the forward direction, if $\mathbb{B}\in \Gamma$ is a complete Boolean algebra, $\dot{a}\in H_\lambda$ is a $\mathbb{B}$-name, $b\in H_\lambda$, $\gamma\geq\lambda$ is a regular cardinal, $X\subset H_\gamma$ is such that $|X|<\kappa$, and $\phi$ is a $\Delta_0$ formula such that $\llbracket\exists x\hspace{3pt}\phi(\dot{a},b, x)\rrbracket=1$, then the fact that $\mathbb{B}$ forces $\exists x\hspace{3pt}\phi(\dot{a},b, x)$ is $\Sigma_1$ and thus by Lemma \ref{lemma:HS1correct} absolute to $H_\gamma$.

    Now applying Lemma \ref{lemma:boundedjensen} to $X\cup\{\dot{a}, b\}$, there is a transitive $N$ of size less than $\kappa$ with an elementary embedding $\sigma:N\rightarrow H_\gamma$ such that $X\cup\{\dot{a}, b, \kappa, \lambda,\mathbb{B}\}\subseteq rng(\sigma)$ and, setting $\bar{\lambda}:=\sigma^{-1}(\lambda)$ and $\bar{\mathbb{B}}:=\sigma^{-1}(\mathbb{B})$, there is a $<\hspace{-2pt}\bar{\lambda}$-weakly $N$-generic filter $F\subset\bar{\mathbb{B}}$. Since $2<\bar{\lambda}$, $F$ is in particular an ultrafilter.

    By elementarity, $N\models \llbracket \exists x\hspace{3pt}\phi(\sigma^{-1}(\dot{a}),\sigma^{-1}(b), x)\rrbracket=1_{\bar{\mathbb{B}}}$. Thus by Lemma \ref{lemma:booleanlos}, $N^{\bar{\mathbb{B}}}/F\models \exists x\hspace{3pt}\phi([\sigma^{-1}(\dot{a})]_F, [\sigma^{-1}(\check{b})]_F, x)$. Since $trcl(\{\sigma^{-1}(\dot{a})\})$ and $trcl(\sigma^{-1}(\{\check{b})\})$ have $N$-cardinality less than $\bar{\lambda}$, Lemma \ref{lemma:rqieqbool} implies that $[\sigma^{-1}(\dot{a})]_F$ and $[\sigma^{-1}(\check{b})]_F$ are in the wellfounded part of $N^{\bar{\mathbb{B}}}/F$, and if the wellfounded part is taken to be transitive they are equal to $\sigma^{-1}(\dot{a})^{((F))}$ and $\sigma^{-1}(b)$ respectively. Applying Lemma \ref{lemma:standeqrqi}, $\sigma^{-1}(\dot{a})^{((F))}=\sigma^{-1}(\dot{a})^F$.

    By Lemma \ref{lemma:HS1correct}, there is some $c\in H_{\bar{\lambda}}^{N^{\bar{\mathbb{B}}}/F}$ such that $\phi(\sigma^{-1}(\dot{a})^F, \sigma^{-1}(b), c)$ holds. (Note that since $N^{\bar{\mathbb{B}}}/F$ does not in general satisfy the power set axiom, $H_{\bar{\lambda}}^{N^{\bar{\mathbb{B}}}/F}$ may be a proper class in it, but the proof of Lemma \ref{lemma:HS1correct} does not rely on it being a set.) Applying Lemma \ref{lemma:rqieqbool} again, $c$ must be in the wellfounded part of $N^{\bar{\mathbb{B}}}/F$, so $wfp(N^{\bar{\mathbb{B}}}/F)$ is a transitive structure satisfying $\phi(\sigma^{-1}(\dot{a})^F,\sigma^{-1}(b), c)$. Since $\Delta_0$ formulas are absolute between transitive classes, $V\models \exists x\hspace{3pt} \phi(\sigma^{-1}(\dot{a})^F, \sigma^{-1}(b), x)$.
\end{proof}

To conclude this section, consider the following open question:

\begin{question}
    Is $MM^{+\omega_1}$ equivalent to $\Sigma_2\mhyphen CMM$?
\end{question}

In other words, can every $\Sigma_2$ property provably preserved by all stationary set-preserving forcing be encoded in a sequence of stationary subsets of $\omega_1$? The answer if we replace $\Sigma_2$ with $\Sigma_3$ or higher is no by Proposition \ref{prop:S2nimpS3}, since $MM^{+\omega_1}$ is implied by $\Sigma_2\mhyphen CMM$ and thus cannot imply $\Sigma_3\mhyphen CMM$.

\section{Separations}
\label{section:separations}

Corey Switzer posed the question of whether $\Sigma_n\mhyphen CPFA$ can ever imply $MM$. Using the work of K\"onig and Yoshinobu on regressive Kurepa trees \cite{KYkurepatrees}, this can be answered in the negative.

\begin{definition}
    (K\"onig and Yoshinobu) For uncountable cardinals $\gamma$ and $\lambda$, a $\gamma$-regressive $\lambda$-Kurepa tree is a tree $(T, <_T)$ such that:
    \begin{itemize}
        \item the height of $T$ is $\lambda$
        \item for each $\alpha<\lambda$, the $\alpha$th level $T_\alpha$ has cardinality less than $\lambda$
        \item there are at least $\lambda^+$ distinct cofinal branches of $T$
        \item for all limit ordinals $\beta<\lambda$ of cofinality less than $\gamma$, there is a function $f:T_\beta\rightarrow T\upharpoonright\beta$ such that $f(x)<_T x$ for all $x\in T_\beta$ and whenever $x\neq y$ in $T_\beta$, $f(y)\not<_T x$ or $f(x)\not<_T y$
    \end{itemize}
\end{definition}

\begin{lemma}
    (K\"onig and Yoshinobu, Theorem 5 \cite{KYkurepatrees}) For each regular uncountable cardinal $\lambda$, there is a $<\hspace{-3pt}\lambda$-closed forcing which adds a $\lambda$-regressive $\lambda$-Kurepa tree.
\end{lemma}

\begin{lemma}
    (K\"onig and Yoshinobu, Theorem 13 \cite{KYkurepatrees}) $MM$ implies that no $\omega_2$-Kurepa tree is even $\omega_1$-regressive.
\end{lemma}

\begin{theorem}
    \label{thm:cpfanimpmm}
    For each $n$, if $\Sigma_n\mhyphen CPFA$ is consistent, it does not imply $MM$.
\end{theorem}
\begin{proof}
    Any model of $\Sigma_n\mhyphen CPFA$ has a forcing extension via a $<\hspace{-3pt}\omega_2$-closed poset to a model with a $\omega_2$-regressive $\omega_2$-Kurepa tree. By Theorem \ref{thm:properpres}, $\Sigma_n\mhyphen CPFA$ still holds in the forcing extension. By the previous lemma, however, $MM$ fails.
\end{proof}

With stronger assumptions and a geological argument somewhat reminiscent of that for Lemma \ref{lemma:noCn+1}, we can get an outright contradiction between strengthenings of PFA and MM. Let
$$\Omega_2:=\{\alpha<\omega_2\sbp cf(\alpha)=\omega\land\exists r\hspace{4pt} (r\text{ defines a ground}\land \alpha=\omega_2^{W_r})\}$$
where we think of $\Omega_2$ as a bounded definable class which varies between models rather than a fixed set. Since $cf(\alpha)=\omega$ and $\alpha=\omega_2$ are $\Sigma_1$ and $\Delta_2$ properties respectively, by Lemmas \ref{lemma:groundrel} and \ref{lemma:groundsuccess} $\Omega_2$ is $\Sigma_3$-definable.

\begin{lemma}
    \label{lemma:S3MMO2}
    $\Sigma_3$-correct Martin's Maximum implies that $\Omega_2$ is stationary in $\omega_2$. The same holds for $\Sigma_3\mhyphen CSCFA$ or $\Sigma_3\mhyphen CFA_{<\omega_2}(\Gamma)$ for any class $\Gamma$ containing Namba forcing.
\end{lemma}
\begin{proof}
   Since Namba forcing forces that $\check{\omega}_2$ has countable cofinality and is the $\omega_2$ of some ground, and these properties can easily be proven to be preserved under all further forcing, then by the Miyamoto-Asper\'o form of the axiom there are stationarily many $Z\prec H_{\omega_3}$ of size $\omega_1$ such that $\pi_Z(\omega_2)=\omega_2\cap Z\in \Omega_2$. By Lemma \ref{lemma:clubequiv}, it follows that $\Omega_2$ is stationary in $\omega_2$.
\end{proof}

\begin{lemma}
    \label{lemma:properO2}
    $\Sigma_4\mhyphen MP_{proper}(\emptyset)$ together with the Bedrock Axiom implies that $\Omega_2$ is bounded below $\omega_2$. The same holds if we replace proper forcing with countably closed forcing or any other $\Sigma_3$-definable forcing class which can collapse arbitrary cardinals to $\omega_1$, is closed under two-step iterations, and has the countable covering property.
\end{lemma}
\begin{proof}
    By Usuba's Corollary \ref{cor:ddgapp}, the Bedrock Axiom implies that the universe (and in fact every model in the generic multiverse) is a set forcing extension of the mantle. We then observe that the $\omega_2$ of any model in the generic multiverse must be regular in the mantle, and that there is some regular $\lambda$ such that $\mathbb{M}$ and $V$ agree on which ordinals are regular cardinals above $\lambda$. In particular, every ordinal which is regular in $\mathbb{M}$ but has countable cofinality in $V$ is below $\lambda$.

    We now consider the sentence "there exists a $\beta<\omega_2$ such that for all $\alpha>\beta$ and all proper posets $\mathbb{P}$, $\Vdash_\mathbb{P} \alpha\not\in\Omega_2$. Since $\alpha\in \Omega_2$ is $\Sigma_3$, this sentence can be written in $\Sigma_4$ form. To see that it is forceable by $Coll(\omega_1, \lambda)$, let $G\subset Coll(\omega_1, \lambda)$ be $V$-generic and $V[G][H]$ be some further proper extension. For any ground $W$ of $V[G][H]$, if $\omega_2^W>\lambda$, then $\omega_2^W$ is regular in both the mantle and $V$. It follows from the countable covering property that it must have uncountable cofinality in $V[G][H]$, so it cannot be in $\Omega_2^{V[G][H]}$. Hence $\Omega_2^{V[G][H]}\subset\lambda <\omega_2^{V[G]}$, so the sentence holds in $V[G]$. Since by construction the sentence is preserved by all further proper forcing, the maximality principle implies that $\sup(\Omega_2)<\omega_2$ in $V$.
\end{proof}

Some examples of models in which the hypotheses of the preceding lemma hold:

\begin{itemize}
    \item Any model of $\Sigma_4\mhyphen CPFA$ with an extendible cardinal
    \item The standard model of $\Sigma_4\mhyphen CPFA$, obtained by forcing as in Theorem \ref{thm:gencon} from a model with a cardinal supercompact for $C^{(3)}$ (and thus many extendibles)
    \item The model of $\Sigma_4\mhyphen MP_{proper}(H_{\omega_2})$ obtained by forcing as in Theorem \ref{thm:symbfa} from a regular $\Sigma_4$-correct cardinal over $L$
    \item The model of $\Sigma_4\mhyphen MP_{proper}(\emptyset)$ obtained from forcing over $L$ up to a singular $\Sigma_4$-correct cardinal, as in Lemma 2.6 of Hamkins \cite{hamkinsMP}
\end{itemize}

Thus we get, for example:

\begin{theorem}
    \label{thm:contradiction}
    If there is an extendible cardinal, then $\Sigma_3\mhyphen CMM$ contradicts $\Sigma_4\mhyphen CPFA$, and $\Sigma_3\mhyphen CSCFA$ contradicts $\Sigma_4\mhyphen CFA_{<\omega_2}(<\omega_1\mhyphen closed)$.
\end{theorem}

It follows that although large cardinal axioms, strengthenings of $MM$, and strengthenings of $PFA$ all express in different ways the idea that the universe of sets is very large, beyond a certain point it cannot be large in all three ways simultaneously.

The hypotheses of Theorem \ref{thm:contradiction} are likely stronger than necessary. The assumption of an extendible cardinal is helpful for controlling $\Omega_2$, but can perhaps be dispensed with.

\begin{question}
    If the mantle is not a ground, is $\Sigma_3\mhyphen CMM$ consistent with $\Sigma_4\mhyphen CPFA$?
\end{question}

Alternatively, it could perhaps be shown that $\Sigma_3$-correct or $\Sigma_4$-correct forcing axioms outright imply the Bedrock Axiom without additional large cardinal assumptions. This is false for classical forcing axioms, since Reitz showed that the existence of a supercompact cardinal is consistent with the failure of the Bedrock Axiom (\cite{ReitzGA}, Corollary 25(2)) and this failure is preserved by set forcing, but may become true at the $\Sigma_3$ level or above.

It also seems likely that the $\Sigma_3$ and $\Sigma_4$ can be reduced to lower complexities, although since $MM$ implies $PFA$, the latter can't be reduced all the way to $\Sigma_1$.

\begin{question}
    Is $\Sigma_2\mhyphen CPFA$ consistent with $MM$?
\end{question}

\chapter{Residual Reflection Principles}
\label{section:rrp}

In light of Theorem \ref{thm:symbfaMP}, it is natural to wonder whether asymmetrically bounded or unbounded $\Sigma_n$-correct forcing axioms can be factored as the conjunction of the appropriate $\Sigma_n$ maximality principle and some sort of reflection principle. The attempted factorization chronicled in this chapter only works for a few forcing classes, but the resulting reflection principles have some independent interest.

\begin{definition}
    For $n$ a positive integer, $\kappa\leq\lambda$ uncountable cardinals, and $\Gamma$ a forcing class, $\Sigma_n\mhyphen RRP(\kappa, \lambda, \Gamma)$ is the assertion that for all provably $\Gamma$-persistent $\Sigma_n$ formulas $\phi$ and all $a\in H_\lambda$ such that $\phi(a)$ holds, the set of $Z\prec H_\lambda$ of size less than $\kappa$ and containing $a$ such that $\phi(\pi_Z(a))$ holds is stationary in $[H_\lambda]^{<\kappa}$.
\end{definition}

This is reminiscent of $\Sigma_n$-correct $H_\lambda$-reflection, but here we assert only that $\phi(\pi_Z(a))$ holds in $V$, not necessarily in $V_\kappa$. Thus $\Sigma_n\mhyphen RRP(\kappa, \lambda, \Gamma)$ does not imply that $\kappa$ is $\Sigma_n$-correct, or even correct with respect to provably $\Gamma$-persistent $\Sigma_n$ formulas, so it is consistent with $\kappa$ being small, at least for some forcing classes $\Gamma$.

These statements can be called residual reflection principles, because they are all that remains of $\Sigma_n$-correct $H_\lambda$-reflection after a length $\kappa$ iteration of posets in $\Gamma\cap V_\kappa$:

\begin{prop}
    \label{prop:rrpfromiter}
    If $\kappa$ is $\Sigma_n$-correctly $H_\lambda$-reflecting for $\lambda>\kappa$, $\Gamma$ is an $n$-nice forcing class, $\mathbb{P}_\kappa\in \Gamma$ is any length $\kappa$ iteration of posets in $V_\kappa\cap\Gamma$ (using the notion of iteration suitable to $\Gamma$) which does not collapse $\kappa$, and $G\subset\mathbb{P}_\kappa$ is a $V$-generic filter, then $V[G]\models\Sigma_n\mhyphen RRP(\kappa,\lambda, \Gamma)$.
\end{prop}
\begin{proof}
    In $V[G]$, given suitable $\phi$, $a$, and a function $h:[H_\lambda^{V[G]}]^{<\omega}\rightarrow H_\lambda^{V[G]}$, we wish to find a $Z\prec H_\lambda^{V[G]}$ of size less than $\kappa$, containing $a$ and $\kappa$, closed under $h$, and transitive below $\kappa$ such that $\phi(\pi_Z(a))$ holds. Now in $V$, let $\dot{a}$ and $\dot{h}$ be suitable $\mathbb{P}_\kappa$-names and $p\in G$ force $\phi(\dot{a})$; as in the proof of Lemma \ref{lemma:reflaver}, we can construct a function $\tilde{h}\in V$ which maps finite sets of names for elements of $H_\lambda^{V[G]}$ to names in $H_\lambda$ for the corresponding values of $h$.

    Therefore we can find a $Z\prec H_\lambda$ of size less than $\kappa$, containing $\dot{a}$, $\mathbb{P}_\kappa$, and $\kappa$, with $Z\cap V_\kappa$ transitive and $\pi_Z(p)=p$, and closed under both $\tilde{h}$ and the map $\alpha\mapsto\mathbb{P}_\alpha$ for $\alpha<\kappa$ such that $V_\kappa\models p\Vdash_{\pi_Z(\mathbb{P}_\kappa)}\phi(\pi_Z(\dot{a}))$. By our conditions on $Z$, $\pi_Z(\mathbb{P}_\kappa)=\mathbb{P}_{\bar{\kappa}}$, where as usual $\bar{\kappa}:=\pi_Z(\kappa)$. Therefore if $G_{\bar{\kappa}}=G\cap\mathbb{P}_{\bar{\kappa}}$, $\phi(\pi_Z(\dot{a})^{G_{\bar{\kappa}}})$ holds in $V_\kappa[G_{\bar{\kappa}}]$ because $p\in G_{\bar{\kappa}}$. By Lemma \ref{lemma:Cnforce}, $\kappa$ is still $\Sigma_n$-correct in $V[G_{\bar{\kappa}}]$, and thus $\phi(\pi_Z(\dot{a})^{G_{\bar{\kappa}}})$ holds in $V[G_{\bar{\kappa}}]$. Since the tail forcing $\mathbb{P}_\kappa/\mathbb{P}_{\bar{\kappa}}\in \Gamma$, $\phi(\pi_Z(\dot{a})^{G_{\bar{\kappa}}})$ continues to be true in $V[G]$ (although not necessarily in $V_\kappa^{V[G]}$, since $\mathbb{P}_\kappa$ is too large to preserve the $\Sigma_n$-correctness of $\kappa$).

Then as in the proof of Lemma \ref{lemma:reflaver}, by Lemma \ref{lemma:extembed}, $Z[G]\prec H_\lambda^{V[G]}$ (where $Z[G]:=\{\dot{x}^G\sbp \dot{x}\in Z\cap H_\lambda^{\mathbb{P}_\kappa}\}$) and for all $\mathbb{P}_\kappa$-names $\dot{x}\in Z$, $\pi_Z(\dot{x})^{G_{\bar{\kappa}}}=\pi_{Z[G]}(\dot{x}^G)$. It follows that $V[G]\models\phi(\pi_{Z[G]}(a))$, and the closure of $Z$ under $\tilde{h}$ implies that $Z[G]$ is closed under $h$, as desired. Furthermore, as $\pi_{Z[G]}^{-1}$ extends $\pi_Z^{-1}$, it also maps its critical point $\bar{\kappa}$ to $\kappa$, so $Z[G]\cap V_\kappa^{V[G]}$ is transitive. Thus $Z[G]\in [H_\lambda^{V[G]}]^{<\kappa}$ has all the properties we wanted, so $\Sigma_n\mhyphen RRP(\kappa, \lambda, \Gamma)$ holds in $V[G]$, as desired.
\end{proof}

Thus residual reflection principles at least hold in the standard models of $\Sigma_n$-correct bounded forcing axioms as constructed in Theorem \ref{thm:asymbfa}, but in fact they can easily be seen to hold in all models of those axioms:

\begin{lemma}
    \label{lemma:cbfaimprrp}
    For any forcing class $\Gamma$ containing the trivial forcing, positive integer $n$, and uncountable cardinals $\kappa<\lambda$, $\Sigma_n\mhyphen CBFA_{<\kappa}^{<\lambda}(\Gamma)$ implies $\Sigma_n\mhyphen RRP(\kappa, \lambda, \Gamma)$.
\end{lemma}
\begin{proof}
    Since residual reflection principles are effectively special cases of the M-A formulation of $\Sigma_n$-correct forcing axioms, we essentially need only prove the bounded version of one direction of Theorem \ref{thm:equivforms}. Let $\phi$ be a provably $\Gamma$-persistent $\Sigma_n$ formula, $a\in H_\lambda$ such that $\phi(a)$ holds, and $h:[H_\lambda]^{<\omega}\rightarrow H_\lambda$ a function. Applying the axiom, there is a transitive structure $N$ of size less than $\kappa$ with an elementary embedding $\sigma: N\rightarrow H_{(2^{<\lambda})^+}$ such that $a, h, H_\lambda, \kappa \in rng(\sigma)$, $\phi(\sigma^{-1}(a))$ holds, and $\sigma(crit(\sigma))=\kappa$. Since $H_{(2^{<\lambda})^+}\models ``H_\lambda \text{ is closed under } h"$, $rng(\sigma)\cap H_\lambda$ is as well. $rng(\sigma)\cap H_\lambda\prec H_\lambda$ because $H_{(2^{<\lambda})^+}$ and $rng(\sigma)$ agree on which formulas $H_\lambda$ satisfies. Therefore $rng(\sigma)\cap H_\lambda$ is an elementary substructure of $H_\lambda$ of size less than $\kappa$ which is closed under $h$ and transitive below $\kappa$ such that $\phi(\sigma^{-1}(a))$ holds, and (a restriction of) $\sigma^{-1}$ is exactly the transitive collapse isomorphism of $rng(\sigma)\cap H_\lambda$. Thus $\Sigma_n\mhyphen RRP(\kappa,\lambda, \Gamma)$ holds.
\end{proof}

We now study the forcing classes for which Theorem $\ref{thm:symbfaMP}$ can be extended to an equivalence for asymmetrically bounded forcing axioms using residual reflection principles.

\begin{definition}
    A forcing class $\Gamma$ is provably self-preserving iff the formula $\mathbb{P}\in \Gamma$ is provably $\Gamma$-persistent.
\end{definition}

Strictly speaking, provable self-preservation is a property of definitions of forcing classes, not an extensional property of the classes themselves.

For any forcing class, we can slim it down to a "provably self-preserving interior", though in doing so we may lose an unacceptably large number of posets:

\begin{lemma}
    \label{lemma:squaregamma}
    For any $\Sigma_n$- or $\Pi_n$-definable forcing class $\Gamma$ which is provably closed under two-step iterations, $\square_\Gamma \Gamma:=\{\mathbb{P}\in\Gamma\hspace{2pt}|\hspace{2pt} \forall\mathbb{Q}\in \Gamma\Vdash_{\mathbb{Q}}\check{\mathbb{P}}\in\Gamma\}$ is a provably self-preserving $\Pi_{n+1}$-definable forcing class.
\end{lemma}
\begin{proof}
    ZFC proves that if $\mathbb{P}\in \square_\Gamma \Gamma$, $\mathbb{Q}\in \Gamma$, and $\dot{\mathbb{R}}$ is a $\mathbb{Q}$-name for a forcing in $\Gamma$, then $\mathbb{Q}*\dot{\mathbb{R}}\in \Gamma$, so by the definition of $\square_\Gamma \Gamma$, $\Vdash_{\mathbb{Q}*\dot{\mathbb{R}}} \mathbb{P}\in\Gamma$. It follows (provably) that $\Vdash_\mathbb{Q} \mathbb{P}\in \square_\Gamma \Gamma$, so the formula $\mathbb{P}\in \square_\Gamma \Gamma$ is provably $\Gamma$-persistent. Since ZFC proves $\square_\Gamma \Gamma\subseteq \Gamma$, it is thus provably $\square_\Gamma \Gamma$-persistent, so $\square_\Gamma \Gamma$ is provably self-preserving.

To see that $\square_\Gamma \Gamma$ is $\Pi_{n+1}$-definable, observe that is definition can be written as $\forall\mathbb{Q}(\mathbb{Q}\not\in\Gamma\lor \Vdash_\mathbb{Q} \mathbb{P}\in \Gamma)$; whether $\Gamma$ is $\Sigma_n$ or $\Pi_n$, the inner disjunction is of a $\Sigma_n$ formula and and $\Pi_n$ formula, so the overall formula is $\Pi_{n+1}$.
\end{proof}

We survey provable self-preservation for common forcing classes:

\begin{prop}
    \label{prop:selfpresexamples} Assume that $ZFC$ is consistent. Then:
    \begin{enumerate}[a)]
        \item The class of countably closed forcing (or more generally $\lk$-closed forcing for regular $\kappa$) is provably self-preserving.
        \item The class of ccc forcing is not provably self-preserving using the most natural definition, but $MA_{\omega_1}$ implies that $ccc=\square_{ccc}ccc$.
        \item The class of subcomplete forcing is not provably self-preserving, and $CH$ implies that $\square_{sc} sc \subsetneq sc$.
        \item The classes of proper and stationary set-preserving forcing are not provably self-preserving, and in fact $ZFC$ proves that $\square_\Gamma \Gamma\subsetneq \Gamma$ when $\Gamma$ is either one
    \end{enumerate}
\end{prop}
\begin{proof}
    (a): Let $\mathbb{P}$ and $\mathbb{Q}$ be $\lk$-closed, $G\subset\mathbb{Q}$ be $V$-generic, and $f\in V[G]$ be a descending sequence $\gamma\rightarrow\mathbb{P}$ for some $\gamma<\kappa$. Then $f\in V$, so there is a $p\in \mathbb{P}$ such that $p\leq f(\alpha)$ for all $\alpha<\gamma$. Therefore $\mathbb{P}$ remains $\lk$-closed in $V[G]$.

    (b): If $T$ is a Suslin tree, then $T\in ccc$, but $\Vdash_T \check{T}\not\in ccc$. Since the existence of a Suslin tree is consistent relative to $ZFC$, it follows that $ZFC$ cannot prove that all posets satisfying the ccc retain this property under ccc forcing.

    However, if we further assume $MA_{\omega_1}$, let $\mathbb{P}$ be a poset, $\mathbb{Q}$ be ccc, $q\in\mathbb{Q}$, and $\dot{a}$ be  such that $q \Vdash_\mathbb{Q} ``\dot{a}$ is an enumeration of an antichain of $\mathbb{P}$ of size $\omega_1"$. For each $\alpha<\omega_1$, we can find a set $D_\alpha$ dense below $q$ which decides the value of $\dot{a}_\alpha$. Applying $MA_{\omega_1}$ below $q$, we obtain a filter $F$ such that there is some $q_\alpha\in F\cap D_\alpha$ for each $\alpha<\omega_1$. Let $p_\alpha$ be the element of $\mathbb{P}$ such that $q_\alpha\Vdash_\mathbb{Q} \dot{a}_\alpha=\check{p}_\alpha$.
    
    If for some $\alpha<\beta<\omega_1$ there is a $p\in\mathbb{P}$ extending both $p_\alpha$ and $p_\beta$, then for any $V$-generic filter $G$ containing both $q_\alpha$ and $q_\beta$, $\dot{a}_\alpha^G\parallel \dot{a}_\beta^G$, contradicting the fact that $q\in G$ forces $\dot{a}^G$ to be an antichain. Therefore $\{p_\alpha\hspace{2pt}|\hspace{2pt} \alpha<\omega_1\}$ is an antichain of $\mathbb{P}$ in $V$, so $\mathbb{P}$ does not satisfy the ccc. Thus under our hypotheses every poset satisfying the ccc must in fact be in $\square_{ccc}ccc$.

    (c): Under $CH$, the Namba forcing $\mathbb{N}$ of $V$ is subcomplete, but $\Vdash_{\mathbb{N}} \check{\mathbb{N}}\not\in sc$, since $\mathbb{N}\times\mathbb{N}$ adds a real. See Theorem 4.1.27 of Kaethe Minden's dissertation \cite{mindensubcomplete} for more details. Since $CH$ is consistent relative to $ZFC$, it follows that $ZFC$ cannot prove that subcomplete forcing is self-preserving.

    (d): For proper forcing, by an argument due to Shelah and related by Goldstern (\cite{GSPres}), the ground model version of $Coll(\omega_1, \omega_2)$ ceases to be proper after any forcing which adds a real, which includes many proper posets. For stationary set-preserving forcing, both Namba forcing and $Coll(\omega_1, \omega_2)$ preserve stationary subsets of $\omega_1$, but forcing with either makes the ground model version of the other collapse $\omega_1$; see Minden (\cite{mindensubcomplete}, proof of Theorem 4.1.29).
\end{proof}

The main result of this chapter is:

\begin{theorem}
\label{thm:cbfafactor}
    The following are equivalent for all uncountable cardinals $\kappa\leq \lambda$, positive integers $n$, and provably self-preserving $n$-nice forcing classes $\Gamma$:
    \begin{enumerate}
        \item $\Sigma_n\mhyphen CBFA_{<\kappa}^{<\lambda}(\Gamma)$
        \item $\Sigma_n\mhyphen MP_\Gamma(H_\kappa)$ and $\Sigma_n\mhyphen RRP(\kappa, \lambda, \Gamma)$
    \end{enumerate}
\end{theorem}
\begin{proof}
    $(1\Rightarrow 2):$ We get the maximality principle from Theorem \ref{thm:symbfaMP} and the residual reflection principle from Lemma \ref{lemma:cbfaimprrp}.

    $(2\Rightarrow 1):$ Given a provably $\Gamma$-persistent $\Sigma_n$ formula $\phi$ and an $a\in H_\lambda$ such that $\phi(a)$ holds in some $\Gamma$-extension, the formula $\exists\mathbb{P}\in\Gamma\hspace{3pt} \Vdash_{\mathbb{P}} \phi(\check{a})$ is $\Sigma_n$ because $\Gamma$ is $\Sigma_n$-definable. To see that it is provably $\Gamma$-persistent, let $\mathbb{P}$ witness it and $\mathbb{Q}\in \Gamma$ be arbitrary. Then $\Vdash_{\mathbb{P}}\check{\mathbb{Q}}\in \Gamma$ because $\Gamma$ is provably self-preserving, so $\Vdash_{\mathbb{P}\times\mathbb{Q}}\phi(\check{a})$ because $\phi$ is provably $\Gamma$-persistent. Since products commute and $\Vdash_{\mathbb{Q}}\check{\mathbb{P}}\in \Gamma$, it follows that $\Vdash_{\mathbb{Q}}(\exists\mathbb{P}\in\Gamma\hspace{3pt} \Vdash_{\mathbb{P}} \phi(\check{a}))$, as desired. Thus by the residual reflection principle, there are stationarily many $Z\prec H_\lambda$ of size less than $\kappa$ containing $a$ such that $\phi(\pi_Z(a))$ is $\Gamma$-forceable. Since $\pi_Z(a)\in H_\kappa$, the maximality principle yields $\phi(\pi_Z(a))$ in $V$. Thus the M-A formulation of $\Sigma_n\mhyphen CBFA_{<\kappa}^{<\lambda}(\Gamma)$ holds, and bounded versions of the arguments for Theorem \ref{thm:equivforms} yield the other formulations.
\end{proof}

Thus for example $\Sigma_n$-correct bounded forcing axioms for countably closed forcing can be factored into the $\Sigma_n$-maximality principle for countably closed forcing and a $\Sigma_n$-residual reflection principle when $n\geq 2$. The same holds for $\square_{ccc}ccc$ when $n\geq 3$ (the ccc is a $\Delta_2$ property because it holds iff it holds in $H_\theta$ for some uncountable $\theta$ iff it holds in $H_\theta$ for all uncountable $\theta$ such that the poset is itself in $H_\theta$). Furthermore, since by the results of the next chapter $\Sigma_2\mhyphen MP_{ccc}(H_{2^{\aleph_0}
})$ implies that $CH$ fails badly and by Proposition \ref{prop:mpbfa} it also implies $MA$, Proposition \ref{prop:selfpresexamples}(b) yields that $\square_{ccc}ccc=ccc$ under the hypotheses of the above theorem when $\Gamma=ccc$. Thus $\Sigma_n\mhyphen CBMA^{<\lambda}$ is equivalent to the conjunction of the $\Sigma_n$-maximality principle for $ccc$ forcing and $\Sigma_n\mhyphen RRP(2^{\aleph_0},\lambda, ccc)$ when $n\geq 3$, but not necessarily when $n=2$. Even though ccc forcing has a $\Sigma_2$ definition, we must use the more complex definition of $\square_{ccc} ccc$ in order for the preservation to be provable in $ZFC$ alone.

For other forcing classes, such as subcomplete, proper, or semiproper forcing, even shifting to more complex definitions does not appear to help. Examining the consistency proofs for $\Sigma_n$-correct forcing axioms sheds some light on the obstacle: when there is a $\mathbb{P}\in \Gamma^{V[G]}$ forcing $\phi(\dot{a})$, we can find an $N$, $\sigma$, and $F$ such that $\phi(\sigma^{-1}(\dot{a})^F)$ already holds not because there is some $\mathbb{Q}\in \Gamma^{V[G]}$ forcing this, as would be required for a straightforward maximality principle argument to work, but because there is such a $\mathbb{Q}$ which was in $\Gamma$ as defined inside of some ground model $V[G_\alpha]$. Thus perhaps, in order to get a factorization for a wider range of forcing classes, a more geological approach is needed. For now, the following (somewhat vague) question remains open:

\begin{question}
    Is $\Sigma_n\mhyphen CBFA_{<\kappa}^{<\lambda}(\Gamma)$ equivalent to the conjunction of some sort of generalized maximality principle and some sort of reflection principle for a wide number of natural forcing classes?
\end{question}

Further open questions arise from the fact that, although the proof of Theorem \ref{thm:cbfafactor} does not appear to generalize beyond the provably self-preserving classes, it is unclear how to establish the separations that would definitively show that the theorem does not generalize:

\begin{question}
    If $\Gamma$ is a forcing class which is never self-preserving, is it consistent for suitable $\kappa$ that $\Sigma_n\mhyphen MP_\Gamma(H_\kappa)$ and $\Sigma_n\mhyphen RRP(\kappa, \lambda, \Gamma)$ hold but $\Sigma_n\mhyphen CBFA_{<\kappa}^{<\lambda}(\Gamma)$ fails?
\end{question}

\begin{question}
    Is it consistent that $\Sigma_2\mhyphen MP_{ccc}(H_\mathfrak{c})$ and $\Sigma_2\mhyphen RRP(\mathfrak{c},\lambda, ccc)$ hold for all $\lambda\geq \mathfrak{c}$, but $\Sigma_2\mhyphen CMA$ fails?
\end{question}

\chapter{Consequences}

We conclude with a brief survey of some of the implications of $\Sigma_n$-correct forcing axioms. It is highly likely that they have many more consequences, perhaps including some much more dramatic than those explored here, but the constraints of time, space, and my own limited knowledge will restrict us to only these few. Particular attention will be paid to $\Sigma_n$-correct Martin's Axiom, since for $n\geq 2$ it has a somewhat different character than and in particular is much stronger than ordinary Martin's Axiom.

\section{The Cardinality of the Continuum}
\label{section:continuum}

We begin with some simple observations on the implications for various axioms on the size of the continuum.

\begin{prop}
    \label{prop:continuum}
    \begin{enumerate}
        \item $\Sigma_n\mhyphen CPFA$ and $\Sigma_n\mhyphen CMM$ imply that $2^{\aleph_0}=\aleph_2$ for any positive $n$
        \item $\Sigma_2\mhyphen CSCFA$ and $\Sigma_2\mhyphen CFA_{<\omega_2}(<\omega_1\mhyphen closed)$ imply $CH$
        \item $\Sigma_2\mhyphen CMA$ implies that the continuum is weakly Mahlo
    \end{enumerate}
\end{prop}
\begin{proof}
    (1) follows from the well-known fact that $PFA$ implies that the continuum is $\aleph_2$.

    (2): Since $\omega_1$ is $\Delta_2$ definable without parameters and $\mathcal{P}(\omega)$ is $\Pi_1$-definable, the assertion that there is a bijection between them is $\Sigma_2$. Since $Coll(\omega_1, \mathfrak{c})$ is a countably closed (and hence subcomplete) poset which forces $CH$, and neither subcomplete nor countably closed posets add reals, either $\Sigma_2$-correct forcing axiom implies $CH$. In fact, $CH$ follows merely from the corresponding $\Sigma_2$ maximality principles.

    (3): First, we show (following Stavi and Vaananen \cite{svMP}) that $\Sigma_2\mhyphen MP_{ccc}(H_\mathfrak{c})$ implies that the continuum is weakly inaccessible. First, by Proposition \ref{prop:mpbfa}, it implies $MA$ and thus that the continuum is regular. To see that the continuum is a weak limit cardinal, let $\lambda<\mathfrak{c}$. Since $\mathcal{P}(\omega)$ is $\Pi_1$ definable without parameters and $\lambda^{++}$ is $\Delta_2$ in $\lambda$, "there is an injective function $\lambda^{++}\rightarrow\mathcal{P}(\omega)$" is a $\Sigma_2$ statement which can be forced by adding $\lambda^{++}$ reals and, since ccc forcing preserves cardinals, provably remains true after any further ccc forcing. Thus it holds in $V$, so the continuum cannot be $\lambda^+$.

    Now strengthen our hypotheses to $\Sigma_2\mhyphen CBMA^{<\mathfrak{c}^+}$. (Ordinary Martin's Axiom is equivalent to all of its bounded forms, since by definition a ccc Boolean algebra only has small antichains, but this fails for the $\Sigma_n$-correct versions, since the boundedness will restrict the size of the names that can be used as parameters.) Let $C\subseteq\mathfrak{c}$ be a club. Then "$\sup(C)$ is weakly inaccessible" is a true $\Pi_1$ formula provably preserved by ccc forcing because ccc forcing preserves weak inaccessibility, so applying the Miyamoto-Aper\'o form of the axiom, there is a $Z\prec H_{\mathfrak{c}^+}$ of size less than $\mathfrak{c}$ containing $\mathfrak{c}$ and $C$ such that the supremum of $\pi_Z(C)$ is weakly inaccessible and $Z\cap\mathfrak{c}$ is transitive. Since by elementarity $\pi_Z(C)=Z\cap C$ is unbounded in $\pi_Z(\mathfrak{c})=Z\cap\mathfrak{c}$, $\sup(Z\cap C)=Z\cap\mathfrak{c}\in C$ by closure, so $C$ contains a weakly inaccessible cardinal. It follows that $\mathfrak{c}$ is weakly Mahlo.
\end{proof}

Thus combined with Theorem \ref{thm:contradiction}, we get that $\Sigma_n\mhyphen CMA$, $\Sigma_n\mhyphen CPFA$, $\Sigma_n\mhyphen CMM$, $\Sigma_n\mhyphen CSCFA$, and $\Sigma_n\mhyphen CFA_{<\kappa}(<\hspace{-3pt}\omega_1\mhyphen closed)$ are mutually inconsistent for $n\geq 4$ (and perhaps for smaller $n$), with some inconsistencies possibly requiring an extendible cardinal (or at least the Bedrock Axiom).

We can strengthen (3) above to show that the continuum is in fact so large as to be undefinable in terms of ccc-persistent formulas, so it is not the least weakly Mahlo cardinal, and in fact should have all properties in weakly hyper-Mahlo hierarchy:

\begin{prop}
    For $n\geq 2$, $\Sigma_n\mhyphen MP_{ccc}(H_{\mathfrak{c}})$ (i.e. $\Sigma_n$-correct symmetrically bounded Martin's Axiom) implies that the continuum is not the least element of any class of ordinals with a provably ccc-persistent $\Sigma_n$ definition. $\Sigma_n\mhyphen CBMA^{<\mathfrak{c}^+}$ (or merely $\Sigma_n\mhyphen RRP(\mathfrak{c}, \mathfrak{c}^+, ccc)$) further implies that any such class containing $\mathfrak{c}$ must be stationary below it.
\end{prop}
\begin{proof}
    Let $\phi$ be a provably ccc-persistent $\Sigma_n$ formula such that $\phi(\mathfrak{c})$ holds. Then any ccc forcing which adds at least $(\mathfrak{c}^V)^+$ reals forces "there is an $\alpha$ smaller than the continuum such that $\phi(\alpha)$ holds". Since any further ccc forcing preserves $\phi(\mathfrak{c}^V)$ and cannot decrease the continuum, the quoted property corresponds to a provably ccc-persistent formula; because the cardinality of the continuum is $\Delta_2$-definable, this formula is also $\Sigma_n$ as long as $n\geq 2$. Thus by the maximality principle, there is some $\alpha<\mathfrak{c}^V$ such that $V\models \phi(\alpha)$, as desired.

    If $\Sigma_n\mhyphen RRP(\mathfrak{c}, \mathfrak{c}^+, ccc)$ holds, then there are stationarily many $Z\in [H_{\mathfrak{c}^+}]^{<\mathfrak{c}}$ such that $Z\cap \kappa$ is transitive (and thus equal to $\pi_Z(\kappa)$) and $\phi(\pi_Z(\kappa))$ holds. Thus by Lemma \ref{lemma:clubequiv}, $\{\alpha<\mathfrak{c}\sbp \phi(\alpha)\}$ is stationary.
\end{proof}

Of course, similar results hold for other forcing classes, but since ccc forcing preserves more properties related to the size of cardinals (most notably the cofinality function), this particular case is more relevant to characterizing how large the continuum is under $\Sigma_n\mhyphen CMA$ (whereas the $\Sigma_n$-correct forcing axioms for other natural classes often imply that $\kappa$ is merely $\omega_2$).

The first result of the following section can be interpreted as saying that under $\Sigma_2$-correct Martin's Axiom, the continuum is "very weakly compact" (i.e. weakly inaccessible and has the tree property). Propostion \ref{prop:genextend} can similarly be read as a statement about the largeness of the continuum when $\Gamma$ is ccc forcing.

\section{Combinatorics}
\label{section:combin}

This sections surveys some of the combinatorial consequences of $\Sigma_n$-correct forcing axioms, starting with the tree property:

\begin{prop}
\begin{enumerate}
    \item $\Sigma_2\mhyphen CBMA^{<\mathfrak{c}^+}$ (or merely $\Sigma_2\mhyphen RRP(\mathfrak{c}, \mathfrak{c}^+, ccc)$) implies that the continuum has the tree property.
    \item If $\Gamma$ is an $n$-nice forcing class, then $\Sigma_{n+1}\mhyphen RRP(\kappa, \kappa^+, \Gamma)$ implies that for every $\kappa$-tree, there is a $\Gamma$-extension in which it has a cofinal branch.
\end{enumerate}
\end{prop}
   \begin{proof} 
   (1): Assume toward a contradiction that $T=(\mathfrak{c}, <_T)$ is a $\mathfrak{c}$-tree with no cofinal branch. Having no cofinal branch is a $\Pi_1$ property of $\mathfrak{c}$ and $T$, since it amounts to saying that all functions from $\mathfrak{c}$ with the ordinal ordering to $T$ fail to be strictly order-preserving. To see that it is provably ccc-persistent, let $\mathbb{P}$ satisfy the ccc and $\dot{b}$ be a $\mathbb{P}$-name for a cofinal branch of $T$; we can assume without loss of generality that all conditions of $\mathbb{P}$ force $\dot{b}$ to be a cofinal branch. Define $T_{\dot{b}}:=\{x\in T\sbp \exists p\in \mathbb{P}\hspace{3pt} p\Vdash \check{x}\in\dot{b}\}$.

   Since $\dot{b}$ is forced to be a cofinal branch, $T_{\dot{b}}$ contains at least one element from every level of $T$. However, conditions of $\mathbb{P}$ which force different elements of the same level to be in $\dot{b}$ are incompatible, so by the ccc there are at most countably many conditions for each level. It follows that each level of $T_{\dot{b}}$ is at most countable. As $CH$ fails under our hypotheses, $T_{\dot{b}}$ is a tree of height $\mathfrak{c}$ with width uniformly bounded below $\mathfrak{c}$, so by a standard result of Kurepa (see e.g. Kanamori Proposition 7.9 \cite{kanamori}), it has a cofinal branch, which is then also a cofinal branch of $T$ in $V$.

   Thus $T$'s lack of a cofinal branch is a provably ccc-persistent $\Pi_1$ property, so there are stationarily many $Z\prec H_{\mathfrak{c}^+}$ of size less than $\mathfrak{c}$ such that $\pi_Z(T)$ has no cofinal branch. Fix a particular such $Z$ which is closed under the mapping from ordinals $\alpha<\mathfrak{c}$ to the $\alpha$th level of $T$ and transitive below $\mathfrak{c}$. Combining these two properties, every node of $T$ below level $\pi_Z(\mathfrak{c})$ is in $Z$, but by elementarity $Z$ contains no nodes from higher levels, so $\pi_Z(T)=T\upharpoonright\pi_Z(\mathfrak{c})$. Thus $T\upharpoonright\pi_Z(\mathfrak{c})$ has no cofinal branch, contradicting the fact that $T_{\pi_Z(\mathfrak{c})}$ is nonempty. Hence $\mathfrak{c}$ must have the tree property under our hypotheses.

   (2): Assume toward a contradiction that $T$ is a $\kappa$-tree with no cofinal branch in any $\Gamma$-extension. By Lemma \ref{lemma:neccompl}, this is a provably $\Gamma$-persistent $\Pi_n$ property. Then we can argue as above to produce a $Z\prec H_{\kappa^+}$ such that $\pi_Z(T)=T\upharpoonright\pi_Z(\kappa)$ has no cofinal branch in any $\Gamma$-extension, contradicting the fact that the predecessors of any node of height $\pi_Z(\kappa)$ in $T$ form a cofinal branch of $T\upharpoonright\pi_Z(\kappa)$ in $V$ itself.
\end{proof}

Next, some implications of the $\Sigma_n$-correct forcing axioms for subcomplete or countably closed forcing:

\begin{prop}
    Let $\Gamma$ be the class of subcomplete forcing or the class of countably closed forcing. Then $\Sigma_2\mhyphen CFA_{<\omega_2}(\Gamma)$ implies:
    \begin{enumerate}
        \item $\diamondsuit$
        \item There are no ($\omega_1$-)Kurepa trees
    \end{enumerate}
\end{prop}
\begin{proof}
    In fact, both consequences follows from $\Sigma_2\mhyphen MP_\Gamma(H_{\omega_2})$.

    (1): It is a standard result that adding a Cohen subset to $\omega_1$ forces $\diamondsuit$ (see e.g. Kunen's Lemma IV.7.30 \cite{kunen}). This forcing is countably closed and thus subcomplete. Jensen proved (\cite{jensenfacch}, Part 3, Lemma 4) that subcomplete forcing and thus countably closed forcing\footnote{The result for countably closed forcing was in fact known well before Jensen.} preserves $\diamondsuit$. $\diamondsuit$ is a $\Sigma_2$ sentence, since it asserts that there is a sequence $\langle A_\alpha\sbp \alpha<\omega_1\rangle$ such that for all $A, C\subseteq \omega_1$ with $C$ a club, there is an $\alpha\in C$ such that $A\cap\alpha=A_\alpha$. Thus $\diamondsuit$ follows from the maximality principles for both forcing classes.

    (2): Let $T\in H_{\omega_2}$ be an $\omega_1$-tree. Then there is a countably closed (hence subcomplete) forcing to add a surjection from $\omega_1$ to the set of cofinal branches of $T$, and further subcomplete forcing provably does not add branches to $T$ by Minden's Theorem 3.1.2 \cite{mindensubcomplete}. Asserting that $T$ is not a Kurepa tree is $\Sigma_2$, since we can say that there is a function $f$ with domain $\omega_1$ such that every strictly order-preserving map $\omega_1\rightarrow T$ with downward-closed range is itself in the range of $f$. Thus $T$ not being a Kurepa tree is a $\Gamma$-forceable and provably $\Gamma$-persistent $\Sigma_2$ formula, so $T$ is not a Kurepa tree in $V$. Since every $\omega_1$-tree is isomorphic to one in $H_{\omega_2}$ it follows that there are no Kurepa trees at all.
\end{proof}

Finally, we consider the following version of simultaneous stationary reflection:

\begin{definition}
    $\nu\mhyphen RP_{<\kappa}(\theta)$, where $\nu<\kappa\leq \theta$ are cardinals with $\kappa$ and $\theta$ regular and uncountable, is the statement that for every sequence of $\nu$ stationary subsets of $[\theta]^\omega$ $\mathcal{S}=\langle S_\alpha\sbp \alpha<\nu\rangle$ and every $X\subset \theta$ with $|X|<\kappa$, there is a $Y$ such that $X\subseteq Y\subset \theta$, $|Y|<\kappa$, and $S_\alpha\cap [Y]^\omega$ is stationary in $[Y]^\omega$ for all $\alpha<\nu$. When $\nu$ and/or $\kappa$ is left unspecified, it is assumed that $\nu=1$ and $\kappa=\omega_2$. When $\theta$ is left unspecified, it is assumed that the principle holds for all regular $\theta\geq \kappa$.
\end{definition}

These principles follow from $\Sigma_2$-correct forcing axioms for a wide range of forcing classes, but different proofs seem to be necessary for two overlapping cases:

\begin{theorem}
    If $\kappa$ and $\lambda$ are regular uncountable cardinals with $\omega_1<\kappa<\lambda$ and $\Gamma$ is a subclass of proper forcing, $\Sigma_2\mhyphen RRP(\kappa, \lambda, \Gamma)$ implies $\nu\mhyphen RP_{<\kappa}(\theta)$ for all cardinals $\nu<\kappa$ and regular $\theta\geq \kappa$ such that $\theta^{\omega\cdot \nu}<\lambda$. Thus $\Sigma_2\mhyphen CFA_{<\kappa}(\Gamma)$ implies $\nu\mhyphen RP_{<\kappa}$ for all $\nu<\kappa$.
\end{theorem}
\begin{proof}
    Whenever the specified bounds on $\nu$ and $\theta$ hold, given $\mathcal{S}$ and $X$ as in the above definition, $\mathcal{S}\in H_\lambda$. Since there is a $\Pi_1$ formula asserting that $\mathcal{S}$ is a $\nu$-sequence of stationary subsets of $[\theta]^\omega$, which is provably $\Gamma$-persistent because every poset in $\Gamma$ is proper, by the residual reflection principle there is a $Z\prec H_\lambda$ of size less than $\kappa$ with $X\cup(\nu+1)\cup \omega_1\cup \{\theta, \mathcal{S}\}\subset Z$ such that, setting $\bar{\theta}:=\pi_Z(\theta)$ and $\bar{\mathcal{S}}:=\pi_Z(\mathcal{S})$, $\bar{\mathcal{S}}$ is a $\nu$-sequence of stationary subsets of $[\bar{\theta}]^\omega$. Set $Y:=\pi_Z^{-1}"\bar{\theta}=Z\cap \theta$.

    Then $Y\subset\theta$ by definition, $|Y|<\kappa$ because $|Z|<\kappa$, and $X\subseteq Y$ because $X\subseteq Z$ and $X\subset \theta$, so it is sufficient to verify that the stationarity of each $S_\alpha$ reflects to $Y$. By elementarity, $\bar{\mathcal{S}}_\alpha=\pi_Z(S_\alpha)$ for each $\alpha<\nu$. For every $a\in Z\cap [\theta]^\omega$, $a\subset Y$, since by elementarity there is a bijection $f:\omega\rightarrow a$ in $Z$, which must then be an actual bijection in $H_\lambda$ and thus in $V$, and since $\omega\subset Z$ we have that $Z$ computes the values of $f$ correctly, so those values must all be in $Z$. Then because by the definition of the transitive collapse $\pi_Z(a)=\pi_Z"(a\cap Z)$, it follows from the injectivity of $\pi_Z$ that $\pi_Z^{-1}"\pi_Z(a)=a$. For similar reasons, $\pi_Z^{-1}"\pi_Z(S_\alpha)=S_\alpha\cap Z$ for all $\alpha<\nu$, so $\pi_Z^{-1}"\pi_Z(S_\alpha)\subseteq S_\alpha\cap [Y]^\omega$.

    To complete the proof, we fix an $\alpha<\nu$ and an $h:[Y]^{\omega}\rightarrow \omega$ and produce an $a\in [Y]^\omega$ such that $\pi_Z(a)\in \pi_Z(S_\alpha)$ and $h"[a]^{<\omega}\subseteq a$. Define $h'=\pi_Z\circ h \circ \pi_Z^{-1}$; then $h':[\bar{\theta}]^{<\omega}\rightarrow \bar{\theta}$, so there is some $\bar{a}\in \pi_Z(S_\alpha)$ closed under it and transitive below $\omega_1$ by stationarity. Then $\bar{a}=\pi_Z(a)$ for some $a\in S_\alpha$ because all elements of $\pi_Z(S_\alpha)$ are of this form. For any finite $\{\beta_1, \dotsc, \beta_k\}\subset a$, $\pi_Z(\{\beta_1,\dotsc,\beta_k\})=\{\pi_Z(\beta_1),\dotsc, \pi_Z(\beta_k)\}\subset \bar{a}$, since as argued in the previous paragraph $a\subset Z$ in this situation. Hence $h'(\{\pi_Z(\beta_1), \dotsc, \pi_Z(\beta_k)\})=\pi_Z(h(\{\beta_1, \dotsc, \beta_k\}))\in \bar{a}$, so $h(\{\beta_1, \dotsc, \beta_k\})\in a$ by elementarity. It follows that $a$ is closed under $h$; to see that it is furthermore transitive below $\omega_1$, note that $\pi_Z$ does not move countable ordinals or map uncountable ordinals to countable because $\omega_1\subset Z$, so this follows from the fact that $\bar{a}\cap\omega_1$ is transitive. Therefore $\pi_Z^{-1}"\pi_Z(S_\alpha)$ is stationary in $[Y]^\omega$, so $S_\alpha\cap [Y]^\omega$ is as well, as desired.
    \end{proof}

    \begin{theorem}
        If $\lambda>\omega_2$ is a regular cardinal and $\Gamma$ is a forcing class which provably preserves stationary subsets of $\omega_1$, $\Sigma_2\mhyphen CBFA_{<\omega_2}^{<\lambda}(\Gamma)$ implies $\omega_1\mhyphen RP(\theta)$ for all regular $\theta\geq \omega_2$ such that $\theta^{\omega_1}<\lambda$ and there is a proper poset in $\Gamma$ which collapses $\theta$ to $\omega_1$. Thus if $\Gamma$ can collapse arbitrarily large cardinals to $\omega_1$ with proper posets, $\Sigma_2\mhyphen CFA_{<\omega_2}(\Gamma)$ implies $\omega_1\mhyphen RP$.
    \end{theorem}
    \begin{proof}
        This is essentially an adaptation of Foreman, Magidor, and Shelah's argument that $FA^+(<\hspace{-2pt} \omega_1\mhyphen closed)$ implies $RP$ (\cite{MMfms}, paragraph before Theorem 13). Given $\theta$, a sequence of stationary subsets of $[\theta]^\omega$ $\mathcal{S}=\langle S_\alpha\sbp \alpha<\omega_1\rangle \in H_\lambda$, $X\subset \theta$ of size $\omega_1$, and a proper $\mathbb{P}\in \Gamma$ which collapses $\theta$ to $\omega_1$, let $\dot{f}$ and $\dot{a}$ be $\mathbb{P}$-names such that $\Vdash_{\mathbb{P}} ``\dot{f}:\check{\omega_1}\rightarrow \check{\theta}$ is a bijection'' and $\Vdash_{\mathbb{P}} \dot{a}=\langle \{\beta<\omega_1\sbp \dot{f}"\beta\in \check{S}_\alpha\}\sbp \alpha<\omega_1\rangle$.

        To show that $\mathbb{P}$ forces $\dot{a}$ to be an $\omega_1$-sequence of stationary subsets of $\omega_1$, let $G\subset\mathbb{P}$ be $V$-generic and argue in $V[G]$. For any club $C\subseteq \omega_1$ in $V[G]$, $\{\dot{f}^G"\beta\sbp \beta\in C\}$ is a club in $([\theta]^\omega)^{V[G]}$, since by the surjectivity of $\dot{f}^G$ every countable subset of $\theta$ is contained in $\dot{f}^G"\beta$ for some $\beta$, and $\bigcup\limits_{n<\omega}\dot{f}^G"\beta_n=\dot{f}^G"\sup\limits_{n<\omega}\beta_n$ for any increasing sequence $\langle \beta_n\sbp n<\omega\rangle$ in $C$. By the properness of $\mathbb{P}$, each $S_\alpha$ remains stationary in $V[G]$, so for all $\alpha<\omega_1$ there is a $\beta_\alpha\in C$ such that $\dot{f}^G"\beta_\alpha \in S_\alpha$. Thus $\beta_\alpha\in C\cap \dot{a}^G_\alpha$. Since this holds for arbitrary clubs $C$ and $\alpha<\omega_1$, each term of the sequence $\dot{a}^G$ is stationary in $\omega_1$.

        Returning to $V$, we have shown $\mathbb{P}\in \Gamma$ forces $\dot{a}$ to be a sequence of stationary subsets of $\omega_1$ and by hypothesis $ZFC$ proves that every poset in $\Gamma$ preserves this $\Pi_1$ property, but in the case where we wish to apply the bounded form of the axiom, we need additionally that $\dot{a}\in H_\lambda$. Observe that for each $\beta<\omega_1$, there are $\theta$ possible values of $\dot{f}(\beta)$, so the values of $\dot{f}$ are determined by $\omega_1$ many antichains of $\mathbb{P}$ each of size at most $\theta$. Since $\dot{a}$ is entirely determined by the values of $\dot{f}$, we can assume without loss of generality that only these conditions of $\mathbb{P}$ appear in its transitive closure, and by replacing $\mathbb{P}$ with an isomorphic poset if necessary, we can arrange that they are all in $H_\lambda$, so $\dot{a}$ is as well.

        Hence for arbitrary regular $\gamma>\lambda$, we can find an elementary embedding $\sigma:N\rightarrow H_\gamma$ for transitive $N$ of size $\omega_1$ with $\omega_1+1 \cup X\cup\{\theta, \lambda, \mathbb{P}, \dot{a}, \dot{f}, \mathcal{S}\}\subset rng(\sigma)$ and a $<\hspace{-3pt}\sigma^{-1}(\lambda)$-weakly $N$-generic filter $F\subset\sigma^{-1}(\mathbb{P})$ such that $\sigma^{-1}(\dot{a})^F$ is an $\omega_1$-sequence of stationary subsets of $\omega_1$, $\sigma^{-1}(\dot{f})^F$ is a bijection $\omega_1\rightarrow \sigma^{-1}(\theta)$, and $\sigma^{-1}(\dot{a})^F=\langle \{\beta<\omega_1\sbp \sigma^{-1}(\dot{f})^F"\beta\in \sigma^{-1}(S_\alpha)\}\sbp \alpha<\omega_1\rangle$. Set $g:=\sigma\circ \sigma^{-1}(\dot{f})^F$, so that $g:\omega_1\rightarrow\theta$ is injective, and let $Y:=rng(g)$.

        Then $|Y|=\omega_1$ because $dom(g)=\omega_1$, $Y\subset\theta$ because $\sigma$ preserves the order of ordinals, and $X\subseteq Y$ because $X\subseteq rng(\sigma)\cap \theta=\sigma"\sigma^{-1}(\theta)=rng(g)$. To verify that $Y$ reflects the stationarity of the $S_\alpha$, let $C\subseteq [Y]^\omega$ be a club. Then since $\{g"\beta\sbp \beta<\omega_1\}$ can easily be verified to be a club in $[Y]^\omega$ using the fact that $g:\omega_1\rightarrow Y$ is a bijection, the intersection $C_g:=\{g"\beta\in C\sbp \beta<\omega_1\}$ is a club as well. We now verify that $\{\beta<\omega_1\sbp g"\beta\in C\}$ is a club in $\omega_1$. It is unbounded because for any $\delta<\omega_1$, by the unboundedness of $C_g$ there is a $\beta<\omega_1$ with $g"\delta\subseteq g"\beta\in C$, and since $g$ is injective we must have $\delta\leq \beta$. It is closed because for any increasing sequence $\langle \beta_n\sbp n<\omega\rangle$ in it with supremum $\beta_\omega$, $\bigcup\limits_{n<\omega}g"\beta_n\in C_g$ because $C_g$ is closed, so since direct images commute with unions we have $g"\beta_\omega\in C$.

        Fix $\alpha<\omega_1$. From the above paragraph, it follows that there is a $\beta\in \sigma^{-1}(\dot{a})^F_\alpha$ such that $g"\beta\in C$. By one of the conditions on $\sigma^{-1}(\dot{a})^F$, $\sigma^{-1}(\dot{f})^F"\beta\in \sigma^{-1}(S_\alpha)$. Hence $\sigma(\sigma^{-1}(\dot{f})^F"\beta)\in S_\alpha$, and since $|\sigma^{-1}(\dot{f})^F"\beta|=\omega<crit(\sigma)$, $\sigma(\sigma^{-1}(\dot{f})^F"\beta)=\sigma"(\sigma^{-1}(\dot{f})^F"\beta)=g"\beta$. It follows that $g"\beta\in C\cap S_\alpha\cap [Y]^\omega$, so the stationarity of each $S_\alpha$ reflects to $Y$, as desired.
    \end{proof}

\section{Large Cardinals}
\label{section:large}

We conclude with some implications of $\Sigma_n$-correct forcing axioms with a large cardinal character.

\begin{prop}
    \label{prop:cma0sharp}
    $\Sigma_2\mhyphen CBMA^{<\mathfrak{c}^{++}}$ implies that $0^\sharp$ exists.
\end{prop}
\begin{proof}
    Since ccc forcing provably preserves cardinalities, by the M-A formulation of the axiom applied to the statement that $|L_{\mathfrak{c}^+}|>\mathfrak{c}$ there is a $Z\prec H_{\mathfrak{c}^{++}}$ of cardinality less than the continuum, transitive below $\mathfrak{c}$, and containing $\mathfrak{c}$ and $L_{\mathfrak{c}^+}$ such that $|\pi_Z(L_{\mathfrak{c}^+})|>\pi_Z(\mathfrak{c})$. By elementarity, $\pi_Z(L_{\mathfrak{c}^+})$ satisfies the sentence forcing it to be of the form $L_\beta$ for some $\beta$, and we must have $\pi_Z(\mathfrak{c})^+\leq \beta<\mathfrak{c}$. Thus by Lemma \ref{lemma:restrictembed}, $\pi_Z^{-1}$ restricts to an elementary embedding $L_\beta\rightarrow L_{\mathfrak{c}^+}$ with critical point $\pi_Z(\mathfrak{c})$. $0^\sharp$ follows from a well-known result of Kunen (see e.g. \cite{kanamori} Theorem 21.1).
\end{proof}

Of course, the existence of $0^\sharp$ follows from the $\Sigma_2$-correct forcing axioms for a wide range of forcing classes by the stationary reflection results of the previous section, but the proof in the case of ccc forcing (and other cardinal-preserving classes) is considerably more straightforward, and fairly striking given the low consistency strength of ($\Sigma_1$-correct) Martin's Axiom.

\begin{theorem}
\label{thm:cfavopenka}
    If $\Sigma_n\mhyphen CFA_{<\kappa}(\Gamma)$ holds (or merely $\Sigma_n\mhyphen RRP(\kappa,\lambda, \Gamma)$ for all $\lambda$), then Vopenka's principle holds for all proper classes with provably $\Gamma$-persistent $\Sigma_n$ definitions.
\end{theorem}
\begin{proof}
    Let $\phi$ define the class and let $A$ be a structure of size at least $\kappa$ in the class. Since $\phi$ is a provably $\Gamma$-persistent $\Sigma_n$ formula, applying the axiom there is some $B\in H_\kappa$ such that $\phi(B)$ and there is an elementary embedding of transitive sets $\sigma$ such that $\sigma(B)=A$. $\sigma$ then restricts to the desired elementary embedding $B\rightarrow A$.
\end{proof}

Admittedly, since many applications of Vopenka's principle involve structures with domain $V_\alpha$, which are not preserved by any nontrivial forcing, the above result is less exciting than it may at first appear. However, it can be improved to an equivalence between $\Sigma_n$-correct forcing axioms and a sort of hybridization of maximality principles and structural reflection principles of the sort Bagaria studies in \cite{BAGARIAstrucref}.\footnote{Note however that this principle is distinct from and considerably stronger than Bagaria's generic structural reflection principles, since the latter require $A\in\mathcal{C}$ to hold in $V$ and only give an elementary embedding in some forcing extension, while the structural reflection principle given here allow $A$ to enter $\mathcal{C}$ in a forcing extension but yields the embedding in $V$.}

\begin{theorem}
\label{thm:cfasr}
    The following are equivalent for regular uncountable cardinals $\kappa$, forcing classes $\Gamma$, and positive integers $n$:
    \begin{enumerate}
        \item $\Sigma_n\mhyphen CFA_{<\kappa}(\Gamma)$
        \item Whenever $\mathcal{C}$ is a class of first order structures with a provably $\Gamma$-persistent $\Sigma_n$ definition with parameters in $H_\kappa$, for every structure $A$ such that $\Vdash_{\mathbb{P}} \check{A}\in \mathcal{C}$ for some poset $\mathbb{P}\in \Gamma$, there is a $B\in \mathcal{C}\cap H_\kappa$ with an elementary embedding $B\rightarrow A$.
    \end{enumerate}
\end{theorem}
\begin{proof}
    $(1\Rightarrow 2):$ Let $\mathcal{C}$ be defined by a provably $\Gamma$-persistent $\Sigma_n$ formula $\phi(x, t)$ for some parameter $t\in H_\kappa$ (so that $\Vdash_{\mathbb{P}} \phi(\check{A}, \check{t})$). Then for any regular $\gamma>\kappa$ large enough that $A\in H_\gamma$, by the $\Sigma_n$-correct forcing axiom there is a transitive $N\in H_\kappa$ and an elementary embedding $\sigma: N\rightarrow H_\gamma$ such that $t\in N$, $\sigma(t)=t$, $A\in rng(\sigma)$, and $\phi(B, \sigma^{-1}(t))$ holds, where $B:=\sigma^{-1}(A)$. As $\sigma^{-1}(t)=t$, we have $B\in\mathcal{C}$, and by Lemma \ref{lemma:restrictembed}, $\sigma\upharpoonright B$ is the desired elementary embedding $B\rightarrow A$.

    $(2\Rightarrow 1):$ Let $\phi$, $\gamma$, $X$, $\mathbb{P}$, $\dot{a}$, and $b$ be as in the statement of $\Sigma_n\mhyphen CFA_{<\kappa}(\Gamma)$. Let $\delta:=|X|<\kappa$ and fix an enumeration $\langle x_\alpha\sbp \alpha<\delta\rangle$ of $X$. Define $\mathcal{C}$ to be the class of all $(M, \mathbb{Q}, \dot{y}, z, \langle c_\alpha\sbp \alpha<\delta\rangle)$ such that $M$ is a transitive structure, $\mathbb{Q}$ is a forcing poset, $\dot{y}$ is a $\mathbb{Q}$-name, and there is an $M$-generic filter $F\subseteq\mathbb{Q}$ such that $\phi(\dot{y}^F, z)$ holds. Then $\mathcal{C}$ has a provably $\Gamma$-persistent $\Sigma_n$ definition with parameter $\delta\in H_\kappa$ because $\phi$ is assumed to be $\Sigma_n$ and provably $\Gamma$-persistent, and $\Vdash_{\mathbb{P}} (H_\gamma^V, \mathbb{P}, \dot{a}, b, \langle x_\alpha\sbp \alpha<\delta\rangle)\in \mathcal{C}$. Applying the structural reflection principle (2), there is an elementary embedding 
    $$\sigma: (N, \sigma^{-1}(\mathbb{P}), \sigma^{-1}(\dot{a}), \sigma^{-1}(b), \langle \sigma^{-1}(x_\alpha)\sbp \alpha<\delta\rangle)\rightarrow (H_\gamma, \mathbb{P}, \dot{a}, b, \langle x_\alpha\sbp \alpha<\delta\rangle)$$
    where $N\in H_\kappa$ is transitive and there is an $N$-generic filter $F\subseteq \sigma^{-1}(\mathbb{P})$ such that $\phi(\sigma^{-1}(\dot{a})^F, \sigma^{-1}(b))$ holds in $V$. Since by construction each $x_\alpha\in X$ is in the range of $\sigma$ as well, $\Sigma_n\mhyphen CFA_{<\kappa}(\Gamma)$ holds.
\end{proof}

This result suggests that there should be generalized forcing axioms corresponding to all of Bagaria's (and Bagaria and L\"ucke's in \cite{BLpatterns} and elsewhere) structural reflection principles. That is, for any suitable forcing class there should be a $\Sigma_n$-product structural reflection forcing axiom equiconsistent with a strong cardinal for $C^{(n-1)}$, a $\Sigma_n$-weak structural reflection forcing axiom equiconsistent with a cardinal strongly unfoldable for $C^{(n-1)}$ (and presumably equivalent to $\Sigma_n\mhyphen CBFA_{<\kappa}^{<\kappa^+}$, given that strong unfoldability is equivalent to $+1$-reflection), a $\Sigma_n$-exact structural reflection forcing axiom of consistency strength somewhere below a 2-huge cardinal, and so on. However, I will not further explore that possibility here.

Finally, in line with the results of Bagaria (mentioned after Proposition \ref{prop:scCneqvopenka}) on the connection between the Vopenka scheme and strengthenings of extendibility, we can get a sort of generic $C^{(n)}$-extendibility from $\Sigma_{n+2}$-correct forcing axioms:

\begin{prop}
\label{prop:genextend}
    If $n$ is a positive integer such that $\Sigma_{n+2}\mhyphen CFA_{<\kappa}(\Gamma)$ holds for some regular uncountable cardinal $\kappa$ and $n+1$-nice forcing class $\Gamma$, then for every $\alpha\in C^{(n)}$ above $\kappa$, there is a $\Gamma$-extension $V[G]$, a $\beta\in (C^{(n)})^{V[G]}$, and an elementary embedding $j: V_\alpha^V\rightarrow V_\beta^{V[G]}$ in $V[G]$ with critical point $\kappa$ and $j(\kappa)\geq\alpha$.
\end{prop}
\begin{proof}
   Assume that no such $j$ exists and let $\phi(V_\alpha, \kappa)$ denote the assertion that for all $\mathbb{P}\in \Gamma$ and ordinals $\beta$, if $\mathbb{P}$ forces that there is an elementary embedding $j:\check{V}_\alpha\rightarrow V_\beta$ with $crit(j)=\check{\kappa}$ and $j(\check{\kappa})\not\in \check{V}_\alpha$, then $\mathbb{P}$ forces that $\beta\not\in C^{(n)}$. Then $\phi$ is $\Pi_{n+1}$ because by Corollary \ref{cor:Cndef} the assertion $\beta\not\in C^{(n)}$ is $\Sigma_n$, quantifying over all $\mathbb{P}\in \Gamma$ adds a $\Pi_{n+1}$ disjunct, and the rest of the statement adds only universal quantifiers and conjuncts or disjuncts of the same or lower complexity. Furthermore, it is provably $\Gamma$-persistent, since if any $\mathbb{Q}\in \Gamma$ forces that $\dot{\mathbb{P}}$ and $\beta$ form a counterexample to $\phi(V_\alpha, \kappa)$, then $\mathbb{Q}*\dot{\mathbb{P}}$ and $\beta$ are already a counterexample in $V$.

   Thus for sufficiently large regular $\gamma$, there is a transitive $N\in H_\kappa$ and an elementary embedding $\sigma: N\rightarrow H_\gamma$ whose range is transitive below $\kappa$ and contains $V_\alpha$ and $\kappa$ such that $\phi(\sigma^{-1}(V_\alpha), \bar{\kappa})$ holds, where as usual $\bar{\kappa}:=\sigma^{-1}(\kappa)$. However, the trivial forcing forces that $\alpha$ is an ordinal such that $\sigma$ restricts to an elementary embedding $j:\sigma^{-1}(V_\alpha)\rightarrow V_\alpha$ with $crit(j)=crit(\sigma)=\bar{\kappa}$, $j(\bar{\kappa})=\sigma(\bar{\kappa})=\kappa\not\in \sigma^{-1}(V_\alpha)$, and $\alpha\in C^{(n)}$, contradicting $\phi(\sigma^{-1}(V_\alpha), \bar{\kappa})$. Hence some $\Gamma$-extension must contain the desired elementary embedding.
\end{proof}

This is somewhat reminiscent of the Woodin-Cox characterization of forcing axioms in terms of a sort of generic supercompactness for $C^{(n-1)}$ in Section \ref{section:equivforms}. However, the above result gives us the generic embedding specifically in a $\Gamma$-extension, whereas the poset from Cox's argument used in Theorem \ref{thm:equivforms} need not be in $\Gamma$.

\appendix

\chapter{Analyses of the Formula Complexity of Common Concepts}

To simplify proofs which involve analyzing formula complexity, this appendix separates out the analyses for various assertions that come up frequently.

"$\alpha$ is a cardinal": $\Pi_1$\\
We can express this as "for all $\beta<\alpha$ and all functions $f:\beta\rightarrow\alpha$, there is some $\gamma<\alpha$ such that for all $\delta<\beta$, $f(\delta)\neq \gamma$." Since the only unbounded quantifier is the one over $f$, this is $\Pi_1$.

\vspace{12pt}

"$\lambda=\kappa^{+\alpha}$": $\Delta_2$\\
$\Pi_2$ expression: "$\lambda$ is a cardinal and for all strictly order-preserving functions $f:\alpha\rightarrow [\kappa, \lambda)$, there is some $\beta<\alpha$ such that $f(\beta)$ is not a cardinal or, for all $\gamma<\lambda$, $\gamma<\kappa$ or $\gamma$ is not a cardinal or there is some $\beta<\alpha$ such that $f(\beta)=\gamma$." Every concept used involves only bounded quantifiers except "is a cardinal", which is $\Pi_1$, "is not a cardinal", which is $\Sigma_1$, and the quantification over order-preserving functions.\\
$\Sigma_2$ expression: "$\lambda$ is a cardinal and there exists a strictly order-preserving function $f:\alpha\rightarrow [\kappa, \lambda)$ such that for all $\beta<\alpha$, $f(\beta)$ is a cardinal, and for all $\gamma<\lambda$, $\gamma<\kappa$ or $\gamma$ is not a cardinal or there is some $\beta<\alpha$ such that $f(\beta)=\gamma$."

\vspace{12pt}

"$x=V_\alpha$": $\Pi_1$\\
The rank function can be defined in a $\Delta_1$ way, since it only involves bounded quantifiers and uniquely defined entities which can be referred to with an existential or universal quantifier interchangeably. Thus "for all $y$, $y\in x$ iff the rank of $y$ is less than $\alpha$" is $\Pi_1$.

\vspace{12pt}

"$x=H_\kappa$": $\Pi_1$\\
We can express this as "for all $y$, $y\in x$ iff there is an $\alpha<\kappa$ and a function $f\in x$ with $dom(f)=\alpha$ and $rng(f)$ a transitive superset of $y$."

\vspace{12pt}

"$S\subseteq [H_\lambda]^{<\kappa}$ is stationary": $\Pi_1$ (with $S$, $\lambda$, and $\kappa$ as parameters)\\
We can express this as "for all functions $h:[H_\lambda]^{<\omega}\rightarrow H_\lambda$, there is some $Z\in S$ such that for all $x, y\in H_\kappa$, if $x\in y\in Z$, then $x\in Z$, and for all $t\in [H_\lambda]^{<\omega}$, if $t\subset Z$ then $h(t)\in Z$." $H_\kappa$ has a $\Pi_1$ definition by the above, $[H_\lambda]^{<\omega}$ has a $\Pi_1$ definition for the same reason as the full power set does, and everything else involves only bounded quantifiers.

\vspace{12pt}

"$M\models \phi(a)$" ($M$ a set model containing $a$, $\phi$ an arbitrary formula in the language of $M$): $\Delta_1$\\
This is expressible either as "there exists a function assigning truth values to the subformulas of $\phi$ which obeys the Tarski recursion relations and assigns `True' to $\phi(a)$" or "all functions assigning truth values to the subformulas of $\phi$ which obey the Tarski recursion relations assign `True' to $\phi(a)$". Since the Tarski relations only involve quantifying over $M$, these are $\Sigma_1$ and $\Pi_1$ respectively.

\vspace{12pt}

"$p\Vdash_{\mathbb{P}} \phi(\dot{a})$": same complexity as $\phi$ (minimum $\Delta_1$)\\
The definition of the forcing relation (see, e.g. Kunen's Definition IV.2.42) involves a recursion on subformulas of $\phi$ as above, where each step only involves quantifiers over $\mathbb{P}$, except for steps dealing with the quantifiers of $\phi$, in which case we need a matching quantifier over $\mathbb{P}$-names. Thus the overall number of quantifiers is the same, except we need an extra quantifier for the function handling the recursion, which as above can be either existential or universal, and thus only increases overall complexity if $\phi$ has no unbounded quantifiers of its own.

\chapter{Various Useful Lemmas}

This appendix collects and proves various basic facts I use repeatedly, for the benefit of any readers unfamiliar with any of them.

\begin{lemma}
For any uncountable cardinal $\kappa$, $H_\kappa\prec_{\Sigma_1} V$.
\label{lemma:HS1correct}
\end{lemma}
\begin{proof}
First, observe that transitive classes agree on the truth of $\Delta_0$ formulas with parameters in both of them, since the element relations are the same and all quantifiers can be taken to range over the same domains.

Now if $V\models\exists y \psi(a, y)$, where $\psi$ is $\Delta_0$ and $a\in H_\kappa$, then for any $b$ such that $\psi(a, b)$ holds and any transitive set $T$ containing $a$ and $b$, $T\models \psi(a, b)$. By the Lowenheim-Skolem theorem, there is an elementary substructure of $T$ of size less than $\kappa$ containing $b$ and all elements of $trcl(\{a\})$, which collapses to a transitive set in $H_\kappa$ containing $a$ and some $\bar{b}$ such that $\psi(a, \bar{b})$ holds. Since $\psi$ is $\Delta_0$, it follows that $H_\kappa\models \psi(a, \bar{b})$ and thus $H_\kappa\models \exists y \psi(a, y)$.
\end{proof}

\begin{lemma}
\label{lemma:clubequiv}
    If $\kappa\leq \lambda$ are uncountable cardinals with $\kappa$ regular, for any club $C\subseteq \kappa$, $\{Z\in [H_\lambda]^{<\kappa}\sbp Z\cap\kappa\in C\}$ is a club in $[H_\lambda]^{<\kappa}$, and for any club $C'\subseteq [H_\lambda]^{<\kappa}$, $\{Z\cap\kappa\sbp Z\in C'\land Z\cap\kappa\in\kappa\}$ contains a club in $\kappa$. Similar statements hold for stationary sets.
\end{lemma}
\begin{proof}
    Given any $X\in [H_\lambda]^{<\kappa}$, let $\alpha\in C$ be greater than any ordinal in $X\cap\kappa$. Then $X\cup \alpha$ is a superset of $X$ in $\{Z\in [H_\lambda]^{<\kappa}\sbp Z\cap\kappa\in C\}$, so the latter set is unbounded. To see that it is closed, note that if $\gamma<\kappa$ and $\langle X_\alpha\sbp \alpha<\gamma\rangle$ is a chain with $X_\alpha\cap\kappa\in C$ for all $\alpha<\gamma$, then 
    $$(\bigcup\limits_{\alpha<\gamma}X_\alpha)\cap\kappa=\sup\{X_\alpha\cap\kappa\sbp \alpha<\gamma\}\in C$$
    because $C$ is closed, so the union is in $\{Z\in [H_\lambda]^{<\kappa}\sbp Z\cap\kappa\in C\}$ as well.

    For $C'$, first note that $\{Z\in [H_\lambda]^{<\kappa}\sbp Z\cap\kappa\in\kappa\}$ is a club, so by intersecting it with $C'$ we can assume that $Z\cap\kappa$ is transitive for all $Z\in C'$. Now we recursively construct a sequence $\langle Z_\alpha\sbp \alpha<\kappa\rangle$ in $C'$. Let $Z_0$ be an arbitrary element of $C'$, at successor stages invoke the unboundedness of $C'$ to arrange $Z_\alpha\cup\{Z_\alpha\cap\kappa\}\subseteq Z_{\alpha+1}\in C'$, and at limit stages $\gamma<\kappa$ use the closure of $C'$ to set $Z_\gamma:=\bigcup\limits_{\alpha<\gamma} Z_\alpha$. Then the map $\alpha\mapsto Z_\alpha\cap\kappa$ is a strictly increasing function $\kappa\rightarrow\kappa$, so its range is unbounded in $\kappa$, and by the definition at limit stages it is continuous, so its range is also closed. Hence $\{Z_\alpha\cap\kappa\sbp \alpha<\kappa\}$ is a club in $\kappa$ contained in $\{Z\cap\kappa\sbp Z\in C'\land Z\cap\kappa\in\kappa\}$, as desired.

    For stationary sets, if $S\subseteq\kappa$ is stationary, then for any club $C\subseteq [H_\lambda]^{<\kappa}$, $S\cap \{Z\cap\kappa\sbp Z\in C\land Z\cap\kappa\in\kappa\}\neq\emptyset$ by the above, so every $Z$ such that $Z\cap\kappa$ is in that intersection lies in $\{Z\in [H_\lambda]^{<\kappa}\sbp Z\cap\kappa\in S\}\cap C$. Thus the latter intersection is also nonempty, so $\{Z\in [H_\lambda]^{<\kappa}\sbp Z\cap\kappa\in S\}$ is stationary. The argument in the other direction works similarly.
\end{proof}

\begin{lemma}
If $j:M\rightarrow N$ is an elementary embedding between transitive classes and $X\in M$ is transitive, then $j\upharpoonright X$ is an elementary embedding $X\rightarrow j(X)$. Furthermore, for any $Y\in M$, $j\upharpoonright X$ is an elementary embedding of the structure $(X, \in, Y\cap X)$ into $(j(X), \in, j(Y)\cap j(X))$.
\label{lemma:restrictembed}
\end{lemma}
\begin{proof}
For any formula $\phi$ in the language of set theory with an added predicate for $Y$ and any $a\in X$ such that $X\models \phi(a)$, $M\models ``(X, \in, Y\cap X)\models \phi(a)"$, so by elementarity $N\models ``(j(X), \in, j(Y)\cap j(X))\models \phi(j(a))"$.
\end{proof}

\begin{lemma}
    If $j:M\rightarrow N$ is an elementary embedding between transitive classes satsifying $ZFC^-$, $\mathbb{P}\in M$ is a forcing poset, $G\subseteq\mathbb{P}$ is an $M$-generic filter, and $H\subseteq j(\mathbb{P})$ is an $N$-generic filter with $j"G\subseteq H$, then $j$ extends to an elementary embedding $j^*:M[G]\rightarrow N[H]$. Furthermore, $rng(j^*)=\{\dot{x}^H\sbp \dot{x}\in N^{j(\mathbb{P})}\cap rng(j)\}$.
    \label{lemma:extembed}
\end{lemma}
\begin{proof}
    Let $j^*(\dot{x}^G)=j(\dot{x})^H$. To see this is a well-defined elementary embedding, let $\dot{x}\in M$ be any $\mathbb{P}$-name such that $M[G]\models\phi(\dot{x}^G)$. Then there is some $p\in G$ such that $p\Vdash \phi(\dot{x})$, so by elementarity $j(p)\Vdash \phi(j(\dot{x}))$. Since $j(p)\in H$ by hypothesis, $N[H]\models \phi(j(\dot{x})^H)$, as desired. By considering check names, we see that $j^*$ extends $j$. The "furthermore" is immediate from the definitions of $j^*$ and $M[G]$.
\end{proof}

\begin{lemma}
For any infinite cardinal $\kappa$ such that $\mathbb{P}\in H_\kappa$ is a forcing poset or Boolean algebra, if $G\subset \mathbb{P}$ is $V$-generic, $H_\kappa[G]=H_\kappa^{V[G]}$.
\label{lemma:namesize}
\end{lemma}

\begin{proof}
$H_\kappa[G]\subseteq H_\kappa^{V[G]}$ follows immediately from the observation that if a name has fewer than $\kappa$ elements, the set it names in $V[G]$ must as well. For the reverse inclusion, we proceed by induction on $\kappa$. In the limit case, if $x\in H_\kappa^{V[G]}$, then $x\in H_\lambda^{V[G]}$ for some $\lambda<\kappa$ (since $\kappa$ is still a limit cardinal in $V[G]$, as $\mathbb{P}$ is in $H_\kappa$), so inductively, $x=\dot{x}^G$ for some $\dot{x}\in H_\lambda^V$.

In the successor case, assume toward a contradiction that the desired inclusion fails and let $x$ be $\in$-minimal in $H_\kappa^{V[G]}- H_\kappa[G]$. Then $x\subseteq H_\kappa[G$]. In $V[G]$, let $\gamma$ be the cardinality of $x$ (so $\gamma<\kappa$), and let $f:\gamma\rightarrow H_\kappa^V\cap V^\mathbb{P}$ be such that $x=\{f(\xi)^G \sbp \xi<\gamma\}$. Since $\mathbb{P}$ is $\kappa$-cc, there is in $V$ a function $g:\gamma\to H_\kappa^V$ such that for all $\xi<\gamma$, $f(\xi)\in g(\xi)$ and the cardinality of $g(\xi)$ is less than $\kappa$. Since $\kappa$ is regular, $r:=\bigcup rng(g)\in H_\kappa$. Now let $x=\dot{x}^G$, where $\dot{x}$ is a $\mathbb{P}$-name but perhaps not in $H_\kappa$. We can then define in V:
$$\dot{y}:=\{(\dot{z},p)\sbp \dot{z}\in r\land p\in\mathbb{P}\land p\Vdash_{\mathbb{P}}\dot{z}\in\dot{x}\}.$$
Then $\dot{y}\in H_\kappa$ and $\dot{y}^G=x$, a contradiction.
\end{proof}

\begin{lemma}
If $\lambda<\gamma$ are regular cardinals, $Y\prec H_
\gamma$ is such that $Y\cap\lambda\in \lambda$, and $A\in Y$ has cardinality less than $\lambda$, then $A\subset Y$.
\label{lemma:macSkolem}
\end{lemma}
\begin{proof}
Let $\kappa:=|A|<\lambda$; then by elementarity $\kappa\in Y$ and there must be some bijection $f:\kappa\rightarrow A$ in $Y$. By hypothesis, each $\alpha<\kappa$ must then be in $Y$, so $A=\{f(\alpha)\sbp \alpha<\kappa\}\subset Y$.
\end{proof}

\begin{lemma}
\label{lemma:itername}
    Suppose $\kappa$ is a regular cardinal, $\langle\mathbb{P}_\alpha\sbp \alpha\leq \kappa\rangle$ is a forcing iteration such that $\mathbb{P}_\alpha\in H_\kappa$ for all $\alpha<\kappa$, $\mathbb{P}_\kappa$ is the direct limit of the preceding $\mathbb{P}_\alpha$, and $\mathbb{P}_\kappa$ satisfies the $\kappa$-cc, and $G\subset \mathbb{P}_\kappa$ is a $V$-generic filter. Then if $a\in H_\kappa^{V[G]}$, there is a $\beta<\kappa$ and a $\mathbb{P}_\beta$-name $\dot{a}\in H_\kappa^V$ such that $\dot{a}^{G_\beta}=a$.
\end{lemma}
    \begin{proof}
    Taking $\delta=|trcl(\{a\})|^{V[G]}<\kappa$, we can code $a$ by some $A\subset \delta$. Specifically, choosing some bijection $f:\delta\rightarrow \delta\times \delta$ in $V$, let $A$ be such that $(trcl(\{a\}), \in\upharpoonright trcl(\{a\}))\cong (\delta, f"A)$, with the isomorphism given by some bjiection $trcl(\{a\})\rightarrow \delta$ in $V[G]$. Let $\dot{A}$ be a nice $\mathbb{P}_\kappa$-name for $A$, i.e. 
    $$\dot{A}=\bigcup\limits_{\alpha<\delta}\{\check{\alpha}\}\times X_\alpha$$
    for some sequence $\langle X_\alpha\sbp \alpha<\delta\rangle$ of antichains of $\mathbb{P}_\kappa$. By the $\kappa$-cc and the regularity of $\kappa$, fewer than $\kappa$ conditions appear in $\dot{A}$. Since each condition of $\mathbb{P}_\kappa$ is supported in a $\mathbb{P}_\alpha$ for $\alpha<\kappa$, again applying the regularity of $\kappa$ there is some $\beta<\kappa$ such that the nontrivial coordinates of every forcing condition used in $\dot{A}$ lie in $\mathbb{P}_\beta$. Taking $\dot{A}\upharpoonright\beta$ to be the $\mathbb{P}_\beta$-name obtained by truncating the forcing conditions of $\dot{A}$ at the $\beta$ stage, $(\dot{A}\upharpoonright\beta)^{G_\beta}=A$. Since $V[G_\beta]$ contains $A$ and $f$, it can compute $a$ as the $\in$-maximal element of the transitive collapse of $(\delta, f"A)$, so there must be some $\mathbb{P}_\beta$-name for $a$. By Lemma \ref{lemma:namesize}, we can take this $\dot{a}\in H_\kappa^V$.
\end{proof}

\begin{lemma}
If $M$ is an inner model closed under $\lambda$-sequences, $\mathbb{P}\in M$ is a $\lambda^+$-cc forcing poset, and $G\subset \mathbb{P}$ is a $V$-generic filter, then $M[G]$ is closed under $\lambda$ sequences in $V[G]$.
\label{lemma:closurepreserved}
\end{lemma}
\begin{proof}
Let $f:\lambda\rightarrow M[G]$ be a sequence in $V[G]$ and let $\dot{f}$ be a name for it in $V$. We construct a new name $\dot{F}$ for $f$ as follows: for each $\alpha<\lambda$, given $p$ incompatible with all conditions we have used with $\alpha$ so far, we find a $q\leq p$ and a name $\dot{x}_{\alpha,q}\in M$ such that $q\Vdash_\mathbb{P} \dot{f}(\check{\alpha})=\dot{x}_{\alpha, q}$ (assuming without loss of generality that $1_\mathbb{P}$ forces $\dot{f}$ to be a function with domain $\lambda$). Letting $\sigma_{\alpha, q}$ be the canonical name for the ordered pair $\langle\check{\alpha}, \dot{x}_{\alpha, q}\rangle$, we put $\langle \sigma_{\alpha, q}, q\rangle$ into $\dot{F}$. We continue until for each $\alpha$, our $q$ form a maximal antichain $A_\alpha$ of $\mathbb{P}$.

By the $\lambda^+$-cc, $\dot{F}$ consists of $\lambda^2=\lambda$ ordered pairs. Since we built it from names in $M$, conditions of $p$, and ordinals, it is a subset of $M$ of size $\lambda$, so it is in $M$. For each $\alpha$, there is an unique $q\in A_\alpha\cap G$ which forces that $\dot{F}^G(\alpha)=\dot{x}_{\alpha, q}^G=\dot{f}^G(\alpha)$, so $f=\dot{F}^G\in M[G]$.
\end{proof}

\printbibliography[heading=bibintoc]

\end{document}